\documentclass[letterpaper]{article}
\usepackage{graphicx} 
\usepackage{amsmath}
\usepackage{amssymb}
\usepackage{amsthm}
\usepackage{bbm}
\usepackage{hyperref}
\usepackage{empheq}
\usepackage{enumitem}
\usepackage{color}
\usepackage{version}
\usepackage{graphicx}
\usepackage{amsfonts}
\usepackage{dsfont}
\usepackage{a4wide}
\usepackage{overpic}
\usepackage{blindtext}
\usepackage{titlesec}
\usepackage{xcolor}
\usepackage{mathabx}
\usepackage{mathtools}
\usepackage{authblk}

\title{The Doob transform and the tree behind the forest, with application to near-critical dimers.}
\author{Lucas Rey \thanks{lucas.rey@dauphine.psl.eu}}
\affil{CEREMADE, Universit{\'e} Paris-Dauphine, Université PSL, CNRS, 75016 Paris, France and DMA, {\'E}cole normale sup{\'e}rieure, Universit{\'e} PSL, CNRS, 75005 Paris, France.}

\theoremstyle{plain}
\newtheorem{Thm}{Theorem}[section]
\newtheorem{Cor}[Thm]{Corollary}
\newtheorem{Prop}[Thm]{Proposition}
\newtheorem{Lem}[Thm]{Lemma}
\theoremstyle{definition}
\newtheorem{Def}[Thm]{Definition}

\theoremstyle{remark}
\newtheorem{Rem}[Thm]{Remark}
\newtheorem{Ex}[Thm]{Example}

\newcommand{\lr}[1]{{\textcolor{black}{#1}}}

\DeclareMathOperator{\Diag}{\mathrm{Diag}}
\renewcommand{\sc}{\mathrm{sc}}
\newcommand{\dc}{\mathrm{dc}}
\newcommand{\dn}{\mathrm{dn}}
\newcommand{\TT}{\mathcal{T}}
\newcommand{\RR}{\mathcal{R}}
\newcommand{\CC}{\mathcal{C}}
\newcommand{\FF}{\mathcal{F}}
\newcommand{\Kc}{\mathcal{K}}
\newcommand{\cH}{\mathcal{H}}
\newcommand{\cI}{\mathcal{I}}
\newcommand{\R}{\mathbb{R}}
\newcommand{\C}{\mathbb{C}}
\newcommand{\N}{\mathbb{N}}
\newcommand{\Z}{\mathbb{Z}}
\newcommand{\T}{\mathbb{T}}
\newcommand{\disk}{\mathbb{D}}

\newcommand{\Proba}{\mathbb{P}}

\newcommand{\E}{\mathbb{E}}

\newcommand{\D}{\mathcal{D}}

\newcommand{\Gd}{\mathsf{\Gamma}_{\delta}}

\newcommand{\tod}{\overset{\delta \to 0}\longrightarrow}

\newcommand{\Dd}{\Delta_{\d}}
\newcommand{\oDd}{\Delta_{\d}}

\newcommand{\Rrz}{\mathcal{R}(r,z)}

\newcommand{\Rrzd}{\mathcal{R}_{\delta}(r,z)}
\newcommand{\Drz}{\mathcal{D}(r,z)}
\newcommand{\Drzd}{\mathcal{D}_{\delta}(r,z)}

\newcommand{\Bst}{B^{\mathrm{s}}}
\newcommand{\Bta}{B^{\mathrm{t}}}
\newcommand{\Bstrz}{\Bst(r,z)}
\newcommand{\Bstrzd}{\Bst_{\delta}(r,z)}
\newcommand{\Btarz}{\Bta(r,z)}
\newcommand{\Btarzd}{\Bta_{\delta}(r,z)}
\newcommand{\Birz}{B^i(r,z)}

\newcommand{\Birzd}{B^i_{\delta}(r,z)}

\newcommand{\M}{\mathcal{M}}
\newcommand{\Ed}{\vec{\mathsf{E}}}
\newcommand{\ed}{\vec{e}}
\newcommand{\fd}{\vec{f}}
\newcommand{\Gr}{\mathsf{G}^{\rho}}
\newcommand{\Vr}{\mathsf{V}^{\rho}}
\newcommand{\Er}{\mathsf{E}^{\rho}}
\newcommand{\Erd}{\Ed^{\rho}}
\newcommand{\Go}{\mathsf{G}^o}

\newcommand{\Vo}{\mathsf{V}^o}

\newcommand{\Eo}{\mathsf{E}^o}
\newcommand{\Eod}{\Ed^o}
\newcommand{\SG}{\mathsf{G}}
\newcommand{\SV}{\mathsf{V}}
\newcommand{\SE}{\mathsf{E}}
\newcommand{\SGd}{\mathsf{G}_{\delta}}
\newcommand{\hSGd}{\hat{\mathsf{G}}_{\delta}}
\newcommand{\SVd}{\mathsf{V}_{\delta}}
\newcommand{\SEd}{\mathsf{E}_{\delta}}
\newcommand{\oSG}{\mathsf{G}^{\infty}}
\newcommand{\oSV}{\mathsf{V}^{\infty}}
\newcommand{\oSE}{\mathsf{E}^{\infty}}
\newcommand{\oSGd}{\mathsf{G}_{\delta}^{\infty}}

\newcommand{\oSVd}{\mathsf{V}_{\delta}^{\infty}}
\newcommand{\oSEd}{\mathsf{E}_{\delta}^{\infty}}
\newcommand{\oc}{c^{\infty}}

\newcommand{\om}{m^{\infty}}

\newcommand{\SB}{\mathsf{B}}
\newcommand{\SW}{\mathsf{W}}
\newcommand{\SU}{\mathsf{U}}
\newcommand{\SA}{\mathsf{A}}
\newcommand{\SM}{\mathsf{M}}
\newcommand{\SH}{\mathsf{H}}
\newcommand{\Vsr}{\mathsf{V}^{\star} \setminus \{\rs\}}
\newcommand{\Gdr}{\underline{\mathsf{G}}^{\mathsf{D}}}
\newcommand{\Edr}{\underline{\mathsf{E}}^{\mathsf{D}}}
\newcommand{\Vdr}{\underline{\mathsf{V}}^{\mathsf{D}}}
\newcommand{\GD}{\mathsf{G}^{\mathsf{D}}}
\newcommand{\VD}{\mathsf{V}^{\mathsf{D}}}
\newcommand{\ED}{\mathsf{E}^{\mathsf{D}}}
\newcommand{\Dsk}{(\Delta^{\star})^{\mathrm{k}}}
\newcommand{\Ds}{\Delta^{\star}}
\newcommand{\Dsl}{\widetilde{\Delta}^{\star}}
\newcommand{\Gsl}{\widetilde{G}^{\star}}
\newcommand{\nd}{\nu_{\mathrm{dim}}}
\newcommand{\Zd}{Z_{\mathrm{dim}}}
\newcommand{\Pd}{\mathbb{P}_{\mathrm{dim}}}
\newcommand{\Pdd}{\mathbb{P}_{\mathrm{dim}}^{\mathrm{d}}}
\newcommand{\Pdk}{\mathbb{P}_{\mathrm{dim}}^{\mathrm{k}}}
\newcommand{\SWr}{\underline{\mathsf{W}}}
\newcommand{\SBr}{\underline{\mathsf{B}}}
\newcommand{\SVrs}{\underline{\mathsf{V}}^{\star}}
\newcommand{\SVs}{\mathsf{V}^{\star}}
\newcommand{\SGs}{\mathsf{G}^{\star}}
\newcommand{\SEs}{\mathsf{E}^{\star}}
\newcommand{\nRST}{\nu_{\mathrm{RST}}}
\newcommand{\nRSF}{\nu_{\mathrm{RSF}}}
\newcommand{\nMTSF}{\nu_{\mathrm{MTSF}}}
\newcommand{\ZRST}{Z_{\mathrm{RST}}}
\newcommand{\ZRSF}{Z_{\mathrm{RSF}}}
\newcommand{\ZMTSF}{Z_{\mathrm{MTSF}}}
\newcommand{\PMTSF}{\Proba_{\mathrm{MTSF}}}
\newcommand{\Dk}{\Delta^{\mathrm{k}}}
\newcommand{\Sk}{S^{\mathrm{k}}}
\newcommand{\Qk}{Q^{\mathrm{k}}}
\newcommand{\ck}{c^{\mathrm{k}}}
\newcommand{\Sr}{S^{\rho}}
\newcommand{\ms}{m^{\star}}
\newcommand{\po}{\partial^{\mathrm{out}}}

\newcommand{\oD}{\Delta}
\newcommand{\oDk}{\Delta^{\mathrm{k}}}

\newcommand{\PRST}{\Proba_{\mathrm{RST}}}
\newcommand{\PRSTl}{\widetilde{\Proba}_{\mathrm{RST}}}
\newcommand{\PRSTr}{\PRST^{\rho}}
\newcommand{\PRSF}{\Proba_{\mathrm{RSF}}}
\newcommand{\PW}{\Proba_{\mathrm{W}}}
\newcommand{\PWk}{\Proba_{\mathrm{W}}^{\mathrm{k}}}
\newcommand{\PWl}{\widetilde{\Proba}_{\mathrm{W}}}
\newcommand{\PWlk}{\widetilde{\Proba}_{\mathrm{W}}}
\newcommand{\PdRST}{\Proba^{\mathrm{d}}_{\mathrm{RST}}}
\newcommand{\Pl}{\widetilde{\Proba}}

\newcommand{\Pldim}{\Proba^{\mathrm{d}}_{\mathrm{dim}}}
\newcommand{\nud}{\nu^{\mathrm{d}}}
\newcommand{\nuk}{\nu^{\mathrm{k}}}
\newcommand{\nuko}{\nu^{\mathrm{k},1}}
\newcommand{\Kd}{K^{\mathrm{d}}}
\newcommand{\Kk}{K^{\mathrm{k}}}
\newcommand{\Kko}{K^{\mathrm{k},1}}

\newcommand{\Lko}{L^{\mathrm{k},1}}

\newcommand{\Dl}{\widetilde{\Delta}}
\newcommand{\Ql}{\widetilde{Q}}
\newcommand{\cl}{\widetilde{c}}
\newcommand{\clo}{\widetilde{c}}
\newcommand{\ml}{\widetilde{m}}
\newcommand{\Sl}{\widetilde{S}}
\newcommand{\Hl}{\widetilde{H}}
\newcommand{\Vl}{\widetilde{V}}
\newcommand{\Hlk}{\widetilde{H}}
\newcommand{\Vlk}{\widetilde{V}}
\newcommand{\tl}{\widetilde{\tau}}
\newcommand{\oSk}{S^{\mathrm{k}}}

\newcommand{\oVk}{V^{\infty, \mathrm{k}}}
\newcommand{\ot}{\overline{\tau}}
\newcommand{\Pk}{\mathbbm{P}^{\mathrm{k}}}
\newcommand{\Vk}{V^{\mathrm{k}}}
\newcommand{\Hk}{H^{\mathrm{k}}}
\newcommand{\Gk}{G^{\mathrm{k}}}
\newcommand{\MG}{\partial^{\mathrm{M}}(\mathsf{G})}
\newcommand{\co}{c}
\newcommand{\cs}{c^{\star}}
\newcommand{\cls}{\widetilde{c}^{\star}}
\newcommand{\x }{u}
\newcommand{\y}{v}
\newcommand{\z}{v}
\newcommand{\xd}{x_{\delta}}
\newcommand{\yd}{y_{\delta}}

\newcommand{\ad}{a_{\delta}}
\newcommand{\bd}{b_{\delta}}
\newcommand{\rs}{r}
\newcommand{\ls}{\lambda^{\star}}
\newcommand{\oS}{S}
\newcommand{\ocl}{\widetilde{c}^{\infty}}
\newcommand{\oSl}{\widetilde{S}}
\newcommand{\otl}{\widetilde{\overline{\tau}}}
\newcommand{\otk}{\tau^{\mathrm{k}}}
\newcommand{\oDl}{\widetilde{\Delta}}
\newcommand{\crho}{c^{\mathrm{k}}}

\newcommand{\Qr}{Q^{\rho}}
\newcommand{\tk}{\tau^{\mathrm{k}}}
\newcommand{\Xk}{X^{\mathrm{k}}}

\newcommand{\cf}{c^{\mathrm{f}}}
\newcommand{\mf}{m^{\mathrm{f}}}
\newcommand{\Slk}{\widetilde{S}}
\newcommand{\Qlk}{\widetilde{Q}}
\newcommand{\Dlk}{\widetilde{\Delta}}
\newcommand{\tlk}{\widetilde{\tau}}
\newcommand{\CPk}{P^{\mathrm{k}}}
\newcommand{\CPl}{\widetilde{P}}
\newcommand{\CPd}{P_{\mathrm{dim}}^{\mathrm{d}}}
\newcommand{\Nk}{N^{\mathrm{k}}}
\newcommand{\me}{\mathsf{e}}

\newcommand{\Dkd}{\Delta^{\mathrm{k}}_{\delta}}
\newcommand{\oDkd}{\Delta_{\delta}^{\mathrm{k}}}
\newcommand{\Dld}{\widetilde{\Delta}_{\delta}}
\newcommand{\Skd}{S^{\mathrm{k}}_{\delta}}
\newcommand{\Slkd}{S^{\mathrm{l,k}}_{\delta}}
\newcommand{\Sld}{S^{\mathrm{l}}_{\delta}}

\newcommand{\Ykd}{Y^{\mathrm{k}}_{\delta}}

\newcommand{\Sd}{S_{\delta}}
\newcommand{\Xd}{X_{\delta}}
\newcommand{\Yd}{Y_{\delta}}
\newcommand{\oSkd}{S^{\mathrm{k}}_{\delta}}
\newcommand{\tSd}{\widetilde{S}_{\delta}}
\newcommand{\Bk}{B^{\mathrm{k}}}
\newcommand{\oSd}{\overline{S}_{\d}}
\newcommand{\zed}{\zeta_{\d}}
\newcommand{\xid}{\xi_{\d}}
\newcommand{\zeld}{\zeta^{\mathrm{l}}_{\d}}
\newcommand{\sigk}{\sigma^{\mathrm{k}}}
\newcommand{\sigd}{\sigma^{\mathrm{d}}}

\renewcommand{\a}{\alpha}
\renewcommand{\d}{\delta}
\renewcommand{\t}{\theta}
\newcommand{\ba}{\bar{\alpha}}
\newcommand{\bt}{\bar{\theta}}
\newcommand{\bu}{\bar{u}}
\newcommand{\td}{\tau_{\delta}}
\newcommand{\sid}{\sigma_{\delta}}

\newcommand{\tld}{\tl_{\delta}}
\newcommand{\eps}{\varepsilon}
\newcommand{\cross}{\mathrm{Cross}_{\delta}(r,z)}
\newcommand{\crossk}{\mathrm{Cross}^{\mathrm{k}}_{\delta}(r,z)}

\newcommand{\be}{\begin{equation}}
\newcommand{\ee}{\end{equation}}

\begin{document}
\mathtoolsset{showonlyrefs}
\maketitle

\begin{abstract}
     The Doob transform technique enables the study of a killed random walk via a random walk with transition probabilities tilted by a discrete massive harmonic function. The main contribution of this paper is to transfer this powerful technique to statistical mechanics by relating two models, namely random rooted spanning forests (RSF) and random spanning trees (RST), and provide applications. More precisely, our first main theorem explicitly relates models on the level of partition functions, and probability measures, in the case of finite and infinite graphs. Then, in the planar case, we also rely on the dimer model: we introduce a killed and a drifted dimer model, extending to this general framework the models introduced in \cite{Chhita,ZinvDirac}. Using Temperley’s bijection between RST and dimers, this allows us to relate RSF to dimers and thus extend partially this bijection to RSF. As immediate applications, we give a short and transparent proof of Kenyon’s result stating that the spectral curve of RSF is a Harnack curve, and provide a general setting to relate discrete massive holomorphic and harmonic functions. The other important application consists in proving universality of the convergence of the near-critical loop-erased \lr{random walk}, RST and dimer models by extending the results of \cite{Chhita,ChelkakWan,Berestycki} from the square lattice to any isoradial graphs: we introduce a loop-erased \lr{random walk}, RST and dimer model on isoradial discretizations of any simply connected domain and prove convergence in the massive scaling limit towards continuous objects described by a massive version of SLE2.\\
    
    \end{abstract}

\tableofcontents


\section{Introduction}
A \emph{rooted spanning forest} (RSF) of a graph $\SG = (\SV,\SE)$ is a set of directed edges $F \subset \Ed$ with no cycles such that every vertex has at most one outgoing edge in $F$. The root vertices $R(F)$ are the vertices with no outgoing edges. A \emph{rooted spanning tree} (RST) with root vertex $r$ is a RSF with only one root vertex $r$. When $\SG$ is finite, given a conductance function $c: \Ed \to \R_{>0}$ and a mass function $m: \SV \to \R_{\geq 0}$ on $\SV$ \lr{which is strictly positive at one vertex at least}, the \emph{Boltzmann measure} on the set of RST, resp. RSF, is defined by
\begin{equation}
    \forall T,~\PRST^r(T) = \frac{\prod_{\ed \in \lr{T}}c_{\ed}}{\ZRST^r(\SG,c)} \quad \text{, resp.} \quad \forall F,~\PRSF(F) = \frac{\prod_{\ed \in F}c_{\ed}\prod_{r \in R(F)}m(r)}{\ZRSF(\SG,c,m)}
\end{equation}
where the normalizing constant is the \emph{partition function}
\begin{equation}
    \ZRST^r(\SG,c) = \sum_{T}\prod_{\ed \in \lr{T}}c_{\ed} \quad \text{, resp.} \quad \ZRSF(\SG,c,m) = \sum_{F}\prod_{\ed \in F}c_{\ed}\prod_{r \in R(F)}m(r).
\end{equation}
Kirchhoff's celebrated matrix-tree theorem \cite{Kirchhoff} computes the partition functions:
\begin{equation}
    \begin{aligned}
        \ZRST^r(\SG,c) = \det\big(\Delta_{|\SV \setminus \{r\}}\big) \quad \text{, resp.} \quad\ZRSF(\SG,c,m) = \det(\Dk),
    \end{aligned}
\end{equation}
where $\Delta$ is the discrete Laplacian on $\SG$ with conductances $c$ and $\Dk$ \lr{($\mathrm{k}$ stand for “killed”)} is the discrete massive Laplacian on $\SG$ with conductances $c$ and masses $m$. The RST model was later shown to exhibit many integrable features: RST can be sampled by Wilson's algorithm \cite{Wilson} using the loop-erasure of the random walk $S$ associated to $\Delta$, the transfer current theorem of Burton and Pemantle \cite{BurtonPemantle} shows that edge correlations are determinantal, \cite{Pemantle} and later \cite{USF} build Gibbs measures on infinite graphs, and more recently the branches of the tree were shown to converge to $SLE_2$ in the scaling limit in \cite{LawlerSchrammWerner}. For specific planar graphs called \emph{Temperleyan graphs}, RST were shown in \cite{Temperley,TreesMatchings} to be in bijection with \emph{dimer covers}, exhibiting a fruitful link with another well-studied model of statistical mechanics. Many properties can be transfered from one model to the other by this bijection.\par
Some of these integrable features also hold for the RSF model, replacing $\Delta$ by $\Dk$ and the random walk $S$ by the killed random walk $\Sk$, but not all of them: in particular no extension of Temperley's bijection is known to be true. One of the main goals of this paper is to fill in this gap by introducing a general method which enables to relate a RSF on a planar graph with a RST and a dimer model on the same graph.  The key underlying idea is to use the \emph{Doob transform technique} (see \cite{kemeny2012denumerable} for a general introduction, \cite{DoobTransform} for a more recent and concise introduction). Given a positive massive harmonic function $\lambda$, this technique relates the \lr{killed random walk} $\Sk$ to a \lr{random walk} $\Sl$ with conductances $\cl_{(x,y)} = \frac{\lambda(y)}{\lambda(x)}c_{(x,y)}$, and also the associated massive and non-massive Laplacian operators:
\be\label{eq:def:doob:intro}
    \forall x \sim y \in \SV,~\Proba_x\big(\Sl_1 = y\big) = \frac{\lambda(y)}{\lambda(x)}\Proba_x(\Sk_1 = y) \quad;\quad \Dl = \Diag(\lambda)^{-1} \Dk \Diag(\lambda).
\ee
The question of existence of massive harmonic functions is addressed in Section~\ref{sec:find:lambda}.

\subsection{General results}
A version of the transfer current theorem for RSF established by Chang \cite{Chang} states that 
\be\label{eq:thm:transfer:intro}
    \forall \{\ed_1, \dots ,\ed_k\} \subset \Ed,~\PRSF(\ed_1, \dots ,\ed_k) = \det\left((\Hk_{\ed_i,\ed_j})_{i,j = 1, \dots ,k}\right)\prod_{i=1}^k c_{\ed_i}
\ee
where the \emph{transfer current operator} $\Hk$ is defined in terms of the massive \emph{potential} (or \emph{Green function}): for all $\ed = (x,y), \fd = (z,w) \in \Ed$,
\be\label{eq:transfer:current:intro}
    \Hk_{\ed,\fd} = \frac{\Vk(w,y)}{\Dk(y,y)}-\frac{\Vk(x,y)}{\Dk(y,y)} \quad \text{where} \quad \forall x,y \in \SV,~\Vk(x,y) = \sum_{n=0}^{\infty}\Proba_x\big(\Sk_n = y).
\ee
Note that the transfer current theorem for RST rooted at $r$ corresponds to $m = \d_r$. When $\SG$ is infinite, $\PRST$ is the \emph{Wilson's measure} defined in Section~\ref{subsec:Wilson}. Our contribution consists in relating a RSF and RST model through the Doob transform of the massive potential (see Theorem~\ref{thm:doob} for a precise statement):
\begin{Thm}\label{thm:doob:intro}
    When $\SG$ is finite, 
    \begin{equation}
        \ZRSF(\SG,c,m) = \ZRST^o(\Go, \cl).
    \end{equation}
    In the right-hand side, RST are considered as rooted “outside” $\SG$. Moreover, for finite or infinite graphs $\SG$ with no loop, if the \lr{killed random walk} $\Sk$ dies almost surely in finite time, the measure $\PRSTl^o$ on RST of $\Go$ weighted by $\cl$ rooted “outside” satisfies the transfer current theorem with transfer current operator
    \be\label{eq:transfer:operator:intro}
        \forall \ed = (w,x), \fd = (y,z) \in \Ed,~\Hl_{\ed,\fd} = \frac{\lambda(y)}{\lambda(w)}\frac{\Vk(w,y)}{\Dk(y,y)} - \frac{\lambda(y)}{\lambda(x)}\frac{\Vk(x,y)}{\Dk(y,y)}.
    \ee
\end{Thm}
The novelty of this theorem is that it relates a RSF and a RST model, at the level of partition functions when $\SG$ is finite, but also in general at the level of edge probabilities as they are expressed in terms of the massive potential both in the RST model (our Theorem~\ref{thm:doob:intro}) and in the RSF model (Equation~\eqref{eq:transfer:current:intro}).
\par
When $\SG$ is furthermore planar, we relate RSF to dimers using the intermediate RST model. This can be seen as an extension of Temperley's bijection to RSF although there is no one-to-one mapping between configurations, as we relate the partition functions, Laplacian and Kasteleyn matrix of a RSF and dimer models. More precisely, assuming that the conductance function $c$ is symmetric, given a \emph{massive harmonic} function $\lambda: \SV \to \R_{>0}$ and an arbitrary function $\ls: \SVs \to \R_{>0}$ on vertices of the dual, we define in Figure \ref{fig:killed:dimers} two dimer models on the double graph $\Gdr$ called the \emph{killed} and \emph{drifted} dimer model (see Figure \ref{fig:double:graph} for the definition of the double graph).
\begin{figure}[!h]
    \centering   
	\begin{overpic}[abs,unit=1mm,scale=3]{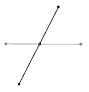}
		\put(2,21){\color{gray}$x$}
		\put(40,26){\color{gray}$y$}
		\put(30,45){$\x $}
		\put(10,0){$\y$}
		\put(23,35){$1$}
		\put(16,13){$1$}
		\put(8,27){$\textcolor{gray}{c_{xy}\frac{\lambda(y)}{\lambda(x)}}$}	
		\put(26,27){\textcolor{gray}{$c_{xy}\frac{\lambda(x)}{\lambda(y)}$}}
        \end{overpic}     
	\begin{overpic}[abs,unit=1mm,scale=3]{killed_dimers.pdf}
		\put(2,21){\color{gray}$x$}
		\put(40,26){\color{gray}$y$}
		\put(30,45){$\x $}
		\put(10,0){$\y$}
		\put(8,36){$\ls(\x)^{-1}\big(c_{xy}\lambda(x)\lambda(y)\big)^{-\frac{1}{2}}$}
		\put(5,10){$\ls(\y)^{-1}\big(c_{xy}\lambda(x)\lambda(y)\big)^{-\frac{1}{2}}$}
		\put(4,27){$\color{gray}\Big(\frac{c_{xy}\lambda(y)}{\lambda(x)}\Big)^{\frac{1}{2}}$}	
		\put(24,27){$\color{gray}\Big(\frac{c_{xy}\lambda(x)}{\lambda(y)}\Big)^{\frac{1}{2}}$}	
        \end{overpic}
        \caption{The local weights $\nud$ for the drifted dimer model (on the left) and the local weights $\nu^m$ for the killed dimer model (on the right). The primal graph is in gray, the dual graph in black.}
        \label{fig:killed:dimers}
\end{figure}
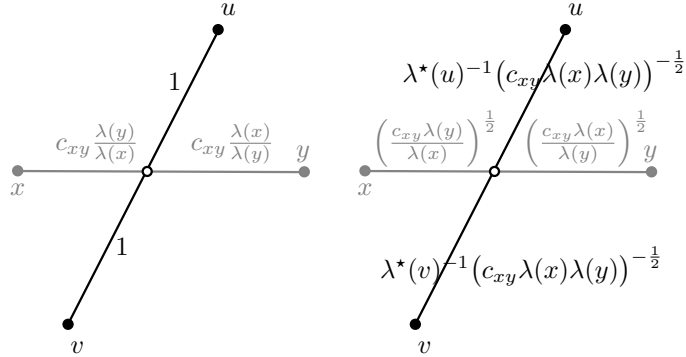
The drifted dimer model is defined such that when $\SG$ is finite, by Temperley's bijection \cite{TreesMatchings},  dimer configurations of the double graph are in weight-preserving bijection with RST of $\Go$ (that is RST of $\SG$ rooted outside) weighted by $\cl$. We state the following result, which is a summarized and simplified combination of Proposition~\ref{prop:gauge:killed:drifted}, Corollary~\ref{cor:temperley:drifted} and Proposition~\ref{prop:massive:holomorphy} (when $\SG$ is finite, the statements have to be refined to take into account boundary effects):
\begin{Thm}
     \leavevmode
    \begin{enumerate}[label=(\roman*)]
        \item When $\SG$ is finite, $\ZRSF(\SG,c,m) = \Zd(\Gdr, \nud)$, where the right-hand side is the partition function of the drifted dimer model on $\Gdr$ with weights $\nud$. 
        \item The drifted and killed dimer model are gauge equivalent: their respective Kasteleyn matrices $\Kd$ and $\Kk$ satisfy $\Kk = \Phi \Kd \Psi$ for some diagonal matrices $\Phi$, $\Psi$.
        \item There exists an operator $\Dsl$ on the dual graph such that the Kasteleyn matrix of the killed dimer model satisfies
            \be\label{eq:massive:holomorphy:intro}
	       (\Kk)^{\dag}\Kk = \begin{pmatrix} \Dk&0 \\ 0&\Dsl \end{pmatrix}.
            \ee 
    \end{enumerate}
\end{Thm}
\begin{Rem}
\leavevmode
\begin{itemize}
    \item When $c$ is symmetric, this gives an alternative proof of our Theorem~\ref{thm:doob:intro} using the local statistics formula for edge probabilities in the dimer model (see the discussion after Proposition~\ref{prop:massive:holomorphy}).
    \item (iii) implies a natural definition of \emph{discrete massive holomorphic} function as the functions $f$ in the kernel of \lr{$\Kk$}. For isoradial graphs with critical or $Z$-invariant weights, it generalizes the discrete holomorphy theory of \cite{Kenyon_2002,DiscreteHolomorphy} (non-massive case) and \cite{MakarovSmirnov,ZinvDirac,FermionicObservable} (massive case). This is discussed in Section~\ref{subsec:massive:holomorphy}.
\end{itemize}
\end{Rem}
\par
This result was previously shown in some specific cases: our drifted and killed dimer models can be seen as generalizations of the drifted and killed dimer models (on $\Z^2$, with specific weights) of \cite{Chhita} which are also central to the subsequent article \cite{Berestycki}. These models were themselves related to a drifted and killed Laplacian operator by ad hoc computations. The idea of relating a massive Laplacian with a dimer model through an intermediate RST model was already used in the special case of isoradial graphs in \cite{ZinvDirac} (see in particular Theorems 19\lr{, }21), with the purpose of relating the massive Laplacian with a $Z$-invariant massive Dirac operator. Let us also mention a related work in a forthcoming paper of Ballu, Boutillier, Mkrtchyan and Raschel of which we are aware.\par
Every time it appeared, this idea was model-specific: our approach unveils that \emph{whenever we can find “nice” massive harmonic functions on a graph, the RSF model can be related to a RST and dimer model}. It has the advantage of using the well-developed Doob transform technique, which is used to study \lr{random walk} or \lr{killed random walk} conditioned on particular events. The link with \emph{quasi-stationnary distribution} might also provide new and interesting applications.

\subsection{Applications} We provide applications of our general technique to at least two settings. 

\paragraph{$\Z^2$-periodic graphs.} On $\Z^2$-periodic graphs, a family of Gibbs measures on the infinite graph is obtained in \cite{DeterminantalForest} as the weak limit of multi-type spanning forests (MTSF) on a toroidal exhaustion. A MTSF is a spanning collection of rooted trees and cycle-rooted trees. In the limit, it becomes a RSF of the infinite graph with infinite unrooted tree components of a given slope and rooted tree components. We build “nice” $\Z^2$-periodic positive massive harmonic functions, and using the Doob transform technique, we find that the characteristic polynomial of the MTSF model is equal up to a change of variable to the characteristic polynomial of a drifted dimer model with explicit $\Z^2$-periodic weights: denoting by $\CPk$ and $\CPd$ the characteristic polynomials of the RSF and drifted dimer models, Proposition~\ref{prop:char:periodic} states that
\begin{Prop}\label{prop:char:periodic:intro}
	There exists $(z_0, w_0) \in (\R_{>0})^2$ such that
	\begin{equation}
		\forall (z,w) \in \C^2,~\CPk(z,w) = \CPd\left(\frac{z}{z_0},\frac{w}{w_0}\right).
	\end{equation}
\end{Prop}
This provides an alternative and direct proof of the fact (stated among other things in Theorem 1.4 of \cite{DeterminantalForest}) that the spectral curve of the MTSF model is a Harnack curve, see our Theorem~\ref{thm:kenyon}.

\paragraph{Near-critical dimers and trees.} Our second and more detailed application is to extend the results of \cite{Chhita} and \cite{Berestycki} to isoradial graphs: we prove universality of the convergence of the RST and drifted dimer model in the \emph{near-critical regime} or \emph{massive scaling limit}. We use the \emph{discrete exponential function} which is massive harmonic for the \emph{$Z$-invariant massive Laplacian} on isoradial graphs (these objects were introduced and studied in \cite{MassiveLaplacian}) to relate a RSF model with a killed and drifted dimer model. In the critical case (corresponding to a vanishing mass $m=0$), our model corresponds to the critical RST model on isoradial graph whose branches are known to converge to $SLE_2$ by combining \cite{LawlerSchrammWerner} and \cite{DiscreteHolomorphy}. The associated killed and drifted dimer models coincide with the critical dimer model on isoradial graphs whose height function is known to converge to the \lr{G}aussian free field \cite{DominoTiling,DominoGFF,deTiliere2007,Li}. In \cite{Berestycki}, they consider a dimer model on $\Z^2$ with specific weights in the \emph{near-critical regime} which corresponds to scaling appropriately the mass with the mesh, and show convergence of the branches of the RST model towards \emph{massive $SLE_2$} building on the proof of convergence of the loop-erased \lr{killed random walk} towards massive $SLE_2$ by Chelkak and Wan \cite{ChelkakWan}. The height function of the near-critical dimer model also converges and the limit exhibits \emph{conformal covariant properties}. \par
In this paper, we show \emph{universality} of the convergence of the near-critical RST and dimer model by extending the results of \cite{Berestycki} to any isoradial graph, answering in particular (iv) of their open questions section (1.7). Let $\SGd = (\SVd, \SE_{\d})$ be an isoradial approximation of a simply connected open set $\Omega$ with smooth boundary. The near-critical regime corresponds to $m(x) \approx M^2 \d^2$ for all $x \in \SVd$ where $M>0$ is a fixed \emph{mass parameter} and $\d$ is the \emph{mesh}. The precise scaling of the mass with the mesh is discussed in Section~\ref{subsec:near-critical:preliminaries} and Remark \ref{rem:mass:positive}. \lr{The relation with the expected conformal covariance of the limit is discussed in \ref{rem:conformal:covariance}.} To each drift parameter $\bu \in (0,2\pi)$ corresponds a discrete massive exponential function $\me$ and hence a RST and drifted dimer models by the Doob transform technique. A precise statement of the following result can be found in Theorems \ref{thm:near-critical:tree} and \ref{thm:near-critical:dimers}.
\begin{Thm}\label{thm:near-critical:intro}
    When $\d \to 0$, the random RST $\TT_{\d}^{M,\bu}$ on $\SGd$ associated with the drifted dimer model with mass parameter $M >0$ and drift parameter $\bu \in \R$ converges in the Schramm sense to a continuum limit tree $\TT^{M,\bu}$. Conditionally on the endpoint $y \in \partial \Omega$, a branch of the tree from a point $x \in \Omega$ has the law of massive radial $SLE_2$ in $\Omega$ with mass parameter $\sqrt{2}M$ from $x$ to $y$. The law of the endpoint of the branch started at $x$ is the exit law from $\Omega$ of the \lr{B}rownian motion with drift vector $\sqrt{2}Me^{i\bu}$. Moreover, the centered height function $h^{M,u}_{\d}- \E[h_{\d}^{M,u}]$ of the drifted dimer model converges in law \lr{in the sense of distributions} to a limit.
\end{Thm}
Our main contribution is to introduce the relevant RST and dimer models. In order to apply a general theorem of \cite{Berestycki} to obtain convergence of the loop-erasure of the \lr{killed random walk} towards massive $SLE_2$, we prove a \emph{uniform crossing estimate} and convergence of the \lr{killed random walk} towards the \emph{killed \lr{B}rownian motion}. This is done by applying to the near-critical regime some results of discrete harmonic analysis on isoradial graphs developed in \cite{DiscreteHolomorphy}. It is not the first time that universality of convergence in the near-critical regime is proved for a model of statistical mechanics: \cite{FermionicObservable} and \cite{UniversalityCorrelations} show universality of the convergence of correlations in the massive scaling limit for respectively the random cluster model and the Ising model on isoradial graphs.

\subsection{Perspectives}
This paper brings new perspectives in the study of RSF models by introducing a generic way to relate them to a RST and dimer model for which more tools are available. As soon as a “nice” massive harmonic function is defined on a graph, new links can be made: we applied this technique to the two settings we were aware of, but there are probably many others.
\begin{itemize}
\item For the model of \cite{DeterminantalForest} on $\Z^2$-periodic graphs, we believe that our technique can be used to show that the law of the infinite branches of the Gibbs measure is related to a drifted random walk, and that the law of these infinite branches for a RSF model coincides with the law of these branches for a “drifted” RST model with no mass, and can hence be related to the dimer model of \cite{Wangru}.
\item For the near-critical dimer model, new results might be obtained by using the comprehensive analysis of the $Z$-invariant massive Laplacian performed in \cite{MassiveLaplacian}. In particular, the explicit asymptotic of the Green function might provide new information. We mention that on the square lattice, the recent work of Mason  \cite{SineGordon} proves convergence of the two point correlation functions of the near-critical dimer model towards those of the sine-Gordon field.
\end{itemize}

\subsection{Organization of the paper}
The first sections present the general theory.
\begin{itemize}
\item In Section~\ref{sec:preliminaries}, we give preliminaries on RST and RSF models, (massive) Laplacians and (massive) random walks, and state useful results on the Doob transform that we will apply all along the article.
\item In Section~\ref{sec:doob}, we develop general theorems to transfer results from RST models to RSF models by using the Doob transform technique. We also explain why it is always possible to find massive harmonic functions on infinite graphs and introduce a notion of “killed Martin boundary”, justifying that in principle our results apply to any conductance and mass functions and any graph. 
\item In Section~\ref{sec:planar}, we focus on planar graphs: we define the killed and drifted dimer models and show how they are related to the RST and RSF model.
\end{itemize}
The last sections are devoted to applications. 
\begin{itemize}
\item Section~\ref{sec:per} shows some applications to $\Z^2$-periodic graphs, explaining in particular how to find explicit massive harmonic functions. For this section, we only need the preliminary results of Section~\ref{sec:preliminaries} on the Doob transform.
\item Section~\ref{sec:near-critical} shows applications to the near-critical dimer and RST models. We recall the setting of isoradial graphs and the results of \cite{MassiveLaplacian}, present our model and prove Theorem~\ref{thm:near-critical:intro}. For this section, we need the preliminary results of Section~\ref{sec:preliminaries}, the setting of \lr{Subsection}~\ref{par:finite} and the definition and basic properties of the drifted dimer model of Section~\ref{sec:planar}.
\item Appendix~\ref{app:critical:laplacian} contains the proof of some technical results on the $Z$-invariant massive Laplacian in the near-critical regime used in Section~\ref{sec:near-critical}. It uses results from \cite{DiscreteHolomorphy} on the critical Laplacian. 
\end{itemize}

\paragraph{Acknowledgments.} The author warmly thanks B{\'e}atrice de Tili{\`e}re for her advising all along this project, and for suggesting to extend the results of \cite{ZinvDirac} and \cite{Berestycki} which was the starting point of the project. We also thank C{\'e}dric Boutillier for explaining the argument of \lr{Subsection}~\ref{subsec:periodic}, Nathana{\"e}l Berestycki and Levi Haunschmid-Sibitz for discussions on near-critical dimers, and Marcin Lis for sharing ideas on the killed dimer model. \lr{The author also thanks two anonymous reviewers for valuable comments which improved the quality of the paper.} The author is
partially supported by the DIMERS project ANR-18-CE40-0033 funded by the French National
Research Agency.


\section{Preliminaries}\label{sec:preliminaries}
\subsection{First definitions}
Consider a connected locally finite countable graph $\SG = (\SV,\SE)$. A \emph{loop} is an edge $e =xx$ with $x \in \SV$. The undirected edge set $\SE$ induces a set of directed edges $\Ed$: to every undirected edge $e = xy \in \SE$ are associated the two directed edges $(x,y)$ and $(y,x)$. We say that two vertices of $\SG$ are \emph{neighbours} and we write $x \sim y$ (or $x \overset{\SG}{\sim} y$ when there is a risk of confusion) if $xy \in \SE$. A \emph{path} $\gamma$ is a set of directed edges $\ed_1 = (x_0,x_1), \dots , \ed_n = (x_{n-1},x_n)$ with $\ed_i \in \Ed$ for all $1 \leq i \leq n$. It is often identified with the sequence of its vertices $(x_0,x_1,x_2, \dots ,x_n)$. The length of the path is $|\gamma|=n$ and for all $0 \leq i \leq |\gamma|$, $\gamma(i) = x_i$. Such a path \emph{connects} the vertices $x_0$ and $x_n$. A \emph{cycle} is a path returning to its starting point: $x_n = x_0$.\par
Consider a positive \emph{conductance function} $c: \Ed \to \R_{> 0}$ which assigns to every directed edge a conductance $c_{(x,y)}$. We sometimes extend $c$ to all pairs of vertices by $c_{(x,y)} = 0$ if $xy \notin \SE$. Define the total conductance at a vertex $c: \SV \to \R_{> 0}$ by $c(x) = \sum_{y \sim x}c_{(x,y)}$. When $\SG$ is infinite, we always assume that the total conductances are uniformly bounded: there exists $C >0$ such that for all $x \in \SV$, $0 < c(x) \leq C$ (this is trivial when $\SG$ is finite). If for all $x \sim y \in \SG$, $c_{(x,y)} = c_{(y,x)}$, we say that $c$ is symmetric and we write $c_{xy} = c_{(x,y)}$.
\par
Finally consider a \emph{mass function} $m : \SV \to \R_{\geq 0}$ assigning a non-negative mass $m(x)$ to each vertex $x $ of $\SV$. We will write $m \neq 0$ when $m$ takes at least one positive value and $m=0$ otherwise.

\paragraph{Matrix and linear operator notations.} A matrix $L$ with rows and columns indexed by $\SV$ and coefficients $L_{x,y} = L(x,y)$ such that $L_{xy} = 0$ except when $x = y$ or $x \sim y$ is identified with the \emph{linear operator} also denoted by $L$ acting on functions $f: \SV \to \C$: 
\begin{equation}
    \forall f \in \lr{\C^{\SV}},~\forall x \in \SV,~(Lf)(x) = \sum_{y \sim x}L(x,y)f(y).
\end{equation}
For $\SU \subset \SV$, we denote by $L_{\SU}$ the restriction of the matrix $L$ to the subset $\SU$. For a function $\phi : \SV \to \C$, we denote by $D(\phi)$ (or sometimes $\Phi$ to lighten notation) the diagonal matrix with entries $\phi$ on the diagonal. 

\paragraph{Spanning trees and forests.}
A \emph{directed rooted spanning forest} (directed RSF) $F$ of $\SG$ is a subset of directed edges $F \subset \Ed$ such that each vertex has at most one outgoing edge in $F$ and there is no cycle constituted of edges of $F$. Denote by $\FF(\SG)$ the set of RSF of $\SG$. Consider $F \in \FF(\SG)$. A vertex $x \in \SV$ is a \emph{root} of $F$ if it has no outgoing edge in $F$: we denote by $R(F) \subset \SV$ the set of root vertices of $F$. Given a distinguished vertex $r \in \SV$, a \emph{directed rooted spanning tree} (directed RST) $T$ of $\SG$ rooted at $r$ is a RSF of $\SG$ satisfying $r(T) = \{r\}$: it is a forest with only one root which is $r$. Denote by $\TT^r(\SG)$ the set of directed RST of $\SG$ rooted at $r$.
\begin{Rem}\label{rem:notation:RST}
    Observe that with this definition, when $\SG$ is infinite, a RST is not necessarily connected: it can have a rooted connected component and several unrooted infinite components. Similarly, a RSF can have rooted connected components and infinite unrooted components. To align with previous notation, we should say “spanning forest” instead of RSF and “spanning forest with exactly one root” for RST.
\end{Rem}
When $\SG$ is finite, the conductance function $c$ induces a weight function on $\TT^r(\SG)$: 
\be\label{eq:def:weight:tree}
    \forall T \in \TT^r(\SG),~\nRST^{r}(T) = \prod_{\ed \in T}c_{\ed}.
\ee
The \emph{partition function} is the weighted sum of RST rooted at $r$:
\be
    \ZRST^{r}(\SG, c) = \sum_{T \in \TT^r(\SG)}\nRST^{r}(T).
\ee
Since $\SG$ is finite and connected, $0 < \ZRST^{r}(\SG, c) < \infty$ and the weight function induces a probability measure on $\TT^r(\SG)$: 
\be\label{def:tree}
    \forall T \in \TT^r(\SG),~\PRST^r(T) = \frac{\nRST^{r}(T)}{\ZRST^{r}(\SG, c)}.
\ee
When $\SG$ is finite, the conductance and mass functions $(c,m)$ also induce a weight function on $\FF(\SG)$: 
\be\label{eq:def:weight:forest}
    \forall F \in \FF(\SG),~\nRSF(F) = \prod_{\ed \in F}c_{\ed}\prod_{x \in R(F)}m(x).
\ee
The partition function is the weighted sum
\begin{equation}
    \ZRSF(\SG,c,m) = \sum_{F \in \FF(\SG)}\nRSF(F).
\end{equation}
When $m \neq 0$, $0<\ZRSF(\SG,c,m)<\infty$, the conductance and mass functions induce a probability measure on $\FF(\SG)$: 
\be\label{def:forest} 
    \forall F \in \FF(\SG),~\PRSF(F) = \frac{\nRSF(F)}{\ZRSF(\SG,c,m)}.
\ee
This measure is supported on RSF with root vertices in the support of the mass function i.e. $R(F) \subset \{x \in \SV: m(x) >0\}$. The probability measures $\PRST^r$ and $\PRSF$ coincide when $m = \delta_r$. For a subset of directed edges  $\{\ed_1, \dots ,\ed_n \} \subset \Ed$ we denote by $\PRST^r(\ed_1, \dots ,\ed_n)$, resp. $\PRSF(\ed_1, \dots ,\ed_n)$, the measure of the set of all trees, resp. forests, containing these edges. 
\begin{Rem}
There is no obvious extension of these probability laws to the case of infinite graphs. This issue is common to many models of statistical physics and we will address it in Section~\ref{subsec:Wilson}.
\end{Rem}

\paragraph{The forest-tree bijection.}
Adding a cemetery vertex to $\SG$ allows to connect RSF of $\SG$ and RST of a modified graph $\Gr$ as we now explain. Define a new graph $\Gr = (\Vr, \Er)$ by adding a vertex called the cemetery state: $\Vr = \SV \cup \{\rho\}$, and $\Er$ is the set $\SE$ plus an edge $x\rho$ for each $x$ such that $m(x) >0$. When $m = 0$, we define $\Gr = \SG$, so the graph $\Gr$ is always connected. Define a new positive conductance function $\crho: \Erd \to \R_{> 0}$ by 
\be\label{eq:conductance:Gr}
    \crho_{(x,y)}=
    \left\{
    \begin{array}{ll}
          c_{(x,y)} &\text{ if }(x,y)\in \Ed\\
           m(x) &\text{ if } m(x) >0,~y = \rho\\
           m(y) &\text{ if } m(y) >0,~x = \rho\\
           0 &\text{ otherwise.}
    \end{array}
    \right.
\ee
An example of the graph $\Gr$ with conductance function $\crho$ is drawn on Figure \ref{fig:Gr}.
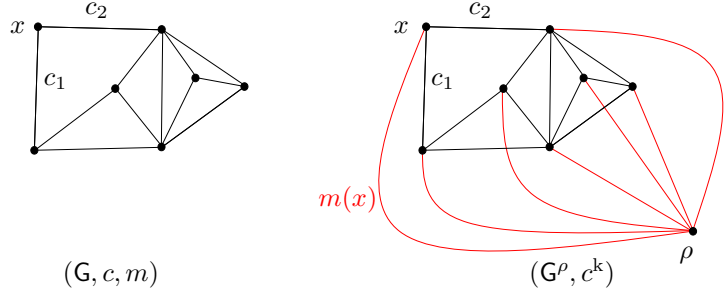
\begin{figure}
    \centering   
        \begin{overpic}[abs,unit=1mm,scale=1]{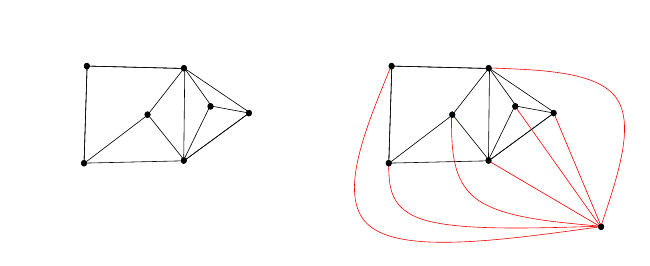}
        \put(100,3){$\rho$}
        \put(18,0){$(\SG,c,m)$}
        \put(80,0){$(\Gr,\crho)$}
        \put(11,33){$x$}
        \put(15.5,26){$c_1$}
        \put(21,35.5){$c_2$}
        \put(62,33){$x$}
        \put(67,26){$c_1$}
        \put(72,35.5){$c_2$}
        \put(52,10){\textcolor{red}{$m(x)$}}
        \end{overpic}
        \caption{On the left: a graph $\SG$ with conductance and mass functions $(c,m)$. On the right: the associated graph $\Gr$ with conductance function $\crho$, assuming that $m(y) >0~\forall y \in \SV$. Edges are labelled by their conductances.}
        \label{fig:Gr}
\end{figure}
Observe that when $c$ is symmetric, $\crho$ also is. The total conductance at a point $x \in \SV$ is $\crho(x) = c(x)+m(x)$.
\begin{Def}\label{def:forest-tree:bijection}
    RSF of $\SG$ with roots in $\{x: m(x)>0\}$ are in bijection with RST of $\Gr$ rooted at $\rho$: to $F \in \FF(\SG)$ we associate $T \in \TT^{\rho}(\SG)$ obtained by adding an edge $(x,\rho)$ for each root vertex $x \in R(F)$. When $\SG$ is finite, this is a weight-preserving bijection called the \emph{forest-tree bijection}:
\be\label{eq:bijection:tree/forest}
    \nRSF(F) = \prod_{\ed \in F}c_{\ed} \prod_{x \in R(F)}m(x) = \prod_{\ed \in F}\crho_{\ed} \prod_{x \in R(F)}\crho_{(x,\rho)} = \prod_{\ed \in T}\crho_{\ed} = \nRST^{\rho}(T).
\ee
In particular, $\ZRSF(\SG,c,m) = \ZRST^{\rho}(\Gr, \crho)$ and when $m \neq 0$, $\PRSF(F) = \PRSTr(T)$.
\end{Def}
\begin{Rem}\label{rem:EST}
    \lr{When $m = 0$ and $\SG$ is infinite, $\rho$ does not exist and $\Gr = \SG$. It is natural to interpret a RST of $\Gr$ rooted at $\rho$ as a RSF of $\SG$ with no root, hence the bijection of Definition \ref{def:forest-tree:bijection} is simply the identity.} This is called an \emph{essential spanning forest} (ESF) of $\SG$, and it is a collection of infinite connected components spanning $\SV$, with no cycle.
\end{Rem}

\paragraph{\lr{Laplacians.}}
A \emph{(discrete) Laplacian operator} $\Delta: \C^{\SV} \to \C^{\SV}$ is naturally associated with the conductance function $c$:
\begin{equation}
    \forall f: \SV \to \C,~\forall x \in \SV,~\Delta f(x) =
    \sum_{y \sim x}c_{(x,y)}(f(x)-f(y)).
\end{equation}
A function $f: \SV \to \C$ such that $\Delta f = 0$ is called \emph{harmonic}. A \emph{(discrete) massive Laplacian operator} $\Dk: \C^{\SV} \to \C^{\SV}$ is naturally associated to the conductance and mass functions $(c,m)$:
\begin{equation}
    \forall f: \SV \to \C,~\forall x \in \SV,~\Dk f(x) =
    m(x)f(x) + \sum_{y \sim x}c_{(x,y)}(f(x)-f(y)).
\end{equation}
A function $f: \SV \to \C$ such that $\Dk f = 0$ is called \emph{massive harmonic}. In the canonical base of $\C^{\SV}$, these Laplacians are represented by square matrices whose rows and columns are indexed by $\SV$:
\begin{equation}
        \Delta_{x,y} = 
        \left\{
        \begin{array}{ll}
             -c_{(x,y)}& \text{ if }x \sim y,~x \neq y\\
             c(x) - c_{(x,x)}& \text{ if } x=y\\
             0& \text{ otherwise}
        \end{array}
        \right.
        ,~
        \Dk_{x,y} = 
        \left\{
        \begin{array}{ll}
             -c_{(x,y)}& \text{ if }x \sim y,~x \neq y\\
             m(x) + c(x) - c_{(x,x)}& \text{ if } x=y\\
             0& \text{otherwise}
        \end{array}
        \right.
        .
\end{equation}
Observe that the massive Laplacian $\Dk$ on $\SG$ coincides with the restriction to $\SV$ of the (non-massive) Laplacian on the graph $\Gr$ with conductance function $\crho$.

\paragraph{Matrix-tree theorem.}
In this \lr{subsection}, we assume that $\SG$ is finite. The objects introduced \lr{above} (the RST/RSF and the Laplacians) are naturally related. Let $r \in \SV$ be fixed. The matrix-tree theorem of \cite{Kirchhoff} (and its extension to directed weighted graphs, which can be found for example in \cite{Chaiken}) relates the determinant of the (non-massive) Laplacian with the weighted sum of RST rooted at $r$: 
\begin{equation}
    \ZRST^{r} = \det \Delta_{\SV \setminus \{r\}}.
\end{equation}
Applying this theorem to $(\Gr,\crho)$ with distinguished vertex $\rho$ and using the forest-tree bijection gives
\be\label{eq:matrix-forest}
    \ZRSF(\SG,c,m) = \ZRST^{\rho}(\Gr, \crho) = \det \Dk
\ee
since, as we already observed, the restriction to $\SV$ of the Laplacian on $\Vr$ with conductance function $\crho$ is $\Dk$. This identity is sometimes called the \emph{matrix-forest theorem}.

\paragraph{\lr{Green function.}}
\lr{
When $m \neq 0$, since $\SG$ is connected there exists at least one RSF of $\SG$ with strictly positive weight, which implies that $\ZRSF(\SG,c,m) > 0$. Hence by the matrix-forest theorem, $\det \Dk > 0$. The inverse of the massive Laplacian is the \emph{massive Green function} $\Gk: \SV \times \SV \to \R$. In matrix notation, it satisfies
\be\label{eq:green:inverse:laplacian}
    \Gk = (\Dk)^{-1}.
\ee
In other words
\begin{equation}
    \forall x, y \in \SV,~\mathbbm{1}_{\{x=y\}} = (\Dk \Gk)_{x,y} = m(x)\Gk(x,y) +  \sum_{z \sim x}c_{(x,z)}(\Gk(z,y)-\Gk(x,y)). 
\end{equation}
\begin{Rem}
    This definition of the Green function can be extended if $\SG$ is infinite under additional conditions. For example if the masses are uniformly bounded away from $0$ and $\infty$, we can still make sense of $\Gk = (\Dk)^{-1}$ by inverting the restriction of the operator $\Dk$ to $L^2(\SV)$, see for example Section 4.1 of \cite{MassiveLaplacian}.
\end{Rem}
}

\paragraph{Random walks and potentials.}
To the conductance function $c$ is associated a discrete time Markov chain $S$ on $\SV$ called the \emph{random walk} with conductances $c$. Given $x_0 \in \SV$, the \lr{random walk} with initial state $x_0$ is defined by $S_0 = x_0$ and transition probabilities
\begin{equation}
    \forall i \in \Z_{\geq 0},~\forall x,y \in \SV,~
    Q(x,y) = \Proba_{x_0}(S_{i+1} = y|S_i=x) = 
        \left\{ 
        \begin{array}{ll}
            c_{(x,y)}/c(x) &\text{if }y \sim x  \\
            0 &\text{otherwise.}
        \end{array}
        \right.
\end{equation}
Since $\SG$ is connected and $c$ is positive, this Markov chain is irreducible. Recall that $D(c)$ denotes the diagonal matrix with entries $c(x)$. The transition kernel is related to the Laplacian:
\be\label{eq:def:transition:kernel}
    Q = I - D(c)^{-1}\Delta.
\ee
For $\SU \subset \SV$, we denote by $\tau(\SU)$, resp. $\tau^+(\SU)$, the first time $i \geq 0$, resp. $i >0$, at which the \lr{random walk} enters the set $\SU$:
\begin{equation}
    \begin{aligned}
    \tau(\SU) &= \inf\{i \geq 0, S_i \in \SU\}, & \tau^+(\SU) &= \inf\{i > 0, S_i \in \SU\}.
    \end{aligned}
\end{equation}
If $\SU  = \{x\}$ is reduced to a point, we compactify the notation $\tau(\SU) =: \tau(x)$ and $\tau^+(\SU) =: \tau^+(x)$.\par
The conductance function $\crho$ associated with $c$ and $m$ induces a Markov chain $\Sr$ on $\Vr$. The \lr{random walk} $\Sr$ stopped at $\rho$ is denoted by $\Sk$ and is called the \emph{killed random walk}: 
\begin{equation}
    \forall i \in \Z_{\geq 0},~\Sk_i := \Sr_{i \wedge \tau(\rho)}
\end{equation}
where we use the standard notation $\wedge$ for the minimum. The cemetery vertex $\rho$ is an absorbing state for the Markov chain $\Sk$. The sub-Markovian transition kernel $\Qk$ associated with $\Sk$ is a matrix with rows and columns indexed by $\SV$ related to the massive Laplacian:
\be\label{eq:def:massive:transition:kernel}
    \Qk = I - D(\ck)^{-1}\Dk.
\ee
We denote by $\tk$, resp. $(\tk)^+$, the hitting times of sets for this Markov chain. When $m = 0$, $\Sk = \Sr = S$ and $\tk(\rho) = (\tk)^+(\rho) = \infty$ by convention, so from now on we only state the definitions and results for the \lr{killed random walk} as they include the \lr{random walk} case $m=0$.
\begin{Rem}\label{rem:coupling:RandomWalk/KilledRandomWalk}
    The \lr{random walk} $S$ associated with $c$ and the \lr{killed random walk} $\Sk$ associated with $(c,m)$ can be coupled. Let $(U_i)_{i \geq 0}$ be a sequence of independent uniform variables in $[0,1]$, independent of $S$. Define $\Xk$ recursively by $\Xk_0 = S_0$ and 
    \begin{equation}
        \forall i \in \Z_{\geq 0},~\Xk_{i+1}=
        \left\{
        \begin{array}{ll}
             \rho & \text{ if }(\Xk_i = \rho) \text{ or } \left(\Xk_i = S_i = x \in \SV \text{ and }U_i > c(x)/(c(x)+m(x))\right)\\
             S_{i+1} & \text{ if } \Xk_i = S_i = x\in \SV \text{ and }U_i \leq c(x)/(c(x)+m(x)).
        \end{array}
        \right.
    \end{equation}
    Then, $\Xk$ has the law of the \lr{killed random walk} associated with the conductance and mass functions $(c,m)$ on $\SG$ so we can denote it $\Xk = \Sk$, and
    \be\label{eq:coupling:KilledRandomWalk-RandomWalk}
        \forall i < \tk(\rho),~\Sk_i = S_i.
    \ee
\end{Rem}
The \lr{killed random walk} $\Sk$ is \emph{transient} if for all $x \in \SV$,
    \begin{equation}
	\Proba_x\big((\tk)^+(x) < \infty\big) < 1,
    \end{equation}
\lr{implying} that with probability $1$, the walk reaches the cemetery state in finite time or “escapes to infinity”. When $m =0$, this is the usual definition of a transient \lr{random walk}.
\begin{Rem}\label{rem:always:transient}
    When $\SG$ is finite and $m = 0$, the \lr{random walk} is never transient. When $m \neq 0$, the \lr{killed random walk} is always transient: if $\Sk$ is coupled to $S$ as in Remark \ref{rem:coupling:RandomWalk/KilledRandomWalk}, either $S$ is transient, in which case $\Sk$ also is, or it is recurrent. In this case, given $x \in \SV$ such that $m(x) >0$, $S$ returns almost surely infinitely many times to $x$ and each time $\Sk$ has the same positive probability of dying, so $\Sk$ dies almost surely and hence it is also transient.
\end{Rem}
The \emph{potential} of a transient \lr{killed random walk} or transient \lr{random walk} is well-defined:
\be\label{eq:def:potential}
    \forall x,y \in \SV,~\Vk(x,y) = \E_x\left[\sum_{n=0}^{\tk(\rho)-1}\mathbbm{1}_{\{\Sk_n=y\}}\right] = \E_x\left[\sum_{n=0}^{\infty}\mathbbm{1}_{\{\Sk_n=y\}}\right] = \sum_{n = 0}^{\infty}\Proba_x(\Sk_n = y) < \infty.
\ee
This definition coincides with the potential of a transient \lr{random walk} when $m = 0$ (recall that in this case, $\tk(\rho) = \infty$ by convention). By the simple Markov property, the potential satisfies,
\begin{equation}
    \forall x,y \in \SV, ~\Vk(x,y) = \mathbbm{1}_{\{x=y\}} + \sum_{z \sim x}\frac{c_{(x,z)}}{\ck(x)}\Vk(z,y).
\end{equation}
In other words, in matrix notation, 
\be\label{eq:harmonicity:potential}
    \Dk \Vk = D(\ck).
\ee
Using \eqref{eq:def:massive:transition:kernel}, this implies
$$
    \Vk = (I-\Qk)^{-1}
$$
which could also be used as a definition of $\Vk$. If $\SG$ is finite and $m \neq 0$, taking inverses implies that the potential is related to the massive Green function: $\Vk = \Gk D(\ck)$, or in other words
\be\label{eq:green=potential}
    \forall x,y \in \SV,~\Vk(x,y) = \ck(y)\Gk(x,y).
\ee
If $\SG$ is infinite, $m = 0$ and the \lr{random walk} $S$ is transient, the potential $V = \Vk$ is well-defined but the Green function is not since the Laplacian is not a priori invertible so Equation~\eqref{eq:green=potential} does not make sense.\par
We will soon need to use results of \cite{Chang}, who has a different definition of the potential. Let $(\Xk_t)_{t \in \R_+}$ be the continuous time pure jump Markovian process on $\SV \cup \{\rho\}$ with sub-Markovian generator $\Dk$ (see Chapter 2.2 of \cite{Chang} for more details). The explosion time of this process is infinite since we always assume that the total conductances are bounded. The embedded Markov chain is the discrete time \lr{killed random walk}: if $N_0 =0, N_1, \dots$ are the times of the jumps, $\Xk_{N_i} = \Sk_i$ for all $i \in \N$. The associated \emph{continuous potential} is 
\begin{equation}
    \forall x,y \in \SV,~\Vk_c(x,y) = \E_x\bigg[\int_0^{\infty}\mathbbm{1}_{\{\Xk_t = y\}}\bigg].
\end{equation}
It is well defined by transience of the embedded Markov chain and the fact that the conductances (hence the jump rates) are bounded. Chang uses the notation $V^x_y = \Vk_c(x,y)$. Since the continuous time process $\Xk$ has a jump rate $\ck(y)$ at $y$, it spends an average time $\ck(y)^{-1}$ in $y$ every time it hits $y$, and since the embedded Markov chain is the \lr{killed random walk}:
\begin{equation}
    \Vk_c(x,y) = \frac{\Vk(x,y)}{\ck(y)},
\end{equation}
hence when $\SG$ is finite and $m \neq 0$, by Equation~\eqref{eq:green=potential} the continuous potential coincides with the Green function: $\Vk_c(x,y) = \Gk(x,y)$ for all $x,y \in \SV$.\par
The \emph{transfer current operator} is defined for all $\ed = (w,x), \fd = (y,z) \in \Erd$ by
\be\label{eq:def:transfer}
	\Hk_{\ed,\fd} = \Vk_c(w,y) - \Vk_c(x,y) = \frac{\Vk(w,y)}{\ck(y)}-\frac{\Vk(x,y)}{\ck(y)},
\ee
with the convention that for all $x \in \SV$, $\Vk(x,\rho) = \Vk(\rho,x) = 0$. When $m=0$, the transfer current operator is defined on $\Ed$ so the conventions are not needed and it coincides with the usual transfer current operator associated with a transient \lr{random walk}. Observe that this transfer current operator is not symmetric and if $y = \rho$ or if $\ed$ is a loop, $\Hk_{\ed,\fd} = 0$.

\paragraph{Induced subgraph and wired boundary conditions.}
Consider an infinite graph $\oSG = (\oSV,\oSE)$. Let $\SV \subset \oSV$ be a strict subset of $\oSV$. The subgraph of $\oSG$ \emph{induced} by $\SV$ is the graph $\SG = (\SV,\SE)$ with edge set $\SE$ such that for all $x, y \in \SV$, $xy \in \SE$ if and only if $xy \in \oSE$. When $\oSG$ is planar, the induced subgraph $\SG$ is \emph{simply connected} if all the vertices of $\oSV \setminus \SV$ belong to the outer face of $\SG$. The (inner) \emph{boundary} $\partial \SV$ of $\SG$ in $\oSG$ is
\begin{equation}\label{eq:def:partial}
    \partial \SV = \{x \in \SV~|~\exists y \in \oSV \setminus \SV,~x \overset{\oSG}{\sim} y\}.
\end{equation}
Given a conductance and mass functions $(\oc, \om)$ on $\oSG$, a natural conductance and mass functions with \emph{wired boundary conditions} are defined on the induced subgraph. It  corresponds to assigning conductances by restriction $c :=  \oc_{|\Ed}$ and adding a mass for each positive conductance between a vertex in $\SV$ and a vertex in $\oSV \setminus \SV$: 
\begin{equation}
    \forall x \in \SV,~m(x) := \om(x) + \sum_{y \overset{\oSG}{\sim} x,~y \in \oSV \setminus \SV}\oc_{(x,y)}.
\end{equation}
\lr{
\begin{Rem}\label{rem:free:bc}
    Note that we slightly abuse notation here: we should write $(c^{\mathrm{w}},m^{\mathrm{w}})$ instead of $(c,m)$ to emphasize that we take \emph{wired} boundary conditions. Since this is the only type of boundary conditions that we consider in this article, this will hopefully create no confusion. The \emph{free boundary conditions} correspond to assigning conductances and masses by restriction $\cf = \oc_{|\Ed}$ and $\mf = \om_{|\SV}$. We only mention it and do not use it in this article. 
\end{Rem} 
}
Note that \lr{$m$} differs from $\om$ only on $\partial \SV$. Denote by $\Dk$ and $\Sk$ the massive Laplacian and \lr{killed random walk} on $\oSG$ associated with the conductance and mass functions $(\oc, \om)$. The massive Laplacian $\SG$ associated to $(c, m)$ is simply the restriction $\oDk_{\SV}$: for all $x, y \in \SV$ such that $x \neq y$ and $x \sim y$, $\oDk_{x,y} = -\oc_{(x,y)} = -c_{(x,y)}$ and 
\be\label{eq:restriction:laplacian:wired}
    \forall x \in \SV,~\oDk_{x,x} = \om(x) + \sum_{y \overset{\oSG}{\sim} x} \oc_{(x,y)} =  m(x) + \sum_{y \overset{\SG}{\sim} x} c_{(x,y)}.
\ee
Equivalently, the \lr{killed random walk} on $\SG$ with conductance and mass functions $(c, m)$ is the \lr{killed random walk} $\oSk$ on $\oSG$ with conductance and mass functions $(\oc,\om)$ killed when it exits $\SV$. More formally, if we define $\Xk$ by $\Xk_i = \oSk_i$ for all $i < \otk\big(\oSV \setminus \SV\big)$ and $\Xk_i = \rho$ for all $i \geq \otk\big(\oSV \setminus \SV\big)$, then $\Xk$ has the law of the \lr{killed random walk} on $\SG$ associated with the conductance and mass functions $(c, m)$.\par

\begin{figure}
    \centering   
        \begin{overpic}[abs,unit=1mm,scale=1]{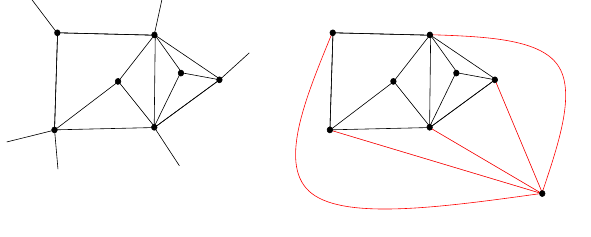}
        \put(92,3){$o$}
        \put(17,0){$(\SG,c)$}
        \put(70,0){$\Go$}
        \put(6,33){$x$}
        \put(11,26){$c_1$}
        \put(16,36){$c_2$}
        \put(1,38){$\oc_3$}
        \put(52,33){$x$}
        \put(57,26){$c_1$}
        \put(62,36){$c_2$}
        \put(46,10){\textcolor{red}{$\oc_3$}}
        \end{overpic}
        \caption{On the left: a simply connected induced subgraph $\SG$ of a planar graph $\oSG$. Edges leaving $\SV$ are marked with no target vertex. On the right: the graph $\Go$ is also planar. Edges are labelled by their conductances.}
        \label{fig:Go}
\end{figure}
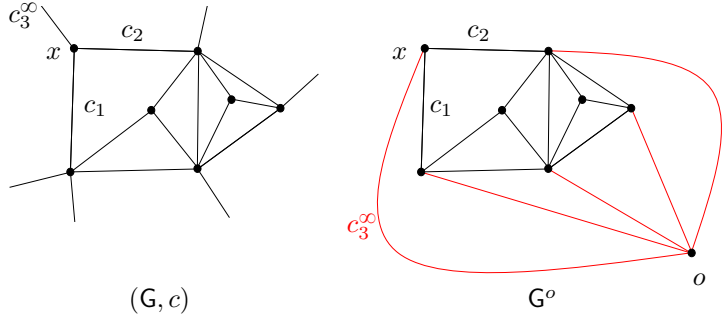
Assume now that $\om =0$. The wired conductance and mass functions are $c := \oc_{|\SV}: \Ed \to \R_{> 0}$ and
    \be\label{eq:mass:wired}
        \forall x \in \SV,~m(x) := \sum_{y \overset{\oSG}{\sim} x, y \in \oSV \setminus \SV}\oc_{(x,y)}.
    \ee
In particular when $\SV$ is a strict subset of $\oSV$, $m \neq 0$ even though $\om = 0$. The \lr{killed random walk} on $\SG$ associated with the conductance and mass functions $(c, m)$ is simply the \lr{random walk} $\oS$ on $\oSG$ with conductance function $\oc$ killed when it exits $\SV$, and the massive Laplacian on $\SV$ associated with the conductance and mass functions $(c, m)$ is simply the restriction $\oD_{\SV}$. In this case, the forest-tree bijection is particularly easy to understand: denote by $\Go$ the graph obtained from $\oSG$ by identifying all vertices of $\oSV \setminus \SV$ to a single vertex $o$. When $\oSG$ is planar and $\SG$ is simply connected, the graph $\Go$ is still planar: this will be useful later on. The forest-tree bijection states that RSF of $\SG$ weighted by $(c,m)$ are in weight-preserving bijection with RST of $\Go$ rooted at $o$ weighted by $c$, with the natural abuse of notation that the conductance of an edge $(x,o)$ is the sum of conductances $\oc_{(x,y)}$ of edges $(x,y)$ with $y \in \oSV \setminus \SV$. An example of the graph $\Go$ is drawn on Figure \ref{fig:Go}. These RST of $\Go$ rooted at $o$ can be thought of as RST of $\SG$ rooted at the outer boundary.

\subsection{Wilson's algorithm and the transfer current theorem}\label{subsec:Wilson}
The transfer current theorem (first proved by Burton and Pemantle \cite{BurtonPemantle}) expresses the probability that some fixed edges belong to a random uniform (undirected) spanning tree of a finite graph as determinants of submatrices of the transfer current operator. It shows in particular that the uniform spanning tree seen as a random point process on the set of edges is determinantal. This result was extended in many directions, but one which is particularly interesting to us is to the case of directed RSF proved by Chang in his PhD thesis \cite{Chang}.\par
Recall that by the forest-tree bijection, the probability measure $\PRSTr$ on $\TT^{\rho}(\Gr)$ associated with the conductance function $\crho$ by Equation~\eqref{def:tree} coincides with the probability measure $\PRSF$ on RSF of $\SG$ weighted by $(c,m)$. Theorem 5.2.3 of \cite{Chang} reads
\begin{Thm}[Finite Directed Transfer Current Theorem]\label{thm:transfer:finite}
     When $\SG$ is finite and $m \neq 0$, 
    \be\label{TransferCurrentFinite}
	\forall \{\ed_1, \dots ,\ed_k\} \subset \Erd,~       \PRSTr(\ed_1,  \dots ,\ed_k) = \det\left((\Hk_{\ed_i,\ed_j})_{i,j = 1, \dots ,k}\right)\prod_{i=1}^k \crho_{\ed_i}.
	\ee 
\end{Thm}
In \cite{Chang} this result is stated for processes with finite state space that are absorbed by the cemetery state almost surely in finite time: it corresponds to $\SG$ being finite and $m \neq 0$. The historical proof of \cite{BurtonPemantle} in the non-directed case relies on electrical networks. Other proofs, including that of \cite{Chang} which deals with the general directed case, rely on Wilson's algorithm. We do not repeat the proof but we introduce a slightly more general version of Wilson's algorithm in the setting of transient \lr{killed random walk} on countable graphs. This algorithm is a combination of the algorithm of \cite{Chang} for a directed \lr{killed random walk} that dies almost surely and Wilson's method rooted at infinity described in Section 10.1 of \cite{LyonsPeres} (for reversible Markov chains with no killing). It actually unifies the two algorithms by identifying the cemetery state with infinity.\par
We first need to define the concept of loop-erasure. A path $\gamma$ in $\SG$, finite or infinite, is said to be \emph{transient} if it visits no vertex infinitely many times. If $\gamma = (x_0,x_1, \dots )$ is a transient path, its \emph{loop-erasure} denoted by $LE(\gamma)$ is also a path in $\SG$ built in the following way: start with $\gamma_1 = (x_0,x_1)$. Define recursively $\gamma_{n+1}$ out of $\gamma_n$ in the following way: let $l_n = |\gamma_n|$. If for all $0 \leq i \leq l_n$, $\gamma_n(i) \neq x_{n+1}$, then $\gamma_{n+1}:= (\gamma_n(0), \dots ,\gamma_n(l_n),x_{n+1})$. Otherwise, let $k = \inf \{0 \leq i \leq l_n, x_{n+1} = \gamma_n(i)\}$ and define $\gamma_{n+1} = (\gamma_n(0), \dots ,\gamma_n(k-1), x_{n+1})$. This produces a new path in $\SG$ with no cycle. If $\gamma$ is finite this algorithm stops after $|\gamma|$ steps. If $\gamma$ is infinite, a countable number of steps must be done. Nevertheless, the loop-erasure is well defined when the path is transient (see the proof of Proposition 10.1 of \cite{LyonsPeres} for details).

\begin{Def}[Wilson's algorithm rooted at the cemetery state]\label{def:Wilson}
    Assume that $\Sk$ is transient. Choose an arbitrary order $L = \{x_1, x_2, \dots \}$ on $\SV$. Let $T_0 = \emptyset$. Assume that for some $i \geq 0$, the random subgraph $T_i \subset \Erd$ was constructed. Let $x_{\phi(i)}$ be the first vertex of $L$ which does not belong to  an edge in $T_i$. Run a \lr{killed random walk} $\Sk$ started at $x_{\phi(i)}$. Since $\Sk$ is transient, almost surely, one of the two following situations occur: 
    \begin{itemize}
        \item $\Sk$ hits $\rho$ or one of the vertices in $T_i$ in finite time. In this case, perform the chronological loop-erasure of the trajectory up to this time: it produces a finite path $s_i$ with no cycle disjoint of $T_i$ except for its endpoint. Add this path to $T_i$ to form $T_{i+1} = T_i \cup s_i$. 
        \item $\Sk$ never hits $T_i$ nor $\rho$. Perform the chronological loop-erasure of this infinite trajectory: this is almost surely well-defined because since $\Sk$ is transient, the path is almost surely transient. This produces an infinite path $t_i$ with no cycle. Form $T_{i+1}$ by adding this infinite path to $T_i$.
	\end{itemize}
Run this algorithm until all the vertices of $\SG$ are spanned by $T_i$ (it can take a countable number of steps if $\SG$ is infinite). This produces a random RST of $\Gr$ (recall that with our definition, a RST can have infinite unrooted components, see Remark \ref{rem:notation:RST}). By the forest-tree bijection, this RST corresponds to a RSF of $\SG$. 
\end{Def}
\begin{Def}\label{def:wilson:measure}
    We call the \emph{Wilson's measure} (following \cite{Heloise}) the probability measure on RST of $\Gr$ rooted at $\rho$ obtained by applying this algorithm and denote it by $\PW$. 
\end{Def}
When $m =0$, this algorithm is Wilson's algorithm rooted at infinity described in Section 10.1 of \cite{LyonsPeres} (for symmetric conductance functions). It produces an ESF of $\SG$ (see Remark \ref{rem:EST}). When $\SG$ is finite and $m \neq 0$, the Wilson's measure is supported on RST of $\Gr$ and coincides with the probability measure defined in Equation~\eqref{def:tree} associated with the conductances $\crho$: 
\begin{Prop}\label{prop:Wilson:finite}
    When $\SG$ is finite and $m \neq 0$, $\PW$ coincides with the measure $\PRSTr$ defined in Equation~\eqref{def:tree} associated with the conductances $\crho$: $\PW = \PRSTr$
\end{Prop}
\begin{proof}
    Proposition 5.2.6 of \cite{Chang} states that
    \begin{equation}
	\forall T \in \TT^{\rho}(\Gr),~ \PW(T) = \det(\Gk) \prod_{\ed \in T} \crho_{\ed}.
    \end{equation}
Since by the matrix-forest theorem (see Equation~\eqref{eq:matrix-forest}) and the definition of the massive Green function, $\ZRSF^{\rho}(\Gr,\crho) = \det(\Dk) = \det(\Gk)^{-1}$, $\PW$ coincides indeed with $\PRSTr$.
\end{proof}
When $\SG$ is infinite, another natural measure on RST of $\Gr$ is the weak limit over any exhaustion of the measure \eqref{def:tree}. We now explain why this coincides with the Wilson's measure. An \emph{exhaustion} of $\SG$ is a sequence of connected induced subgraphs $\SG_n = (\SV_n, \SE_n)$ such that $\SV_n$ is increasing (for union) and $\bigcup_{n \in \N} \SV_n = \SV$. We denote by $(c_n, m_n)$ the associated conductance and mass functions with wired boundary conditions. Recall that even when $m = 0$, $m_n$ takes at least one positive value since the finite set $\SV_n$ is a strict subset of the infinite set $\SV$. Consider an exhaustion $(\SG_n)_{n \in \N}$ of $\SG$, and denote by $\PW^n$ the Wilson's measures on RST of $\SG_n^{\rho}$ rooted at $\rho$ associated with the conductance and mass functions $(c_n, m_n)$.

\begin{Prop}\label{prop:Wilson}
    Assume that $\SG$ is infinite and one of the following holds:
    \begin{enumerate}[label=(\roman*)]
        \item $m=0$ and $\Sk = S$ is transient. 
        \item $m \neq 0$ and for all $x \in \SV$, $\Proba_x(\tk(\rho) < \infty)=1$: $\Sk$ dies almost surely in finite time. 
    \end{enumerate}
    The Wilson's measure $\PW$ on RST of $\Gr$ rooted at $\rho$ associated with the conductance and mass functions $(c,m)$ coincides with the weak limit of the $\PW^n$. More precisely,  
    \be\label{eq:weak:cv:exhaustion:KilledRandomWalk}
        \forall \{\ed_1,  \dots , \ed_k\} \subset \Erd,~\PW^n(\ed_1, \dots ,\ed_k) \overset{n \to \infty}{\longrightarrow} \PW(\ed_1, \dots ,\ed_k).
    \ee 
    When $m=0$, we simply have $\Erd = \Ed$ and $\PW$ is supported on EST of $\SG = \Gr$.
\end{Prop}
Observe that the assumption in case (ii) is stronger than the transience assumption. Case (ii) is exactly Proposition 5.2.6 of \cite{Chang}. Case (i) is an immediate extension of Proposition 10.1 of \cite{LyonsPeres} to the directed transient case. In their setting, the conductance function is symmetric, but the exact same proof works. 
\begin{Rem}
    Proposition~\ref{prop:Wilson} holds more generally for transient \lr{killed random walk} (which includes cases (i) and (ii): this corresponds to identifying $\rho$ with the cemetery state) by simply combining the two proofs, but we will not need it. Case (ii) of Proposition~\ref{prop:Wilson} also holds for free boundary conditions (defined in Remark \ref{rem:free:bc}) instead of wired boundary conditions, but in general case (i) does not. Nevertheless the limits with free and wired boundary conditions coincide in case (i) for the simple random walk on $\Z^d$ (see for example Corollary 10.9 of \cite{LyonsPeres}).
\end{Rem}
This proposition combined with Proposition~\ref{prop:Wilson:finite} shows that Wilson's measure that we introduced on RSF of infinite graphs corresponds to the weak limit along any exhaustion with wired boundary conditions of the measures defined in Equation~\eqref{def:tree} (for the conductances $\crho$ on $\Gr$). Taking weak limits for edge probabilities, one deduces that the transfer current theorem also holds for infinite directed graphs. Recall the definition of the transfer current operator $\Hk$ when $\Sk$ is transient. Recall that when $m = 0$, $\Hk = H$ is the usual transfer current operator associated with a transient \lr{random walk}.

\begin{Thm}[Infinite Transfer Current Theorem]\label{thm:transfer}
    Assume that $\SG$ is infinite and one of the following holds:
    \begin{enumerate}[label=(\roman*)]
        \item $m=0$ and $\Sk = S$ is transient.
        \item $m \neq 0$ and for all $x \in \SV$, $\Proba_x(\tk(\rho) < \infty)=1$. 
    \end{enumerate}
    The Wilson's measure $\PW$ on RST of $\Gr$ satisfies
        \be\label{TransferCurrentRecurrent}
		\forall \{\ed_1, \dots ,\ed_k\} \subset \Erd,~\PW(\ed_1, \dots ,\ed_k) = \det\left((\Hk_{\ed_i,\ed_j})_{i,j = 1, \dots ,k}\right)\prod_{i=1}^k \crho_{\ed_i}.
	\ee
    When $m=0$, $\Gr = \SG$, $\Erd = \Ed$, $\Hk = H$ and $\PW$ is supported on EST of $\SG$.
\end{Thm}
\begin{proof}
Case (ii) is proved in the discussion after Remark 18 of \cite{Chang}. Let us prove case (i) in the same spirit. Assume that $m =0$ and $S$ is transient. Let $(\SG_n)_{n \in \N}$ be an exhaustion of $\SG$. Denote by $\PW^n$ the Wilson's measure on RST of $\SG_n$ associated with $(c_n,m_n)$. Recall that $m_n \neq 0$ since $\SV_n$ is a strict subset of $\SV$. Let $\ed_1,\dots,\ed_k \in \Ed$. Case (i) of Proposition~\ref{prop:Wilson} shows that
    \be\label{eq:cv:transfer:exhaustion}
        \PW(\ed_1,\dots,\ed_k) = \lim_{n \to \infty}\PW^n(\ed_1,\dots,\ed_k).
    \ee
    Since by Proposition~\ref{prop:Wilson:finite}, on finite graphs the Wilson's measure coincides with the natural measure defined in \eqref{def:tree}, the finite transfer current Theorem~\ref{thm:transfer:finite} applies with the transfer current operator $\Hk_n$ associated with the potential
    \begin{equation}
        \forall x,y \in \SV_n,~\Vk_n(x,y) = \E_x\Bigg[\sum_{i=0}^{\tau(\SV \setminus \SV_n)-1}\mathbbm{1}_{\{S_i=y\}}\Bigg],
    \end{equation}
    by definition of the potential of a \lr{killed random walk} and since the \lr{killed random walk} on $\SG_n$ associated with $(c_n, m_n)$ is simply $S$ killed when it exits $\SV_n$. For fixed $x,y \in \SV$, by transience of $S$:
    \begin{equation}
        \Vk_n(x,y) \overset{n \to \infty}{\longrightarrow} V(x,y) = \E_x\Bigg[\sum_{i=0}^{\infty}\mathbbm{1}_{\{S_i=y\}}\Bigg].
    \end{equation}
    Together with \eqref{eq:cv:transfer:exhaustion}, it proves the theorem with the associated transfer current operator
    \[
	\forall \ed = (wx), \fd = (yz) \in \Ed,~ H_{\ed,\fd} = \frac{V(w,y)}{c(y)}-\frac{V(x,y)}{c(y)}.\qedhere
    \]
\end{proof}

\subsection{The Doob transform technique}\label{subsec:doob}
Let $(\oSG,\oc,\om)$ be a graph with conductance and mass functions. Let $\SV \subset \oSV$ and denote by $\SG$ the induced subgraph and by $(c,m)$ the conductance and mass function with wired boundary conditions. Recall that the restriction $\oDk_{\SV}$ is the massive Laplacian on $\SG$ with conductance and mass functions $(c,m)$. Consider a positive function $\lambda: \oSV \to \R_{>0}$, and define a new conductance function on $\oSG$:
\be\label{eq:def:doob:conductances}
    \forall x \overset{\oSG} \sim y \in \oSV,~\ocl_{(x,y)} := \frac{\lambda(y)}{\lambda(x)}\oc_{(x,y)}.
\ee
Let $\oSl$, $\oDl$, $\Ql$ be respectively the \lr{random walk}, non-massive Laplacian and transition kernel on $\oSG$ associated with these conductances. The \lr{killed random walk} on the induced subgraph $\SG$ associated with the conductance and mass functions with wired boundary conditions $(\cl, \ml)$ induced on $\SG$ by $\ocl$ is simply the \lr{random walk} $\oSl$ killed when it exits $\SV$. As it was already noted, the Laplacian associated with $(\SG, \cl, \ml)$ is the restriction $\oDl_{\SV}$. Recall that when $\SV$ is a strict subset of $\oSV$, $\ml$ takes at least one positive value even though $\oSl$ is a non-massive \lr{random walk}, so $\Dl_{\SV}$ is a massive Laplacian. When $\SV = \oSV$, $c = \oc$ and $\Dl_{\SV} = \Dl$ is non-massive. This is the case that the reader should keep in mind as it is the more intuitive, the additional technicalities are only here to handle the general case.

\begin{Def}\label{def:doob:transform}[Definition 8-11 of \cite{kemeny2012denumerable}]
    The \lr{random walk} $\Slk$ is called the \emph{Doob transform} (or sometimes \emph{$h$-transform}) of $\oSk$ by $\lambda$.
\end{Def}
Let $\Lambda = D(\lambda)$ be the diagonal matrix with rows and columns indexed by $\oSV$ and diagonal entries $\lambda$. Recall that $\lambda$ is said to be massive harmonic on $\SV$ if for all $x \in \SV$, $(\Dk \lambda)(x) = 0$. Recall that $\Qk$ is the transition kernel of the \lr{killed random walk} $\Sk$. The following holds:

\begin{Prop}\label{prop:doob}
    If $\lambda\lr{: \oSV \to \R_{>0}}$ is massive harmonic for $\oDk$ on $\SV$, then
    \begin{equation}
        \Dl_{\SV} = \Lambda_{\SV}^{-1} \Dk_{\SV} \Lambda_{\SV}. 
    \end{equation}
    Equivalently, 
    \be\label{eq:kernel:doob}
        \forall x,y \in \SV,~\Qlk(x,y) = \frac{\lambda(y)}{\lambda(x)}\Qk(x,y).
    \ee
\end{Prop}
The Laplacians $\Dl_{\SV}$ and $\Dk_{\SV}$ are said to be \emph{gauge equivalent} with gauge $\Lambda_{\SV}$. When $\SV = \oSV$, this is a gauge equivalence between a non-massive Laplacian $\Dl$ and a massive Laplacian $\Dk$.
\begin{proof}
For all $x \sim y \in \SV$, $x \neq y$
\be\label{eq:equality:out:diagonal}
     (\Lambda^{-1}_{\SV} \Dk_{\SV} \Lambda_{\SV})_{x,y} = \frac{\lambda(y)}{\lambda(x)}\Dk_{x,y} = -\frac{\lambda(y)}{\lambda(x)}c_{(x,y)} = -\cl_{(x,y)} = \Dlk_{x,y}. 
\ee
Moreover, for $x \in \SV$, by massive harmonicity of $\lambda$ at $x$ for $\oDk$:
\be\label{eq:equality:total:c}
    \ck(x) = m(x) + c(x) = \om(x) + \oc(x) = \om(x) + \sum_{y \overset{\oSG}{\sim} x}\oc_{(x,y)} = \sum_{y \overset{\oSG}{\sim} x}\frac{\lambda(y)}{\lambda(x)}\oc_{(x,y)} = \ml(x) + \cl(x).
\ee
This implies, for $x \in \SV$
\be\label{eq:equality:diagonal}
    \left(\Lambda^{-1}_{\SV}\Dk_{\SV}\Lambda_{\SV}\right)_{x,x} = \Dk_{x,x} = \ck(x) - c_{(x,x)} = \ml(x)+\cl(x)-\cl_{(x,x)} = \Dlk_{x,x},
\ee
which proves the first point of the Lemma. The second point of the Proposition is a direct consequence of Equation~\eqref{eq:equality:total:c} and the definition of the transition kernels in Equation~\eqref{eq:def:transition:kernel} (and Equation~\eqref{eq:def:massive:transition:kernel} when $\SV = \oSV$):
\begin{equation}
    \Ql_{\SV} = I - D(\cl + \ml)^{-1}\Dl_{\SV} = I - D(\ck) \Lambda_{\SV}^{-1} \Dk_{\SV} \Lambda_{\SV} = \Lambda^{-1}_{\SV}\Qk_{\SV}\Lambda_{\SV}.
\end{equation}
\end{proof}
A first property of the Doob transform is the following: 
\begin{Prop}\label{prop:doob:transient}
    Assume that $\Sk$ is transient and $\SV = \oSV$. If $\lambda$ is massive harmonic for $\oDk$ on $\SV$, the Doob transform $\Sl$ is also transient. 
\end{Prop}
This holds more generally for $\SV \subset \oSV$ with the same proof but we will not need it and we find it more clear to provide a proof in the intuitive case $\SV = \oSV$ where the notation is less cumbersome.
\begin{proof}
    It is left as an exercise after Definition 8-11 of \cite{kemeny2012denumerable}, we prove it to illustrate our definitions. For $\SA \subset \SV$, we denote respectively by $(\tk)^+(\SA)$, $(\tl)^+(\SA)$ the first positive hitting time of $\SA$ for $\Sk$ and $\Sl$. Let $x, y \in \SV$, $n \in \N$. By summing over all paths in $\SV$ of length $n$ from $x$ to $y$ that do not hit $y$ except at their extremities,
    \begin{equation}
        \Proba_x\big((\tl)^+(y) = n\big) 
        \overset{\eqref{eq:kernel:doob}}{=} \frac{\lambda(y)}{\lambda(x)}\Proba_x\big((\tk)^+(y) = n\big).
    \end{equation}
    Since $\Sk$ is transient, summing over all $n$ and taking $y =x$ gives
    \[
        \Proba_x\big((\tl)^+(x) < \infty\big) = \Proba_x\big((\tk)^+(x) < \infty\big) <1.\qedhere
    \]
\end{proof}
The \emph{Doob transform technique} is often used to study general Markov chains conditioned by a certain event. We refer to Chapter 8 of \cite{kemeny2012denumerable} for a general introduction from the potential theory point of view and to Section 7 of \cite{DoobTransform} for a recent and clear introduction. It is a powerful tool that works in very general settings. The difficulty is to find a setting where these results can be applied and give quantitative estimate.

\section{From rooted spanning forests to directed spanning trees}\label{sec:doob}
This section has two goals: the first one (developed in Section~\ref{sec:doob:rsf}) is to show some applications of the Doob transform technique to the random RSF and RST models when a positive massive harmonic function is available. The idea is to show that a RSF model on a graph $\SG$ can be related to a RST model \emph{on the same graph} (infinite case) or on a modified but similar graph $\Go$ (finite case). \lr{The crucial difference with the forest-tree bijection, which relates RSF on $\SG$ with RST on $\Gr$, is the nature of the graph $\Gr$ or $\Go$. The former is obtained by adding an edge from every vertex (with positive mass) to an additional vertex $\rho$ while the latter is obtained by adding an edge only from boundary vertices to an additional vertex $o$. In the applications, we will see that $\Go$ (in the finite case) or $\SG$ (in the infinite case) retains some properties of $\SG$ (planarity, periodicity), contrary to $\Gr$. Our main applications will be to the case of planar graphs. To apply dimers techniques, it is very important that the modified graph $\Go$ is planar, which is typically not the case of $\Gr$ (see Figure \ref{fig:Gr}).}

The second goal (developed in Section~\ref{sec:find:lambda}) is to justify that it is always possible to find positive massive harmonic functions on $\SG$ (infinite case) or on $\Go$ (finite case) and hence to apply Theorem~\ref{thm:doob}. In the finite case it is impossible to find positive massive harmonic functions on $\SG$, see \lr{Subsection}~\ref{FiniteRandomWalk} (which explains why the statements are heavier in the finite case).

\subsection{The Doob transform technique for rooted spanning forests}\label{sec:doob:rsf}
We start by introducing the notation in the intuitive case of an infinite graph, and then show how it has to be adapted in the finite case. The reader can in a first time skip the finite case and jump directly to the statement of Theorem~\ref{thm:doob}.

\subsubsection{The infinite case}\label{par:infinite}
Assume that $\SG$ is infinite and that $\Sk$ dies almost surely in finite time as in (ii) of Theorem~\ref{thm:transfer}. Let $\Vk,\Hk,\PWk$ be respectively the associated potential, transfer current operator and Wilson's measure on RSF of $\SG$.\par
Let $\lambda$ be a positive massive harmonic function on $\SV$. The Doob transform technique of Section~\ref{subsec:doob} applies: denote by $\Sl$ the Doob transform of $\Sk$ by $\lambda$ on $\SV$: it is the \lr{random walk} associated with the conductance function $\cl$ and the non-massive Laplacian $\Dl$. Denote by $\Vl$ the associated potential of the transient \lr{random walk} $\Sl$:
\be\label{def:Vlambda}
	\forall x,y \in \SV,~\Vl(x,y) = \E_x\left[\sum_{i =0}^{\infty} \mathbbm{1}_{\{\Sl_i=y\}}\right]
\ee 
and by $\Hl, \PWl$ the associated transfer current operator on $\SV$ and Wilson's measure on EST of $\SG$. Proposition~\ref{prop:doob} shows that $\Dk$ and $\Dl$ are gauge equivalent:
\be\label{eq:gauge:doob:infinite}
    \Dl = \Lambda^{-1} \Dk \Lambda.
\ee
The implications of this equation for the RSF and ERST models are the content of Theorem~\ref{thm:doob}.

\subsubsection{The finite case}\label{par:finite}
Consider a connected graph $\oSG = (\oSV,\oSE)$ with conductance and mass functions $(\oc,\om)$ and denote by $\oDk$ and $\oSk$ the associated massive Laplacian and \lr{killed random walk}. Let $\SV \subset \oSV$ be a strict \emph{finite} subset such that the induced subgraph $\SG$ is connected, and denote by $(c,m)$ the conductance and mass functions with wired boundary conditions induced by $(\oc, \om)$ on $\SG$. Recall that the restriction $\oDk_{\SV}$ is the massive Laplacian on $\SG$ associated with $(c,m)$. Let $\lambda: \oSV \to \R_{>0}$ be massive harmonic for $\oDk$ on $\SV$. Denote by $\Slk$ and $\Dlk$ the Doob transform of $\oSk$ by $\lambda$ on $\SV$, and by $(\cl, \ml)$ the conductance and mass function with wired boundary conditions induced by $\ocl$ on $\SG$: $\Dlk_{\SV}$ is the Laplacian associated with $(\cl, \ml)$. Observe that unlike in \lr{Subsection}~\ref{par:infinite}, $\SV \neq \oSV$ so $\ml \neq 0$. Proposition~\ref{prop:doob} holds and states that:
\be\label{eq:gauge:doob:finite}
    \Dlk_{\SV} = \Lambda^{-1}_{\SV} \Dk_{\SV} \Lambda_{\SV}.
\ee
The Laplacians $\Dk_{\SV}$ and $\Dlk_{\SV}$ are related to RSF models on $\SG$ with respective conductance and mass functions $(c,m)$ and $(\cl, \ml)$. Since $\oSl$ is a \lr{random walk} with no mass, we are in the setting described around Equation~\eqref{eq:mass:wired}. Recall that $\Go$ is the graph obtained from $\SG$ by identifying all the vertices outside $\SV$ to an additional vertex $o$. By the forest-tree bijection, RSF of $\SG$ weighted by $(\cl, \ml)$ are in weight-preserving bijection with RST of $\Go$ weighted by $\cl$. \emph{We use the notation $o$ instead of $\rho$ to insist on the fact that the graphs $\Go$ and $\Gr$ are different (compare Figure \ref{fig:Gr} and \ref{fig:Go})}, and this is the key point of our work.\par
\lr{Let $\Vk, \Hk$ denote the potential and transfer current operators associated with the \lr{killed random walk} $\Sk$ on $\SG$ with conductance and mass functions $(c,m)$.} Let $\Vlk$, $\Hlk$ be the potential and transfer current operator associated with the \lr{random walk} $\Slk$ killed when it exits $\SV$ (or more formally, with the \lr{killed random walk} on $\SG$ associated with the conductance and mass functions $(\cl, \ml)$). Observe the difference with \lr{Subsection}~\ref{par:infinite} where they are the potential and transfer operators $\Vl$ and $\Hl$ associated with a \lr{random walk} \emph{with no killing}. Let $\PWk$ and $\PWlk$ be the Wilson's measure on RST of $\Gr$ (rooted at $\rho$) and RST of $\Go$ (rooted at $o$) associated with $\Sk$ and $\Slk$. 

\subsubsection{Statement and proof of the main result}
Consider either the setting of \lr{Subsection}~\ref{par:infinite} ($\SG$ is infinite) or the setting of \lr{Subsection}~\ref{par:finite} ($\SG$ is finite). The following holds:
\begin{Thm}\label{thm:doob}
    The Wilson's measures $\PWk$ and $\PWlk$ are naturally related: they both satisfy the transfer current Theorem~\ref{thm:transfer:finite} with respective transfer current operators
    \be\label{eq:thm:transfer:doob}
        \left\{
        \begin{array}{ll}
             \forall \ed = (w,x), \fd = (y,z) \in \Erd,~
	&\Hk_{\ed,\fd} = \frac{\Vk(w,y)}{\crho(y)}-\frac{\Vk(x,y)}{\crho(y)}\\
             \forall \ed = (w,x), \fd = (y,z) \in \Eod,~
        &\Hlk_{\ed,\fd} = \frac{\lambda(y)}{\lambda(w)}\frac{\Vk(w,y)}{\crho(y)} - \frac{\lambda(y)}{\lambda(x)}\frac{\Vk(x,y)}{\crho(y)}.
        \end{array}
        \right.
    \ee
    When $\SG$ is infinite, $\Eod = \Ed$. When $\SG$ is finite, we use the conventions $\Vk(o,\cdot) = \Vk(\cdot,o) = 0$. Moreover, when $\SG$ is finite, the following equality of partition functions holds:
    \be\label{eq:doob}
	\ZRSF(\SG,c,m) = \ZRST^{\rho}(\Gr, \crho) = \ZRST^{o}(\SG^o, \clo).
    \ee
\end{Thm}
When $\SG$ is infinite, this theorem relates the RSF model on $\SG$ with conductance and mass functions $(c,m)$ with an EST model on $\SG$ with conductance function $\cl$. When $\SG$ is finite, it relates a RSF model on $\SG$ with another RSF model on $\SG$ with the mass “sent to the boundary”.

\begin{proof}
    We first prove the equality of partition functions in the finite case. The first equality in \eqref{eq:doob} is the matrix-forest theorem stated in Equation~\eqref{eq:matrix-forest}. For the second equality, the matrix-tree theorem implies on the one hand 
    \begin{equation}
        \ZRST^{o}(\SG^o, \clo) = \det(\Dlk_{\SV})
   \end{equation} 
   and on the other hand
   \begin{equation}
        \ZRST^{\rho}(\Gr, \crho) = \det(\Dk_{\SV}).
   \end{equation}
    The gauge equivalence of Equation~\eqref{eq:gauge:doob:finite} and the fact that the determinant is invariant by gauge change imply the second equality of Equation~\eqref{eq:doob}.\par
    In \lr{Subsection}~\ref{par:infinite}, we assumed that $\Sk$ dies almost surely in finite time so the first statement of Equation~\eqref{eq:thm:transfer:doob} is case (ii) of Theorem~\ref{thm:transfer}. In this case, the \lr{random walk} $\Sl$ is transient by Proposition~\ref{prop:doob:transient} so it satisfies the hypothesis of case (i) of Theorem~\ref{thm:transfer} with the transfer current operator $\Hl$ on $\SG$. In the setting of \lr{Subsection}~\ref{par:finite}, the first statement of Equation~\eqref{eq:thm:transfer:doob} is simply case (i) of Theorem~\ref{thm:transfer:finite}. In this case, the \lr{random walk} $\Slk$ killed when exiting $\SV$ is a \lr{killed random walk} on a finite graph and Theorem~\ref{thm:transfer:finite} holds with the potential and transfer current operators $\Vlk, \Hlk$ on $\Go$. Note that by convention, $\Vlk(x,o) = \Vlk(o,x) = 0$ for all $x \in \SV$.\par
    We only have to show that in both cases, the transfer current operator $\Hlk$ takes the form given in Equation~\eqref{eq:thm:transfer:doob}. By Equation~\eqref{eq:kernel:doob}, 
    \begin{equation}
        \forall~i \in \N,~\forall x,y \in \SV,~ \Proba_x\big(\Slk_i = y\big) = \frac{\lambda(y)}{\lambda(x)}\Proba_x(\Sk_i = y),
    \end{equation}
    which implies for all $x,y \in \SV$
    \be\label{eq:Green:gauge}
        \Vlk(x,y) = \sum_{i =0}^{\infty} \Proba_x\big(\Slk_i = y\big) = \frac{\lambda(y)}{\lambda(x)} \sum_{i =0}^{\infty} \Proba_x(\Sk_i = y) = \frac{\lambda(y)}{\lambda(x)} \Vk(x,y).
    \ee
    Note that in the finite case, this also holds when $x=o$ or $y=o$ since both sides are $0$. Moreover, Proposition~\ref{prop:doob} (and more precisely Equation~\eqref{eq:equality:diagonal}) implies that for all $x \in \SV$, $\crho(x) = \cl(x)$. Hence, for all $\ed = (w,x) \in \Eod, \fd = (y,z) \in \Eod$ (this is simply $\Ed$ when $\SG$ is infinite), the definition of the transfer current operator $\Hlk$ becomes
    \[
        \Hlk_{\ed,\fd} = \frac{\Vlk(w,y)}{\cl(y)}-\frac{\Vlk(x,y)}{\cl(y)} = \frac{\lambda(y)}{\lambda(w)}\frac{\Vk(w,y)}{\crho(y)} - \frac{\lambda(y)}{\lambda(x)}\frac{\Vk(x,y)}{\crho(y)}.\qedhere
    \]
\end{proof}

\subsection{Positive massive harmonic functions}\label{sec:find:lambda}

\subsubsection{The finite case}\label{FiniteRandomWalk}
Consider a finite graph $\SG$ with conductance and mass functions $(c,m)$ and associated massive Laplacian $\Dk$. 
\begin{Prop}
    If $m \neq 0$, there are no positive massive harmonic functions on $\SV$.
\end{Prop}
\begin{proof}
    We proceed by contradiction and assume that there exists a positive massive harmonic function $\lambda$. We apply the maximum principle to show that $\lambda$ is constant: by contradiction, if $\lambda$ is not constant, since $\SG$ is connected we can pick a point $x \in \SV$ where $\lambda$ reaches its maximum and a neighbouring point $y \sim x$ with $\lambda(y) < \lambda(x)$. Massive harmonicity at $x$ implies
    \begin{equation}
        \sum_{y \sim x}c_{(x,y)}\lambda(x) \leq \left(m(x) + \sum_{y \sim x}c_{(x,y)}\right)\lambda(x) = \sum_{y \sim x}c_{(x,y)}\lambda(y) < \sum_{y \sim x}c_{(x,y)}\lambda(x),
    \end{equation}
    which is a contradiction. Hence $\lambda$ is constant (and positive), and for all $x \in \SV$, massive harmonicity at $x$ implies $c(x) = \crho(x) = c(x)+m(x)$ which contradicts the fact that $m \neq 0$.
\end{proof}
This justifies why when $\SG$ is finite we cannot use the same simple setting as in \lr{Subsection}~\ref{par:infinite}. Consider now $(\oSG, \oc, \om)$, \lr{let $\SV \subset \oSV$ be a \emph{strict} finite subset of $\oSV$ and let $(\SG,c,m)$ be the finite connected induced subgraph with wired boundary conditions as in \lr{Subsection}~\ref{par:finite}.}
\begin{Prop}\label{prop:find:lambda:finite}
    There exists a positive function $\lambda$ on $\oSV$ satisfying $(\oDk \lambda)(x) = 0$ for all $x \in \SV$. 
\end{Prop}
\begin{proof}
    Recall that $\oSk$ is the \lr{killed random walk} on $\oSG$ associated with $(\oc, \om)$ and let $\oVk$ be the associated potential operator on $\oSV$. 
    Fix any $z \in \oSV \setminus \SV$ and let $\lambda(\cdot) = \oVk(\cdot,z)$. This defines a function which is positive on $\oSV$ since $\oSG$ is connected and massive harmonic on $\oSV \setminus \{z\}$ by Equation~\eqref{eq:harmonicity:potential}, so in particular it is massive harmonic on $\SV$.
\end{proof}
Proposition~\ref{prop:find:lambda:finite} justifies that the statement of Theorem~\ref{thm:doob} \lr{is} non empty \lr{in the finite setting}. A more explicit way of finding massive harmonic functions as in \lr{Subsection}~\ref{par:finite} is to use an explicit massive harmonic function $\lambda$ on an infinite graph $\oSG$ with massive Laplacian $\oDk$ (see Example~\ref{ex:explicit:lambda} in the next \lr{subsection}) and apply the Doob transform technique on a finite induced subgraph $\SG$ of $\oSG$. This is the starting point of Section~\ref{sec:near-critical}.

\subsubsection{The infinite case: Martin boundary of a killed random walk}\label{subsec:Martin}
For some specific infinite graphs and Laplacians (for example isoradial graphs or $\Z^d$-periodic graphs), explicit positive massive harmonic functions are known. This will be detailed in \lr{Example~\ref{ex:explicit:lambda}} and in Section~\ref{sec:near-critical} and Section~\ref{subsec:periodic}.\par
In this \lr{subsection}, we show that there always exist positive massive harmonic functions on an infinite graph, which justifies that the statements of Section~\ref{sec:doob:rsf} are non-empty. For non-massive \lr{random walk}, or \lr{killed random walk} with a constant killing function $m(x) = m >0~\forall x \in \SV$, harmonic and superharmonic functions are described by the theory of the Martin boundary (see Section 24 of \cite{Woess} for a general introduction in the transient case). To our knowledge, the general theory has not been developed for \lr{killed random walk} with non-constant killing functions, but some examples (which we will recall at the end of the \lr{subsection}) were considered in \cite{MartinBoundary}. We give a sketch of what such a general theory would look like and prove a first result. It would be interesting to push this further.

Consider an infinite irreducible graph with a conductance function $c$ and a mass function $m \neq 0$. Recall that the \lr{killed random walk} is transient and its potential $\Vk$ is defined by Equation~\eqref{eq:def:potential}. Let $x_0 \in \SV$ be a reference point. For $y \in \SV$, the \emph{killed Martin kernel} is defined as the function $\Kc(\cdot,y):\SV \to \R_{>0}$
\begin{equation}
    \forall x \in \SV,~\Kc(x,y) = \frac{\Vk(x,y)}{\Vk(x_0,y)}.
\end{equation}
Consider a sequence of vertices $\zeta = (z_n)_{n \in \N} \in \SV^{\N}$ going to $\infty$. We say that $\zeta$ converges in the \emph{killed Martin \lr{compactification}} if for all $x \in \SV$, the killed Martin kernel converges towards a limit in $(0,\infty)$
\be\label{eq:martin:cv}
    \Kc(x, \zeta) := \lim_{n \to \infty} \frac{\Vk(x,z_n)}{\Vk(x_0,z_n)} \in (0, +\infty).
\ee 
\lr{This defines a mode of convergence for sequences of vertices.} Two sequences $\zeta_1, \zeta_2$ converging \lr{in} the Martin \lr{compactification} are \emph{equivalent} and we write $\zeta_1 \sim \zeta_2$ if the corresponding Martin kernel $\Kc(\cdot, \zeta_1)$ and $\Kc(\cdot, \zeta_2)$ coincide. The \emph{killed Martin boundary} $\MG$ is defined as \lr{the equivalence classes of sequences which do not contain a constant subsequence, equal to a vertex of $\SG$.} 
\begin{Prop}\label{prop:Martin:non-empty}
    The killed Martin boundary of a transient \lr{killed random walk} is non empty. More precisely, for all sequence $(z_n)_{n \in \N} \in \SV^{\N}$ such that $z_n \overset{n \to \infty}{\longrightarrow} \infty$, there exists an extraction $k_n \overset{n \to \infty}{\longrightarrow} \infty$ such that the subsequence $(z_{k_n})_{n \in \N}$ converges \lr{in} the killed Martin boundary.
\end{Prop}
We could not find this result under such general hypothesis so we provide a proof even though this will not seem new to the reader acquainted with the usual Martin boundary theory. 
\begin{proof}
For all $x \in \SV$, since $\SG$ is connected there exists $n_0 \in \N$ such that $\Proba_{x_0}(\Sk_{n_0}=x) > 0$. Hence by the simple Markov property,
\begin{equation}
    \frac{\Vk(x,z_n)}{\Vk(x_0,z_n)} \geq \frac{\Proba_{x_0}(\Sk_{n_0}=x, \tau^{\rho}(z_n) > n_0)\Vk(x_0,z_n)}{\Vk(x_0,z_n)} \overset{n \to \infty}{\longrightarrow} \Proba_{x_0}(\Sk_{n_0}=x) > 0
\end{equation}
where the convergence is justified by the fact that $z_n \to \infty$. This shows that the ratio is bounded from below: there exists $\eps(x) >0$ such that for all $n$ large enough,
\begin{equation}
    \frac{\Vk(x,z_n)}{\Vk(x_0,z_n)} \geq \eps(x).
\end{equation}
The reverse argument works the same way (by exchanging $x_0$ and $x$): there exists $E(x)>0$ such that for all $n$ large enough,
\begin{equation}
    \frac{\Vk(x,z_n)}{\Vk(x_0,z_n)} \leq E(x).
\end{equation}
Hence, for fixed $x$, there exists a subsequence $z_{k_n}$ with $k_n \to \infty$ such that the ratio
$\frac{\Vk(x,z_{k_n})}{\Vk(x_0,z_{k_n})}$ converges. Since this holds for all $x \in \SV$ \lr{and $\SV$ is countable}, by diagonal extraction we can find a common converging subsequence.
\end{proof}
In the usual Martin boundary theory (i.e. the \lr{killed random walk} with constant mass function $m \geq 0$), an integral representation theorem shows that all harmonic functions are obtained as integrals of Martin kernels. We do not know if such a result holds for \lr{killed random walk} with general killing function, but we can still use that limits of killed Martin kernels provide us with massive harmonic functions on the full graph:
\begin{Prop}
    If $\zeta \in \MG$, $\Kc(\cdot, \zeta)$ is a positive massive harmonic function on $\SV$.
\end{Prop}
Observe that this is in some sense the limit of the finite case since it corresponds to sending $z \to \infty$ in Proposition~\ref{prop:find:lambda:finite}.
\begin{proof}
Let $\zeta = (z_n)_{n \in \N}$. Since by Equation~\eqref{eq:harmonicity:potential} $\Vk(\cdot, z_n)$ is a massive harmonic function on $\SV \setminus \{z_n\}$, we have that for fixed $x \in \SV$ and $n$ large enough such that $x \neq z_n$:
\begin{equation}
   (\Dk \Vk(\cdot, z_n))(x) = 0. 
\end{equation}
Dividing by $\Vk(x_0, z_n)$ and letting $n \to \infty$ gives the massive harmonicity of the killed Martin kernel.
\end{proof}
We now give some examples where this killed Martin boundary was studied.
\begin{Ex}\label{ex:explicit:lambda}
\leavevmode
    \begin{itemize}
        \item The degenerate case of our model where the \lr{random walk} associated to the conductances $c$ is recurrent and the mass is infinite at a point and $0$ elsewhere is the setting of the recurrent Martin boundary theory, which is well developed and written out for example in Chapter 9 of \cite{kemeny2012denumerable}. In particular, for the simple random walk on $\Z^2$ killed almost surely at the origin, the convergence in Equation~\eqref{eq:martin:cv} holds and the left hand side is the usual potential kernel for any sequence $(z_n)_{n \in \N}$ going to infinity.
        \item If the graph $\SG$ is isoradial and $\Dk$ is the $Z$-invariant massive Laplacian of \cite{MassiveLaplacian}, the main result of \cite{MartinBoundary} is the asymptotics of the killed Martin kernel $\Kc(\cdot, \zeta)$ depending on how the sequence $\zeta$ goes to infinity. This asymptotic is described explicitly: it is the discrete massive exponential evaluated at a certain point, so the killed Martin boundary is fully described in this case.
        \item If the graph $\SG$ with conductance and mass functions $(c,m)$ are $\Z^d$-periodic, the asymptotic of the killed Martin kernel is also discussed in \cite{MartinBoundary} (building on earlier results). Theorem 5 of \cite{MartinBoundary} establishes convergence of the killed Martin kernel for translated copies of the same vertex $x \in \SV$. More precisely, they show that for $x \in \SV$, $y = x + (k_1,\dots,k_d)$ and $z_n = u+(i_1^n,\dots,i_d^n)$ such that $z_n \to \infty$ and $\frac{z_n}{|z_n|}$ converges to a limit $r$,
    \begin{equation}\label{expdecrease}
        \lim_{n \to \infty} \frac{\Vk(x,z_n)}{\Vk(y,z_n)} = \exp(\nu(r) \cdot (k_1,\dots,k_d))
    \end{equation}
    where $\nu$ is defined precisely in \cite{MartinBoundary}. 
    \end{itemize}
\end{Ex}

\section{The planar case: from rooted spanning forests to dimers}\label{sec:planar}
The goal of this section is to explain how our results relating a RSF model on a planar graph $\SG$ with a RST model on a related planar graph $\Go$ can be used to relate the RSF model on $\SG$ with a dimer model on a planar graph related to $\Go$. 

\subsection{Preliminary: Temperley's bijection}\label{SubsecTemperley}
Temperley's bijection was first stated in \cite{Temperley74} as a bijection between (undirected) spanning trees of a finite grid and dimer configurations of an associated grid. It was later extended by several authors. The version that we describe is a weight-preserving bijection between weighted RST of a planar graph $\SG$ and weighted dimer configurations of its double graph. This extension can be found in Section 2 of \cite{TreesMatchings}. For the notation we align as much as possible with Section 2.3.1 of \cite{ZinvDirac}. \emph{Warning}: in the bijection that we present here, the primal and dual graph are exchanged compared to the literature.\par

\paragraph{Dimers on a bipartite graph.} \lr{For a more general introduction to the subject, see for example the lecture notes \cite{KenLectureNotes}.} Let $\SG = (\SV, \SE)$ be a bipartite graph (finite or infinite). Its vertex set can be partitioned into white and black vertices $\SV = \SW \sqcup \SB$ such that there is no edge between two vertices of the same color. We assume that $\SG$ has the same number of white and black vertices $|\SW| = |\SB|$ \lr{(when $\SV$ is infinite, both $\SW$ and $\SB$ must be infinite)}. A \emph{dimer configuration} on $\SG$ is a subset of edges $\SM \subset \SE$ such that every vertex in $\SV$ is incident to exactly one edge of this subset. The set of dimer configurations of $\SG$ is denoted by $\M(\SG)$. A dimer model on $\SG$ is specified by a weight function on the edges $\nu: e \in \SE \to \nu_e \geq 0$. The weight of a dimer configuration is 
\be\label{eq:def:weight:dimer}
    \forall \SM \in \M(\SG),~\nd(\SM) = \prod_{e \in \SM}\nu_e.
\ee
The partition function of the dimer model is the weighted sum
\be\label{eq:partition:dimer}
    \Zd(\SG,\nu) = \sum_{\SM \in \M(\SG)}\nd(\SM).
\ee
When $\SG$ is finite and $\M(\SG) \neq \emptyset$, $0 < \Zd(\SG,\nu)< \infty$ so this weight function induces a probability measure on dimer configurations:
\be\label{eq:def:proba:dimer}
    \forall \SM \in \M(\SG),~ \Pd(\SM) = \frac{\nd(\SM)}{\Zd(\SG,\nu)}.
\ee
For all $\{e_1,\dots,e_k\} \subset \SE$, we denote by $\Pd(e_1,\dots,e_k)$ the measure of the set of all dimer configurations containing $\{e_1,\dots,e_k\}$.

\lr{Let us assume that $\SG$ is finite.} A \emph{set of phases} $\zeta$ is the attribution of a complex number of modulus one $\zeta_{wb}$ to each edge $wb \in \SE$. It satisfies the \emph{Kasteleyn property} if, for all inner face of $\SG$ bounded by the $2p$ vertices $w_1,b_1,\dots,w_p,b_p$ in clockwise order, 
\begin{equation}
    \frac{\zeta_{w_1b_1} \dots \zeta_{w_pb_p}}{\zeta_{w_2b_1} \dots \zeta_{w_1b_p}} = (-1)^{p-1}.
\end{equation}
Given a set of phases $\zeta$ on $\SG$ satisfying the Kasteleyn property, the \emph{Kasteleyn matrix} of the dimer model associated with these phases is the matrix with rows indexed by white vertices and columns by black vertices
\be\label{eq:def:kasteleyn:matrix}
    \forall w \in \SW,~b \in \SB,~K_{w,b} = 
    \left\{
    \begin{array}{ll}
         \zeta_{wb}\nu_{wb} & \text{ if }w \sim b \\
         0& \text{ otherwise.} 
    \end{array}
    \right.
\ee
It has been known since \cite{Kas61, Temperley} (see \cite{Kuperberg1998} for the version with phases that we use here) that the determinant of the Kasteleyn matrix counts the weighted sum of dimer configurations:
\be\label{eq:Kasteleyn:formula}
    |\det K| = \Zd(\SG, \nu).
\ee
When $\M(\SG) \neq \emptyset$, this implies that the Kasteleyn matrix is invertible and the probability measure $\Pd$ is related to the Kasteleyn matrix $K$ by the \emph{local statistics formula} of \cite{Kenyon_1997}:
\be\label{eq:local:statistics}
    \forall~\{w_1b_1,\dots,w_kb_k\} \subset \SE,~\Pd\big(w_1b_1,\dots,w_kb_k\big) = \lr{\prod_{i=1}^k K_{w_i,b_j} \det \big(K^{-1}_{b_i,w_j}\big)_{1\leq i,j \leq k}}.
\ee 
The dimer models associated with two weight functions $\nu^1, \nu^2$ are \emph{gauge equivalent} with gauge functions $\phi: \SW \to \R_{>0}, \psi: \SB \to \R_{>0}$ if
\begin{equation}
    \forall~wb \in \SE,~\nu^2_{wb} = \phi(w)\psi(b)\nu^1_{wb}. 
\end{equation}
The dimer models are gauge equivalent with gauge functions $\phi: \SW \to \R_{>0}, \psi: \SB \to \R_{>0}$ if and only if their respective Kasteleyn matrix $K^1$, $K^2$ (with respect to the same fixed set of phases) satisfy 
\begin{equation}
    K^2 = \Phi K^1 \Psi,
\end{equation}
where $\Phi = D(\phi)$ and $\Psi = D(\psi)$. The Kasteleyn matrices are also said to be \emph{gauge equivalent} with gauge $(\Phi, \Psi)$. 
If $\SG$ is finite, gauge equivalence preserves the measure since for all $\SM \in \M(\SG)$, 
\be\label{eq:gauge:preserves:measure}
    \nu^2(\SM) = \prod_{e \in \SM}\nu^2_e = \prod_{w \in \SW}\phi(w) \prod_{b \in \SB}\psi(b) \prod_{e \in \SM}\nu^1_e = \left(\prod_{w \in \SW}\phi(w) \prod_{b \in \SB}\psi(b)\right) \nu^1(\SM).
\ee
where $\prod_{w \in \SW}\phi(w) \prod_{b \in \SB}\psi(b)$ is a constant which does not depend on $\SM$. In particular, gauge equivalence preserves the local statistics formula Equation~\ref{eq:local:statistics}.
\par

\paragraph{Dual and double graph of a planar graph.}\label{paragraph:dual}
For the rest of the section, let $(\SG,c,m)$ be \lr{either
\begin{itemize}
    \item an infinite (i.e. $|\SV| = \infty$) connected planar graph with mass and conductance functions as in \lr{Subsection}~\ref{par:infinite} 
    \item or a finite simply connected induced subgraph with wired boundary conditions of a planar graph with conductance and mass functions $(\oSG, \oc, \om)$  as in \lr{Subsection}~\ref{par:finite}.
\end{itemize}
Recall our definition of the boundary $\partial \SV$ of a finite subgraph $\SG$ of an infinite graph $\oSG$. Recall that when $\SG$ is infinite, we use the convention $\SG = \oSG$ so it is natural to set $\partial \SV = \emptyset$. Note that the boundary of a planar graph is often defined as “the vertices lying on the unbounded face”. This is not the definition that we use here, since we consider the boundary of $\SG$ as a subgraph of $\oSG$ (and not as a graph embedded in the plane).}

Recall from \lr{Subsection}~\ref{par:finite} the definition of $\Go$ obtained by identifying all vertices in $\oSV \setminus \SV$ to a unique vertex $o$, and recall that we abuse notation and also write $c$ for the conductance function on $\Go$. We already observed that since $\SG$ is simply connected, $\Go$ is planar. If $\SG$ is infinite, we use the convention $\Go = \SG$.\par
Let $\SGs = (\SVs, \SE^{\star})$ be the restricted dual graph of $\Go$ (in black on Figure \ref{fig:double:graph}). It is the dual graph of $\Go$ from which the dual vertex corresponding to the outer face and all the edges incident to it are removed. If $\SG$ is infinite, it is simply the dual graph of $\SG = \Go$. Observe that $\SG$ is the restricted dual of $\SGs$ so we are in the same situation as \cite{TreesMatchings} with primal and dual exchanged. If $\SG$ is finite, given $\rs \in \SGs$, every RST of $\Go$ rooted at $o$, $T \in \TT^{o}(\Go)$, has a \emph{dual directed spanning tree} rooted at $\rs$ which we denote by $T^{\star}$. It is the unique RST of $\SGs$ rooted at $\rs$ which does not intersect $T$.\par
\begin{figure}[!h]
    \centering
    \begin{overpic}[abs,scale=3.7]{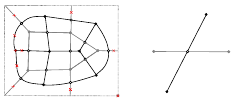}
        \put(216,5){\color{gray}$o$}
        \put(14,85){$\rs$}
        \put(263,85){\color{gray}$x$}
        \put(414,85){\color{gray}$y$}
        \put(360,157){$\x $}
        \put(290,8){$\y$}
        \put(337,77){$w$}
        \put(352,115){$\nu_{w\x } = 1$}
        \put(322,50){$\nu_{w\y} = 1$}
        \put(278,93){\color{gray}$\nu_{wx} = c_{(x,y)}$}
        \put(345,93){\color{gray}$\nu_{wy} = c_{(y,x)}$}
    \end{overpic}
    \caption{On the left: $\SG$ (in gray) is a finite simply connected induced subgraph of a planar graph. The additional vertex $o$ in the outer face is represented in a spread-out way by a dotted gray line. The restricted dual $\SGs$ is drawn in black. The white vertices of the double graph $\GD$ are drawn as white balls. The graph $\Gdr$ is obtained from $\GD$ by removing the vertices $o$ and $\rs$ and the edges incident to them (crossed in red). On the right: the weight function $\nu$ on the edges of $\Gdr$.}
    \label{fig:double:graph}
\end{figure}
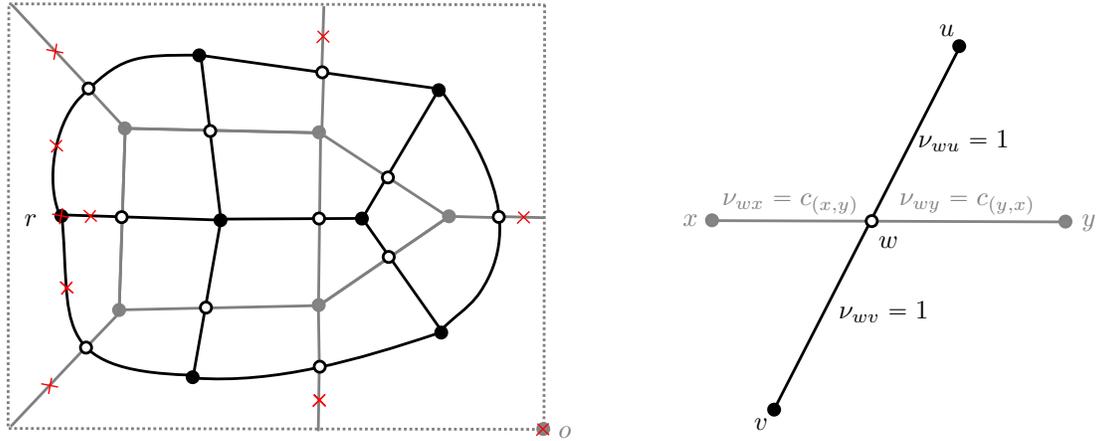
Embed the graphs $\SG^o$ and $\SGs$ in the plane so that the primal and dual edges intersect at a single point. The \emph{double graph} $\GD= (\VD, \ED)$ is constructed by superimposing the graphs $\SG^o$ and $\SGs$ and adding an extra vertex at each crossing of a primal and dual edge. It has two types of vertices: the original vertices are called \emph{black vertices} (drawn as gray and black bullets on Figure \ref{fig:double:graph}) and denoted by $\SB = \Vo \sqcup \SVs$ while the extra vertices are called \emph{white vertices} and denoted by $\SW$ (drawn as white bullets on Figure \ref{fig:double:graph}). The graph $\GD$ is bipartite: its edges correspond to half-edges of $\SG^o$ or $\SGs$ linking a white vertex with a black vertex.\par
Every (undirected) edge of the double graph $\GD$ corresponds to a directed edge of the primal graph $\Go$ or dual graph $\SGs$ (and conversely): if an edge $e = xw \in \ED$ links a primal vertex $x \in \Vo$ to a white vertex $w \in \SW$, since this white vertex corresponds to the intersection between a primal edge $xy \in \SE^o$ and a dual edge $\x \y  \in \SE^{\star}$, we say that the (undirected) half-edge $xw \in \ED$ corresponds to the (directed) edge $(x,y) \in \Eo$ of the primal graph. The same correspondence holds between edges $\x w \in \ED$ with $\x  \in \SVs, w \in \SW$ and directed edges of the dual graph $\SGs$.\par
For the rest of the section, when $\SG$ is finite, $\rs \in \partial \SGs$ denotes a vertex on the boundary of the dual graph (which is also planar). Denote by $\Gdr = (\Vdr,\Edr)$ the graph obtained from $\GD$ by removing $o$, $\rs$ and all the edges incident to them (on Figure \ref{fig:double:graph} it corresponds to removing everything crossed in red). Using our notation for induced subgraphs, $\Gdr = \GD_{\VD \setminus \{o,\rs\}}$. When $\SG$ is infinite, we use the convention $\Gdr = \GD$. The graph $\Gdr$ is also bipartite with white vertices $\SWr = \SW$ and black vertices $\SBr := \SV \sqcup \SVrs$ with $\SVrs = \SVs \setminus \{\rs\}$. Note that if $\SG$ is infinite, $\SV = \Vo$ and $\SVrs = \SVs$. When $\SG$ is finite, $\Gdr$ has the same number of black and white vertices: $|\SWr| = |\SBr|$ (by Euler's formula applied to the planar graph $\SGs$, which has $|\Vo|$ faces, $|\SVs|$ vertices and $|\SWr|$ edges).

\bigskip

Even if $\SG$ is not bipartite, its double graph $\GD$ is and there is a convenient way to choose phases on $\GD$. It goes as follows: around each white vertex, there are four vertices of $\GD$ alternating between vertices of $\SG$ and $\SGs$, say $x, \x , y, \y$ in the clockwise order. We attribute phases to edges 
\be\label{eq:phases:double:graph}
    \zeta_{wx}=1, \zeta_{w\x }=i, \zeta_{wy}=-1, \zeta_{w\y}=-i
\ee
(the choice of starting with $1$ on $x$ or $y$ is arbitrary and can be made independently for each $w \in \SW$). Such a choice of phases is represented on Figure \ref{fig:kasteleyn:phases}. 
\begin{figure}[!h]
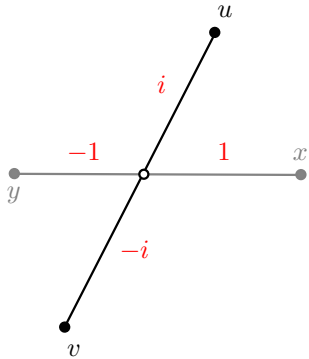

    \centering   
        \begin{overpic}[abs,unit=1mm,scale=3]{killed_dimers.pdf}
            \put(2,21){\color{gray}$y$}
		\put(40,26){\color{gray}$x$}
		\put(30,45){$\x $}
		\put(10,0){$\y$}
            \put(10,26){\color{red}$-1$}
            \put(30,26){\color{red}$1$}
            \put(22,35){\color{red}$i$}
            \put(17,13){\color{red}$-i$}
        \end{overpic}
        \caption{A choice of phases $\zeta$ on the double graph $\Gdr$ satisfying the Kasteleyn property (in red).}
        \label{fig:kasteleyn:phases}
\end{figure} 

These phases satisfy the Kasteleyn property: the alternating product around the faces is $-1$ (and all the inner faces of $\GD$ are bounded by four vertices). The restriction of these phases to $\Gdr$ also satisfies the Kasteleyn property since inner faces of $\Gdr$ are also inner faces of $\GD$. For the rest of this section, such a choice of phases is fixed.

\lr{Recall the definition of the dual directed spanning tree of a directed spanning tree $T \in \TT^o(\Go)$ from Subsection \ref{SubsecTemperley} just after the definition of the restricted dual graph.}
\paragraph{Temperley's bijection.} We define weights on the edges of $\Gdr$ (see Figure \ref{fig:double:graph}). 
\be\label{def:dimer:weight}
    \forall~wb \in \Edr,~\nu_{wb} = 
    \left\{
    \begin{array}{ll}
        1 & \text{ if } b = \x  \in \SVrs\\
        c_{(x,y)} & \text{ if }b = x \in \SV \text{ and } xw \in \Edr \text{ corresponds to } (x,y) \in \Eod
    \end{array}
    \right.
    .
\ee
\emph{Temperley's bijection} is a weight-preserving correspondence between dimer configurations and RST of a finite planar graph. Denote by $\nRST^{o}$ the weight function on $\TT^{o}(\Go)$ associated to $c$ by Equation~\eqref{eq:def:weight:tree} and by $\nd$ the weight function on $\M(\Gdr)$ associated to $\nu$ by \eqref{eq:def:weight:dimer}.

\begin{Thm}[Temperley's bijection, Theorem 1 of \cite{TreesMatchings}]\label{thm:Temperley}
    Assume that $\SG$ is finite. Dimer configurations of $\Gdr$ are in weight-preserving bijection with RST of $\Go$ rooted at $o$. More precisely, for every $T \in \TT^{o}(\Go)$ (see \ref{paragraph:dual}), the corresponding dimer configuration $\SM \in \M(\Gdr)$ is obtained by replacing each directed edge of $T$ and of its dual directed spanning tree $T^{\star}$ by the corresponding (undirected) edge of $\Gdr$, and satisfies
    \begin{equation}
        \nRST^o(T) = \nd(\SM).
    \end{equation}
\end{Thm} 
Pictures of Temperley's bijection in this context can be found in Figure 1 of \cite{TreesMatchings} or Figure 2 of \cite{ZinvDirac}. Recall the notation $\Pd$ for the probability measure on $\M(\Gdr)$ associated with $\nu$ by Equation~\eqref{eq:def:proba:dimer} and denote by $\PRST^{o}$ the probability measure on RST of $\Go$ associated with the conductance function $c$ by Equation~\eqref{def:tree}.
\begin{Cor}\label{cor:consequence:Temperley}
    Assume that $\SG$ is finite. The partition function of the dimer and tree models are equal:
    \begin{equation}
    \ZRST^{o}(\Go, c) = \Zd(\Gdr,\nu).
    \end{equation}
    Moreover, for all $\ed_1, \dots, \ed_k \in \Eod$, if $f_1,\dots, f_k$ are the corresponding undirected edges of $\Gdr$:
    \begin{equation}
        \PRST^{o}(\ed_1,\dots,\ed_k) = \Pd(f_1,\dots,f_k).
    \end{equation}
\end{Cor}

\subsection{The drifted and killed dimer models}\label{subsec:killed:drifted}
The goal of this Section is to combine Theorem~\ref{thm:doob} with Temperley's bijection to relate the RSF model on $\SG$ with a dimer model. We introduce two families of dimer models and show how they are related to the RSF model. For the rest of this section, we assume that $c$ is a \emph{symmetric} conductance function (the symmetry is necessary here unlike in the rest of the paper).\par
Let $\lambda: \oSV \to \R_{>0}$, not necessarily massive harmonic for the moment. Recall that when $\SG$ is finite, it is an induced subgraph of a larger graph $\oSG$. Recall the definition of the conductance function associated with the Doob transform $\forall (x,y) \in \SE,~\cl_{(x,y)} = \frac{\lambda(y)}{\lambda(x)}c_{xy}$. The conductance function $\cl$ is not in general symmetric, even though $\co$ is. Recall that when $\SG$ is finite, we abuse notation and also denote by $\cl$ the conductance function on $\Go$ obtained from $\ocl$ by identifying vertices of $\oSV \setminus \SV$ to $o$.\par
We introduce two dimer models on $\Gdr$ (local weights are pictured in Figure \ref{fig:killed:dimers}): they are a generalization on any graph with any symmetric conductance function of the “drifted” and “killed” dimer models of \cite{Chhita} (which were defined on $\Z^2$ for specific conductance functions) and of the “killed” dimer model defined in Section 3.1 of \cite{ZinvDirac} underlying the $Z^u$-Dirac operator (which was defined on an isoradial graph, with specific elliptic weights). The local weights associated with these dimer models are pictured on Figure \ref{fig:killed:dimers}.
\begin{Def}[The drifted dimer model]\label{def:dimer:drifted}
Denote by $\nud$ the weight function on the edges of $\Gdr$ associated to $\clo$ by Equation~\eqref{def:dimer:weight}:
\be\label{def:dimer:weight:drifted:finite}
    \forall~ wb \in \Edr,~\nud_{wb} = 
    \left\{
    \begin{array}{ll}
        1 & \text{ if } b = \x  \in \SVrs\\
        \clo_{(x,y)}  & \text{ if }b = x \in \SV \text{ and } xw \in \Edr \text{ corresponds to } (x,y) \in \Eod.
    \end{array}
    \right.
\ee
(see Figure \ref{fig:killed:dimers}).
When $\SG$ is infinite, $\Gdr = \GD$, $\Eod = \Ed$ and $\forall (x,y) \in \SE,~\clo_{(x,y)} = \frac{\lambda(y)}{\lambda(x)}c_{xy}$. The \emph{drifted dimer model} associated with $\lambda$ is the dimer model on $\Gdr$ with weights $\nud$.
\end{Def}
Let $\ls: \SVs \to \R_{>0}$ be an arbitrary positive function on the dual. 
\begin{Def}[The killed dimer model]\label{def:killed:dimer}
We define a weight function $\nuk$ associated to $\lambda, \ls$ on the edges of $\Gdr$ (see Figure \ref{fig:killed:dimers}). If $w \in \SWr$ corresponds to the intersection of the primal and dual edges $xy \in \Eo$ (with $x \neq o$) and $\x \y  \in \SE^{\star}$ (with $\x  \neq r$), let
\be\label{def:dimer:weight:killed}
    \nuk_{wb} = 
    \left\{
    \begin{array}{ll}
        \ls(\x )^{-1}\big(c_{xy}\lambda(x)\lambda(y)\big)^{-\frac{1}{2}} & \text{ if } b = \x \\
        \big(c_{xy}\lambda(y)\big)^{\frac{1}{2}}\lambda(x)^{-\frac{1}{2}}  & \text{ if }b = x
    \end{array}
    \right.
\ee
(see Figure \ref{fig:killed:dimers}) \lr{with the slight abuse of notation that when $y = o$, $\lambda(y)$ should be interpreted as $\lambda(z)$ where $z \in \oSV \setminus \SV$ is the only vertex such that $w$ is the middle of the edge $xz$}. Recall that when $\SG$ is infinite, $\Gdr = \GD$ and $\Eo = \SE$. The dimer model on $\Gdr$ with weight function $\nuk$ is called the \emph{killed dimer model}.
\end{Def}
We denote by $\Kd$ and $\Kk$ the Kasteleyn matrices of the drifted and killed dimer models, with rows indexed by $\SWr$ and columns by $\SBr$ (associated with the fixed phases $\zeta$). If $\SG$ is finite we denote by $\Pdd$ and $\Pdk$ the associated probability measure on $\M(\Gdr)$ (see Equation~\eqref{eq:def:proba:dimer}).
\begin{Prop}\label{prop:gauge:killed:drifted}
    \lr{For any choice of $\lambda: \oSV \to \R_{>0}$ and $\ls: \SVs \to \R_{>0}$,} the drifted and killed dimer models are gauge equivalent.
\end{Prop}
\begin{proof}
    Define a gauge function $\psi$ on $\SBr$ coinciding with $\lambda$ on $\SV$ and $(\ls)^{-1}$ on $\SVrs$. We also define a function $\phi$ on the white vertices. For $w$ corresponding to the primal edge $xy$ (we can assume that $x \neq o$ upon exchanging $x$ and $y$) let $\phi(w) = \big(c_{xy}\lambda(x)\lambda(y)\big)^{-\frac{1}{2}}$. It is left to the reader to check that for all $w \in \SWr$, $b \in \SBr$
\be\label{eq:gauge:killed:drifted}
	(\Kk)_{w,b} = \phi(w)\psi(b)\Kd_{w,b}.
\ee 
\end{proof}
We first show that when $\lambda$ is massive harmonic on $\SV$, a consequence of the Doob transform technique is that the partition function of the drifted dimer model is equal to the partition function of the RSF model. It is a direct corollary of Temperley's bijection (Corollary~\ref{cor:consequence:Temperley}) and of Theorem~\ref{thm:doob}:
\begin{Cor}[Temperley's bijection at the level of partition functions]\label{cor:temperley:drifted}
    Assume that $\SG$ is finite and $\lambda$ is massive harmonic for $\oDk$ on $\SV$. The RST model on $\Go$ with conductance function $\clo$ and the drifted dimer model are in weight-preserving bijection. In particular:
    \begin{equation} 
        \ZRSF(\SG,c,m) = \ZRST^{o}(\SG^o, \clo) = \Zd(\Gdr, \nud)
    \end{equation}
\end{Cor}
\begin{proof}
    Temperley's bijection holds by definition of the drifted dimer model. The first equality is stated in Theorem~\ref{thm:doob}. The last one is Corollary~\ref{cor:consequence:Temperley}.
\end{proof}
This result can be seen as a partial extension of Temperley's bijection since it relates RSF of $\SG$ with dimer configurations on the double graph $\Gdr$. Our equality only holds for partition function, unlike Temperley's bijection which is a one-to-one bijection. This result has no direct extension for an infinite planar graph $\SG$ since the partition function is not well-defined, but an analogous result for infinite $\Z^2$-periodic graphs is given in Section~\ref{subsec:Temperley:per}, relating the free energy (which can be thought of as the infinite equivalent of the partition function) of a $\Z^2$-periodic forest model with the free energy of a $\Z^2$-periodic dimer model.

\bigskip

When $\lambda$ is massive harmonic, the killed dimer model is also related to the RSF model via the massive Laplacian $\Dk$. To state this result, we need a bit of notation. The \emph{\lr{H}ermitian conjugate} $M^{\dag}$ of a matrix $M$ is the conjugate of the transposed matrix. If a square matrix $M$ has rows and columns indexed by $\SBr$, it is \emph{block diagonal} if all its coefficients outside of the blocks $\SV \times \SV$ and $\SVrs \times \SVrs$ are $0$. In which case, we write
\begin{equation}
    M = \begin{pmatrix} \Asterisk &0 \\ 0& \Asterisk\end{pmatrix}.
\end{equation}
A natural \emph{dual conductance function} $\cs$ and \emph{modified dual conductance function}  and $\cls$ are defined on $\SE^{\star}$. Each edge $\x \y$ of $\SGs$ intersects a unique edge $xy$ of $\Go$. The dual conductance function is then $\cs_{\x \y} = \big(c_{xy}\big)^{-1}$ and the dual modified conductance function is then $\cls_{\x \y} = \big(c_{xy}\lambda(x)\lambda(y)\big)^{-1}$. \lr{Recall that we abuse notation and denote by $c$ the conductance function on $\Go$, so $c_{xy}$ is well defined even if $xy = zo$ for some $z \in \SV$, in which case it is equal to the sum of the conductances $\oc_{zw}$ of edges $zw$ with $w \in \oSV \setminus \SV$.} 

\lr{The conductance functions $\cs$ and $\cls$ are both symmetric.} The Laplacian operator on $\SGs$ associated with the conductance function $\cs$ is denoted by $\Ds$. We also define a matrix $\Dsl$ with rows and columns indexed by $\SVrs$ and coefficients
    \be\label{eq:def:dual:laplacian}
        \forall \x , \y  \in \SVrs,~ \Dsl_{\x ,\y} = 
        \left\{
        \begin{array}{ll}
            -\cls_{\x \y}\ls(\x )^{-1}\ls(\y)^{-1} & \text{ if }\x  \overset{\SGs}{\sim} \y,~\x  \neq \y\\
             \sum_{\z \overset{\SGs}{\sim} \x , \z \neq \x } \cls_{\x \z }\ls(\x )^{-2} & \text{ if } u=v\\
             0& \text{ otherwise}
        \end{array}
        .
        \right.
    \ee
We will comment on this definition after the statement of Proposition~\ref{prop:massive:holomorphy}.
\begin{Prop}\label{prop:massive:holomorphy}
    The matrix $(\Kk)^{\dag}\Kk$ is block diagonal. Moreover, if $\lambda$ is massive harmonic for $\oDk$ on $\SV$,
    \be\label{eq:massive:holomorphy}
	(\Kk)^{\dag}\Kk = \begin{pmatrix} \Dk_{\SV}&0 \\ 0&\Dsl \end{pmatrix}.
    \ee 
    When $\SG$ is infinite, reciprocally, \lr{let} $\nu: \Edr \to \R_{>0}$ \lr{be} a weight function and $K$ denotes the associated Kasteleyn matrix (with the set of phases $\zeta$ \lr{fixed in \eqref{eq:phases:double:graph}}), if 
    \be\label{eq:massive:holomorphy:reciprocal}
	K^{\dag}K = \begin{pmatrix} \Dk_{\SV}&0 \\ 0& \Asterisk \end{pmatrix},
    \ee
    \lr{then} there exist $\lambda:\SV \to \R_{>0}$ massive harmonic for $\Dk$ and $\ls: \SVs \to \R_{>0}$ such that $\nu$ is the weight function of the killed dimer model associated with $\lambda, \ls$.
\end{Prop}
We first make a few observations on the definitions and statements, then state some consequences, and finally provide a proof of this proposition at the end of the section. 
\begin{itemize}
    \item Recall that when $\SG$ is infinite, $\Dk_{\SV} = \Dk$.
    \item When $\SG$ is finite, $\rs \in \SVs$ acts as a cemetery state in the definition of $\Dsl$ since for $\x  \overset{\SGs}{\sim} \rs$, $\rs$ contributes to the diagonal coefficient $\Dsl_{\x , \x }$ even though there is no extra diagonal coefficient $\Dsl_{\x, \rs}$ since the matrix is indexed by $\SVrs = \SVs \setminus \{\rs\}$.
    \item The value of $\ls$ at $\x $ is not used in the construction of the killed dimer model, but it will be useful that $\ls$ is defined on all $\SVs$ when we discuss duality.
    \item The matrix $\Dsl$ is not in general the (massive) Laplacian operator associated with a conductance (and mass) function. In the infinite case though, for the choice $\ls = 1$, $\Dsl$ is the (non massive) Laplacian operator on $\SGs$ associated with the symmetric conductance function $\cls$. In the finite case, for the choice $\ls=1$, $\Dsl$ is the massive Laplacian associated with the conductance function $\cls$ restricted to $\SVrs = \SVs \setminus \{\rs\}$. It is associated with a RST model on $\SGs$ rooted at $\rs$.
\end{itemize}
We state some consequences of Proposition~\ref{prop:massive:holomorphy}. When $\SG$ is finite this Proposition gives an alternative proof of our Theorem~\ref{thm:doob} for symmetric conductance functions. We explain why. 
\begin{proof}[Alternative proof of Theorem~\ref{thm:doob} when $c$ is symmetric.]
    Denote by $\PdRST$ the probability measure on RST of $\Go$ rooted at $o$ associated with the conductance function $\clo$. By Temperley's bijection Theorem~\ref{thm:Temperley} and the local statistics formula \eqref{eq:local:statistics}, for all $\ed_1 = (x_1,y_1), \dots, \ed_k = (x_k,y_k) \in \Eod$ with $x_1,\dots, x_k \in \SV$, if $f_1 = x_1w_1, \dots f_k = x_kw_k \in \Edr$ are the corresponding edges of the double graph, 
	\be\label{eq:tree:local:statistics}
		\begin{aligned}
			\PdRST(\ed_1,\dots,\ed_k) 
            = \Pldim(f_1,\dots,f_k)
			&= \lr{\prod_{i=1}^k \Kd_{w_i,x_i} \det \big(((\Kd)^{-1})_{x_i,w_j}\big)_{1\leq i,j \leq k}}\\
			&=  \lr{\prod_{i=1}^k \Kk_{w_i,x_i} \det \big(((\Kk)^{-1})_{x_i,w_j}\big)_{1\leq i,j \leq k}}
		\end{aligned}
	\ee
    because gauge equivalence preserves the local statistics formula (see Equation~\eqref{eq:gauge:preserves:measure}) and by Proposition~\ref{prop:massive:holomorphy} the drifted and killed dimer models are gauge equivalent. The inverse Kasteleyn matrix $(\Kk)^{-1}$ can be computed using Equation~\eqref{eq:massive:holomorphy}: for every $1 \leq i,j \leq k$:
	\begin{equation}
		\begin{aligned}
			((\Kk)^{-1})_{x_i,w_j} 
			\overset{\eqref{eq:massive:holomorphy}}{=} \left((\Dk_{\lr{\SV}})^{-1}(\Kk)^{\dag}\right)_{x_i,w_j}
			& \overset{\eqref{eq:green:inverse:laplacian}}{=}  \left(\Gk(\Kk)^{\dag}\right)_{x_i,w_j}\\
            &= \Gk_{x_i,x_j}\overline{\Kk_{w_j,x_j}} + \Gk_{x_i,y_j}\overline{\Kk_{w_j,y_j}}\\
                &= \overline{\zeta_{w_jx_j}}\left(\nuk_{w_jx_j}\Gk_{x_i,x_j} - \nuk_{w_jy_j}\Gk_{x_i,y_j} \right)\\
                &\overset{\eqref{def:dimer:weight:killed}}{=} \frac{\overline{\zeta_{w_jx_j}}\big(c_{(x_j,y_j)}\lambda(x_j)\lambda(y_j)\big)^{\frac{1}{2}}}{\lambda(x_i)}\left(\frac{\lambda(x_i)\Gk_{x_i,x_j}}{\lambda(x_j)} - \frac{\lambda(x_i)\Gk_{x_i,y_j}}{\lambda(y_j)}\right).
            \end{aligned}
	\end{equation}
 	Observe that since the conductance function $c$ is symmetric, $(\Dk_{\lr{\SV}})^T = \Dk_{\lr{\SV}}$ and hence $(\Gk)^T = \Gk$. This also holds for the extension of $\Gk$ to $\Vo$ since we defined by convention $\Gk_{x,o} = \Gk_{o,y} = 0$ for any $x, y \in \Vo$. Using this symmetry, we can write
    \be\label{eq:K-1}
        \begin{aligned}
        ((\Kk)^{-1})_{x_i,w_j} 
        &= \frac{\overline{\zeta_{w_jx_j}}\big(c_{x_jy_j}\lambda(x_j)\lambda(y_j)\big)^{\frac{1}{2}}}{\lambda(x_i)}\left(\frac{\lambda(x_i)\Gk_{x_j,x_i}}{\lambda(x_j)} - \frac{\lambda(x_i)\Gk_{y_j,x_i}}{\lambda(y_j)}\right)\\
        &= \frac{\overline{\zeta_{w_jx_j}}\big(c_{x_jy_j}\lambda(x_j)\lambda(y_j)\big)^{\frac{1}{2}}}{\lambda(x_i)} \Hlk_{\ed_j, \ed_i}.
        \end{aligned}
    \ee
     Inserting Equation~\eqref{eq:K-1} in Equation~\eqref{eq:tree:local:statistics} and using the multilinearity of the determinant gives
    \begin{equation}
        \begin{aligned}
            \PdRST(\ed_1,\dots,\ed_k) 
            &= \prod_{i=1}^k \nuk_{w_ix_i} \prod_{j=1}^k \big(c_{x_jy_j}\lambda(x_j)\lambda(y_j)\big)^{\frac{1}{2}}\prod_{i=1}^k \lambda(x_i)^{-1} \det (\Hlk_{\ed_j,\ed_i})_{1\leq i,j \leq k}\\
            &= \prod_{i=1}^k \clo_{x_i, x_j} \det (\Hlk_{\ed_j,\ed_i})_{1\leq i,j \leq k}.
        \end{aligned}
    \end{equation}
    which is precisely the statement of Theorem~\ref{thm:doob}.
\end{proof}
When $\SG$ is finite, Proposition~\ref{prop:massive:holomorphy} implies the equality of determinants $\lr{|}\det \Kk\lr{|}^2 = \det \Dk_{\lr{\SV}} \det \Dsl$. We can actually prove something stronger. \lr{For $x \in \SV$, let $\deg_{\oSG}(x)$ denote the degree of $x$ in the graph $\oSG$.} 
\begin{Prop}
    Assume that $\SG$ is finite and $\lambda$ is massive harmonic for $\oDk$ on $\SV$. There exists an explicit \lr{positive} constant 
    \lr{
    $$
        C(\lr{c,}\lambda, \ls) := \prod_{e \in \Eo}c_e^{-\frac{1}{2}} \prod_{x \in \SV}\lambda(x)^{1-\frac{1}{2}\deg_{\oSG}(x)}\prod_{x \sim y, x \in \SV, y \notin \SV}\lambda(y)^{-\frac{1}{2}}\prod_{\x \in \SVrs} \ls(\x )
    $$
    }
    such that
    \be\label{eq:det:K}
        \det \Kk = C(\lr{c,}\lambda,\ls)\det \Dk_{\SV} =  C(\lr{c,}\lambda,\ls)^{-1}\det \Dsl.
    \ee
\end{Prop}
Observe that by the Kasteleyn formula Equation~\eqref{eq:Kasteleyn:formula} and the matrix-forest theorem, this proposition gives an equality up to an explicit constant $C(\lr{c,}\lambda,\ls)$ between the partition function of the killed dimer model and the partition function of the RSF model associated with $\Dk_{\SV}$. This result is an extension of Theorem 21 of \cite{ZinvDirac} in the case of a general finite graph (instead of an isoradial graph), and the proof is the same. 
\begin{proof}
    For $w \in \SWr$, we denote by $x_w y_w \in \Eo$ with $x_w \neq o$ the associated primal edge. Using successively the gauge equivalence between the drifted and killed Kasteleyn matrices obtained in Equation~\eqref{eq:gauge:killed:drifted}, the multilinearity of the determinant, Temperley's bijection between directed trees and drifted dimers, the matrix-forest theorem and finally the gauge equivalence between $\Dk_{\SV}$ and $\Dl_{\SV}$, we obtain that
    \[
	\begin{aligned}
		\det \Kk &\overset{\eqref{eq:gauge:killed:drifted}}{=}  \left\{ \prod_{w \in \SWr} \big(c_{x_wy_w}\lambda(x_w)\lambda(y_w)\big)^{-\frac{1}{2}} \prod_{x \in \SV} \lambda(x) \prod_{\x \in \SVrs} \ls(\x )\right\} \det \Kd\\
		&\overset{\eqref{eq:Kasteleyn:formula}}{=} \lr{\left\{\prod_{e \in \Eo}c_e^{-\frac{1}{2}} \prod_{x \in \SV}\lambda(x)^{1-\frac{1}{2}\deg_{\oSG}(x)}\prod_{x \sim y, x \in \SV, y \notin \SV}\lambda(y)^{-\frac{1}{2}}\prod_{\x \in \SVrs} \ls(\x ) \right\}} \Zd(\Gdr, \nud)\\
        &\overset{\eqref{cor:consequence:Temperley}}{=} C(\lr{c,}\lambda,\ls) \ZRST^{o}(\Go, \clo) \overset{\eqref{eq:matrix-forest}}{=} C(\lr{c,}\lambda,\ls) \det \Dlk_{\SV} \overset{\eqref{eq:gauge:doob:finite}}{=} C(\lr{c,}\lambda,\ls) \det \Dk_{\SV}.\qedhere
	\end{aligned}
    \]
\end{proof}

We now prove Proposition~\ref{prop:massive:holomorphy}
\begin{proof}
We first check that  $(\Kk)^{\dag}\Kk$ is block diagonal. Let $x \in \SV, \x  \in \SVrs$. If $\x $ is not adjacent to $x$ the coefficient $((\Kk)^{\dag}\Kk)_{x,\x }$ is zero because $\Kk$ is non zero only on edges of $\Gdr$ so for all $w \in \SWr$ either $(\Kk)^{\dag}_{x,w} =0$ or $\Kk_{w,x} = 0$. If $x$ is a corner of the face $\x $ in $\SG$, $x$ and $\x $ are the opposite vertices of a face $xw_1\x w_2$ (with vertices in the clockwise order) of $\Gdr$ and 
\begin{equation}
((\Kk)^{\dag}\Kk)_{x,\x } = (\Kk)^{\dag}_{x,w_1}\Kk_{w_1,\x } + (\Kk)^{\dag}_{x,w_2}\Kk_{w_2,\x } = \sum_{i =1,2}\nu_{w_ix}\nu_{w_i\x }\overline{\zeta_{w_ix}}\zeta_{w_i\x }.
\end{equation}
On the one hand, by definition of $\nuk$,
\begin{equation}
    \nuk_{xw_1}\nuk_{\x w_1} = \nuk_{xw_2}\nuk_{\x w_2} = (\ls(\x )\lambda(x))^{-1}.
\end{equation}
On the other hand, the Kasteleyn property of the phases around the face $xw_1\x w_2$ implies that
\begin{equation}
    \overline{\zeta_{w_1x}}\zeta_{w_1\x } = -\overline{\zeta_{w_2x}}\zeta_{w_2\x }
\end{equation}
and hence
\be\label{eq:proof:extra:diagonal:zero}
	((\Kk)^{\dag}\Kk)_{x,\x } = 0.
\ee
The same holds for $((\Kk)^{\dag}\Kk)_{\x x}$ since 
\begin{equation}
    ((\Kk)^{\dag}\Kk)_{\x ,x} = \overline{((\Kk)^{\dag}\Kk)^{\dag}_{x,\x }} = \overline{((\Kk)^{\dag}\Kk)_{x,\x }} = 0.
\end{equation}
Assume further that $\lambda$ is massive harmonic for $\oDk$ on $\SV$. We compute the $\SVrs \times \SVrs$ block of $(\Kk)^{\dag}\Kk$.
For $\x ,\y  \in \SVrs$, $((\Kk)^{\dag}\Kk)_{\x, \y}$ can be non zero only if $\x  \overset{\SGs}{\sim} \y$ or $\x =\y$ (since otherwise for all $w \in \SWr$ either $\Kk_{w,\x }=0$ or $\Kk_{w,\y}=0$). For $\x  \overset{\SGs}{\sim} \y$ with $\x  \neq \y$, if $w$ is the white vertex of $\Gdr$ corresponding to the edge $\x \y  \in \SE^{\star}$ 
\begin{equation}
	((\Kk)^{\dag}\Kk)_{u,v} = (\Kk)^{\dag}_{u,w}\Kk_{w,v}=-\cls_{\x \y}(\ls(\x )\ls(\y))^{-1} = \Dsl_{\x, \y}.
\end{equation}
For $\x  \in \SVrs$, if we denote by $u_1,\dots,u_k$ the neighbours of $\x $ in $\SGs$ (including $\rs$) and let $w_1, \dots, w_k \in \SWr$ such that for all $1 \leq i \leq k$ the directed edge $(\x ,u_k)$ of $\SGs$ corresponds to the undirected edge $\x w_k$ of $\Gdr$,
\begin{equation}
	((\Kk)^{\dag}\Kk)_{\x, \x } = \sum_{i=1}^k \cls_{\x u_k}\ls(\x )^{-2} = \Dsl_{\x, \x }.
\end{equation}
We identify the $\SV \times \SV$ block of $(\Kk)^{\dag}\Kk$. For $x,y \in \SV$, $((\Kk)^{\dag}\Kk)_{x,y}$ can be non zero only if $x \sim y$ or $x=y$ (since otherwise for all $w \in \SWr$ either $\Kk_{w,x}=0$ or $\Kk_{w,y}=0$). For $x \sim y$ with $y \neq x$, if $w \in \SWr$ is the white vertex corresponding to the edge $xy \in \SE$ 
\begin{equation}
	((\Kk)^{\dag}\Kk)_{x,y} = (\Kk)^{\dag}_{x,w}\Kk_{w,y}=-c_{xy} = \Dk_{x,y}.
\end{equation}
For $x \in \SV$, if we denote by $y_1,\dots,y_k$ the neighbours of $x$ in $\Go$ and let $w_1, \dots, w_k \in \SWr$ such that for all $1 \leq i \leq k$ the directed edge $(x,y_k)$ of $\Go$ corresponds to the undirected edge $xw_k$ of $\Gdr$,
\be\label{eq:KsK:xx}
	((\Kk)^{\dag}\Kk)_{x,x} = \sum_{i=1}^k c_{xy_i}\frac{\lambda(y_i)}{\lambda(x)}.
\ee
This is equal to $\Dk_{x,x}$ if and only if $\lambda$ is massive for $\oDk$ at $x$, which concludes the first part of the proof.\par
We now assume that $\SG$ is infinite so the notation simplify, in particular $\SB = \SBr$, $\SW = \SWr$ and $\SVrs = \SVs$. We prove the reciprocal statement. Let $\nu: \Ed \to \R_{>0}$ be any weight function such that the associated Kasteleyn matrix $K$ satisfies Equation~\eqref{eq:massive:holomorphy:reciprocal}. In particular, for any $x \in \SV,~\x  \in \SVs$, $(K^{\dag}K)_{x,\x }=0$. Let $xw_1\x w_2$ (with vertices in the clockwise order) be any face of $\GD$. Since $K^{\dag}K$ vanishes outside the diagonal blocks,
    \be\label{eq:alternating:K}
        0 = (K^{\dag}K)_{x,\x } = \overline{K_{w_1,x}}K_{w_1,\x }+\overline{K_{w_2,x}}K_{w_2,\x }.
    \ee
    We define the auxiliary matrix $L$ with rows and columns indexed by $\SW$ and $\SB$:
    \be\label{eq:def:auxiliary}
        \forall w \in \SW,~b \in \SB,~L_{w,b} = \left\{
        \begin{array}{ll}
             \overline{K_{w,x}}& \text{ if }b=x \in \SV\\
             \big(K_{w,\x })^{-1}& \text{ if }b=\x \in \SVs.
        \end{array}
        \right.
    \ee
    Equation~\eqref{eq:alternating:K} implies that $K$ has the alternating product property: around each face $xw_1\x w_2$ of $\SGd$ (with vertices in the clockwise order),
    \be\label{eq:alternating:L}
        \frac{L_{w_1,x}L_{w_2,\x }}{L_{w_1,\x} L_{w_2,x}}=-1.
    \ee
    Let $\Kko$ and $\nuko$ be the Kasteleyn matrix and weight function of the killed dimer model with the particular choice $\lambda(x) = 1$ for all $x \in \SV$ (which is not massive harmonic) and $\ls(\x ) = 1$ for all $x \in \SVs$ and $\Lko$ the auxiliary matrix associated to $\Kko$ by Equation~\eqref{eq:def:auxiliary}. For $w \in \SW$ the intersection of the primal and dual edges $xy$ and $uv$, $\Kko_{w,x} = \zeta_{w,x}c_{xy}^{1/2}$ and $\Kko_{w,u} = \zeta_{w,u}c_{xy}^{-1/2}$. By the first statement of Proposition \ref{prop:massive:holomorphy} which we already proved, $\Kko$ vanishes outside the diagonal blocks so $\Lko$ satisfies the alternating product property Equation~\eqref{eq:alternating:L}. It is classical (see for example Lemma 72 in the Appendix of \cite{ZinvDirac}) that for two matrices with non-zero coefficients on the edges of a bipartite planar graph, equality of the alternating products around every inner face implies gauge equivalence. Hence $L$ and $\Lko$ are gauge equivalent and there exist gauge functions $\phi: \SW \to \C \setminus \{0\}$ and $\psi: \SB \to \C \setminus \{0\}$ such that
    \begin{equation}
        \forall w \in \SW,~\forall b \in \SB,~L_{w,b} = \phi(w)\psi(b)\Lko_{w,b}.
    \end{equation}
    The gauge functions $\phi, \psi$ are real and positive since they are written in \cite{ZinvDirac} as product and quotient of ratios of the form $\frac{L_{w,b}}{\Lko_{w,b}} = \left(\frac{\nu_{w,b}}{\nuko_{w,b}}\right)^{\pm 1} \in \R_{>0}$. Let $\lambda = \psi_{|\SV}$, $\ls = \psi_{|\SVs}$ and because the Kasteleyn phases are the same. We can write for all $w \in \SW,~b \in \SB$,
    \be\label{eq:K:extra:diagonal}
        K_{w,b} = 
        \left\{
        \begin{array}{llll}
             \overline{L_{w,x}} &= \overline{\lambda(x)\phi(w)\Lko_{w,x}} &= \lambda(x)\phi(w)\Kko_{w,x} &\text{ if }b = x \in \SV \\
             \big(L_{w,\x }\big)^{-1} &= \left(\ls(\x )\phi(w)\Lko_{w\x }\right)^{-1} &= \left(\ls(\x )\phi(w)\right)^{-1}\Kko_{w\x }& \text{ if }b = \x  \in \SVs.
        \end{array}
        \right.
    \ee
    We can now identify the gauge function $\phi$. Since Equation~\eqref{eq:massive:holomorphy:reciprocal} holds, for all $w \in \SW$ corresponding to the primal edge $xy \in \SE$:
	\begin{equation}
		\Dk_{x,y} = -c_{xy} = (K^{\dag}K)_{x,y} = \overline{K_{w,x}}K_{w,y} = \lambda(x)\lambda(y)\phi(w)^2 \overline{\Kko_{w,x}}\Kko_{w,y} = -\lambda(x)\lambda(y)\phi(w)^2c_{xy}. 
	\end{equation}
	and hence
	\begin{equation}
		\phi(w)^2 = \big(\lambda(x)\lambda(y)\big)^{-1}.
	\end{equation}
    Combined with Equation~\eqref{eq:K:extra:diagonal}, it implies that the weight function $\nu$ coincides with the weight function of the killed dimer model associated with the positive functions $\lambda, \ls$. It remains to prove that the function $\lambda$ is massive harmonic: for $x \in \SV$, Equation~\eqref{eq:massive:holomorphy:reciprocal} applied at $x$ implies massive harmonicity of $\lambda$ at $x$ by \eqref{eq:KsK:xx}.
\end{proof}

\subsection{Massive holomorphy and self-duality}\label{subsec:massive:holomorphy}
For the clarity of exposition, we assume in this subsection that $\SG$ is infinite. Similar statements could be given in the finite case, but the self-duality would be less clear since $\Go$ with distinguished vertex $o$ and $\SGs$ with distinguished vertex $r$ do not play perfectly dual roles.\par
Proposition~\ref{prop:massive:holomorphy} can be seen as a generalization of the theory of discrete massive holomorphy on isoradial graphs developed in \cite{Kenyon_2002} for the non-massive case and in \cite{MakarovSmirnov}, \cite{ZinvDirac}, \cite{FermionicObservable} (among others) for the massive case. Equation~\eqref{eq:massive:holomorphy} implies that the Kasteleyn matrix $\Kk$ of the killed dimer model is a discrete massive Dirac operator associated with the massive Laplacian $\Dk$. In other words, a function $f: \SB \to \C$ satisfying $\Kk f = 0$ can be decomposed into \lr{a \emph{real}} \emph{part} $f_0 : \SV \to \C$ harmonic for $\Dk$ and an \emph{imaginary part} $f_1: \SVrs \to \C$ harmonic for $\Dsl$.

When $\Kk$ is invertible, $\Dsl$ also is and we can define the \emph{dual Green function} $\Gsl = (\Dsl)^{-1}$. Similarly to the papers in the isoradial setting, the inverse Dirac operator $(\Kk)^{-1}$ can be expressed in terms of “discrete derivatives” of the (massive) Green function and the dual Green function: Equation~\eqref{eq:massive:holomorphy} implies that 
\begin{equation}
    (\Kk)^{-1} = \begin{pmatrix} \Gk&0 \\ 0&\Gsl \end{pmatrix} (\Kk)^{\dag}.
\end{equation}
This means that for all $b \in \SB$, $w \in \SW$, if we denote by $x_w, y_w, \x _w, \y_w$ the neighbours of $w$ in $\GD$ in clockwise order,
\begin{equation}
    ((\Kk)^{-1})_{b,w} = \overline{\zeta_{wx_w}}
    \left\{
    \begin{array}{ll}
         \nuk_{wx_w}\Gk_{x,x_w}  - \nuk_{wy_w}\Gk_{x,y_w}  &\text{ if }b=x \in \SV  \\ 
         i(\nuk_{w\x _w} \Gsl_{\x ,\x _w} - \nuk_{w\y_w} \Gsl_{\x ,\y_w})&\text{ if }b = \x  \in \SVs
    \end{array}
    \right.
\end{equation}
In both papers \cite{Kenyon_2002} and \cite{ZinvDirac}, the (massive) Laplacian and (massive) Dirac operator are “self-dual”: the dual operator $\Dsl$ is the (massive) Laplacian on $\SGs$ associated with the dual conductance function $\cs$ that we already introduced (and a natural dual mass function $\ms$ in \cite{ZinvDirac}). This leads to the following question: is there a choice of function $\ls: \SV \to \R_{>0} $ that makes Equation~\eqref{eq:massive:holomorphy} “self-dual” ? Or in other words, does there exist a function $\ls: \SVs \to \R_{>0}$ such that $\Dsl$ is the massive Laplacian operator on $\SGs$ associated with the dual conductance function $\cs$ and a dual mass function $\ms$ ? More formally, let $\ms: \SV \to \R_{\geq 0}$ be an arbitrary mass function. Recall that a dual conductance function $\cs$ is naturally associated to $c$: if $xy \in \SE$ is the unique primal edge intersecting the dual edge $uv \in \SEs$, $\cs_{uv} = c_{xy}^{-1}$. Denote by $\Dsk$ the Laplacian on $\SGs$ associated with the conductance and mass functions $(\cs,\ms)$.
\begin{Def}
    The killed dimer model associated with the positive functions $\lambda$, $\ls$ is \emph{self-dual} for the conductance function $c$ and mass functions $(m,\ms)$ if and only if 
\begin{equation}
    (\Kk)^{\dag}\Kk = \begin{pmatrix} \Dk&0 \\ 0&\Dsk \end{pmatrix}.
\end{equation}
\end{Def}

\begin{Prop}
    If $\lambda$ is massive harmonic for $\Dk$ on $\SV$, $\ls$ is massive harmonic for $\Dsk$ on $\SVs$ and for all intersecting primal and dual edges $xy \in \SE$, $\x \y  \in \SEs$,
    \be\label{eq:meta}
        \lambda(x)\lambda(y)\ls(\x )\ls(\y) = 1.
    \ee
    the killed dimer model associated with the positive functions $\lambda, \ls$ is self dual for the conductance function $c$ and mass functions $(m, \ms)$ 
\end{Prop}

\begin{Rem}
    The hypothesis of this proposition are satisfied in the isoradial settings of \cite{Kenyon_2002} and \cite{ZinvDirac} where $\lambda$ and $\ls$ are the restriction to the primal and dual graph of the discrete (massive) exponential function.
\end{Rem}

\begin{proof}
    Comparing the extra-diagonal coefficients of the Laplacian $\Dsl$ (see Equation~\eqref{eq:def:dual:laplacian}) with the extra diagonal coefficients of $\Dsk$, shows that when Equation~\eqref{eq:meta} holds, for all $\x  \neq \y  \in \SVs$ such that $\x  \overset{\SGs}{\sim} \y$, if $xy \in \SE$ is the unique primal edge intersecting $uv$
    \begin{equation}
        \Dsl_{\x ,\y} = \cls(\ls(u)\ls(v))^{-1} = (c_{xy} \lambda(x)\lambda(y)\ls(u)\ls(v))^{-1} = \Dsk_{\x, \y}.
    \end{equation}
    Moreover, for all $\x  \in \SVs$, by massive harmonicity of $\ls$ at $\x $ for $\Dsk$.
    \begin{equation*}
            \Dsl_{\x ,\x } = \sum_{\z \overset{\SGs}{\sim} \x ,\z \neq \x } \cls_{(\x, v)}\ls(\x )^{-2} 
            \overset{\eqref{eq:meta}}{=} \sum_{\z \overset{\SGs}{\sim} \x , \z \neq \x } \cs_{\x \y} \frac{\ls(\y )}{\ls(\x )}
            = \ms(\x ) + \sum_{\z \overset{\SGs}{\sim} \x , \z \in \SVs}\cs_{\x \y} = \Dsk_{u,u}.\qedhere
    \end{equation*}
\end{proof}

\section{Application: rooted spanning forests of \texorpdfstring{$\Z^2$}{Z2}-periodic graphs}\label{sec:per}

In this \lr{section} we study random RSF (and more generally, cycle rooted spanning forests) of $\Z^2$-periodic graphs. This study was initiated independently in \cite{DeterminantalForest} (for the massive undirected case) and \cite{Wangru} (for the non-massive directed case). We first give some preliminaries and definitions, then explain how to find massive harmonic functions on such graphs, and finally we give applications of the Doob transform technique.

\subsection{Definitions and first results}
We recall the setting and some results of \cite{DeterminantalForest}. In this section, we always assume that $\SG$ is infinite.

\paragraph{$\Z^d$-periodic graphs and Laplacians.}
The graph $\SG$ is \emph{$\Z^d$-periodic} if it can be embedded in $\R^d$ in such a way that it is invariant under translations by the canonical basis $(e_1,\dots,e_d)$ of $\R^d$ with finite quotient $\SG_0 = (\SV_0,\SE_0) = \SG/\Z^d$. The graph $\SG_0$ is naturally embedded on the $d$-dimensional torus. We choose an embedding, and define the \emph{fundamental domain} $\SV_0 = \SV \cap [0,1)^d$: it is identified with the vertices of $\SG_0$. The vertices of $\SG$ can be written uniquely under the form $(x_0,i)$ with $x_0 \in \SV_0$ and $i = (i_1,\dots,i_d) \in \Z^d$, and for all $x,y \in \SV$, $x \sim y$ if and only if $\forall i \in \Z^d,~x+i \sim y+i$. The conductance and mass functions $(c,m)$ are \emph{$\Z^d$-periodic} if for all $x, y \in \SV$ and $i \in \Z^d$, $c_{(x+i,y+i)} = c_{(x,y)}$ and $m(x+i) = m(x)$. The \lr{simplest} example $\SG = \Z^d$, $c$ and $m$ are constant corresponds to the simple random walk in $\Z^d$ killed with constant rate $m$: in this case, $\SV_0$ is reduced to a single point. In the rest of this section, we always assume that $\SG$, $c$ and $m$ are $\Z^d$-periodic. For $z = (z_1,\dots,z_d) \in (\C \setminus \{0\})^d$ and $i \in \Z^d$, we write
\begin{equation}
    z^i = \prod_{k=1}^d z_k^{i_k}.
\end{equation}
As in Section 4 of \cite{DeterminantalForest}, for any $z \in (\C \setminus \{0\})^d$, a function $f:\SV \to \C$ is \emph{$z$-periodic} if it satisfies
\begin{equation}
    \forall x \in \SV,~\forall i \in \Z^d,~f(x + i) = f(x)z^i.
\end{equation}
\lr{We let $\Omega_0(z)$ denote the set of $z$-periodic functions.} We denote by $\Dk(z): \Omega_0(z) \to \Omega_0(z)$ the restriction of the massive Laplacian operator $\Dk$ to $\Omega_0(z)$: this is well-defined since $c$ and $m$ are $\Z^d$-periodic. The canonical base of $\Omega_0(z)$ is the set of functions $(f_{x_0})_{x_0 \in \SV_0}$ where for all $x_0 \in \SV_0$, $f_{x_0}: \SV \to \C$ is defined by 
\begin{equation}
    \forall y_0 \in \SV_0,~\forall i \in \Z^d,~f_{x_0}(y_0,i) = \mathbbm{1}_{\{x_0=y_0\}}z^i. 
\end{equation}
In this base, $\Dk(z)$ is represented by a square matrix $\Dk(z)$ with rows and columns indexed by $\SV_0$: for all $x_0, y_0 \in \SV_0$,
\begin{equation}
    \Dk(z)_{x_0,y_0} = 
    \left\{
    \begin{array}{ll}
        - \sum_{(y_0,j) \sim (x_0,0)}c_{((x_0,0),(y_0,j))}z^j &\text{ if }x_0y_0 \in \SE_0, x_0 \neq y_0,\\
        m(x_0) + \sum_{(w_0,k) \sim (x_0,0)} c_{((x_0,0),(w_0,k))} \\\qquad- \sum_{(x_0,j) \sim (x_0,0)}c_{((x_0,0),(x_0,j))}z^j &\text{ if } x_0=y_0,\\
        0 &\text{ otherwise.}
    \end{array}
    \right.
\end{equation}
(The underlying notion is that of Laplacian on a line bundle, as in \cite{DeterminantalForest}, but we will not need to introduce it in full generality here so we restrain to elementary definitions). The \emph{characteristic polynomial} of $\Dk$ is the (Laurent) polynomial of $d$ complex variables
\begin{equation}\label{eq:def:polynomial}
    \CPk(z) = \det \Dk(z).
\end{equation}
We denote by $\Nk$ its \emph{Newton polygon}: it is the convex hull in $\R^d$ of the set of points $i \in \Z^d$ such that the monomial $z^i$ has a non-zero coefficient in $\CPk$.

\subsection{Massive harmonic functions on \texorpdfstring{$\Z^d$}{Zd}-periodic graphs}\label{subsec:periodic}
Consider a $\Z^d$-periodic graph $\SG$ with $\Z^d$-periodic mass and conductance functions $(c,m)$.
\begin{Prop}\label{prop:harmonic:periodic}
    There exists $z \in (\R_{>0})^d$ and a  $z$-periodic positive function $\lambda: \SV \to \R_{>0}$ which is massive harmonic, i.e. $(\Dk \lambda)(x) = 0$ for all $x \in \SV$. It can be found explicitly by solving a finite system of linear equations.
\end{Prop}
\begin{proof}
This argument was shown to us by C{\'e}dric Boutillier when we realized that we were working on related problems (see a forthcoming paper of Ballu, Boutillier, Mkrtchyan and Raschel). The first step is to observe that finding a periodic massive harmonic function is not a hard problem in principle: for fixed $z \in (\R_{>0})^d$, finding a $z$-periodic positive massive harmonic function is equivalent to finding a vector with positive coordinates in the kernel of the finite matrix $\Dk(z)$, i.e. solving a finite linear system.\par
If $m=0$ any positive constant is a $(1,\dots,1)$-periodic function so we can assume that $m \neq 0$. Recall that $D(\ck)$ denote\lr{s} the diagonal matrix on $V_0$ with coefficients $\ck(x_0)$. If we define, similarly to Equation~\eqref{eq:def:massive:transition:kernel}, $\Qk(z) = I - D(\ck)^{-1}\Dk(z)$, $Q(z)$ is the operator on $\Omega_0(z)$ satisfying for all $f \in \Omega_0(z)$, for all $x \in \SV$,
\begin{equation}
    (\Qk(z)f)(x) = \sum_{y \sim x}\frac{c_{(x,y)}}{\ck(x)}f(y).
\end{equation}
It is also represented by a square matrix with rows and columns indexed by the canonical base of $\Omega_0(z)$:
\begin{equation}
    \forall x_0,y_0 \in \SV_0,~\Qk(z)_{x_0,y_0} = 
        \frac{1}{\ck(x_0,0)}\sum_{(y_0,j) \overset{G}{\sim} (x_0,0)}c_{((x_0,0),(y_0,j))}z^j.
\end{equation}
Since $D(\ck)$ is diagonal with positive coefficients, 
\begin{equation}
   (\R_{>0})^{\SV_0} \cap \ker(\Dk(z)) = (\R_{>0})^{\SV_0} \cap \ker(I-\Qk(z)).
\end{equation}
Since $\SG$ is connected, the matrix $\Qk(z)$ is irreducible for all $z \in (\R_{>0})^d$ (i.e. for all $x_0,y_0 \in \SV_0$, there exists $n > 0$ such that $(\Qk)^n_{x_0,y_0}>0$). Observe that $\Qk := \Qk(1,\dots,1))$ is the transition kernel of the \lr{killed random walk} on the $d$-dimensional torus $\SG_0$ associated to $\Dk$ since the space of $(1,\dots,1)$-periodic function\lr{s} on $\SV$ is exactly the space of functions on the torus. It is a sub-Markovian matrix:
\be\label{eq:sub-Markovian}
    \forall x_0 \in \SV_0,~ \sum_{y_0 \overset{\SG_0}{\sim} x_0} \Qk(x_0,y_0) = \sum_{y \overset{\SG}{\sim} (x_0,0)}\frac{c_{((x_0,0),y)}}{\ck(x_0)} = 1-\frac{m(x_0)}{\ck(x_0)} \leq 1,
\ee
\lr{and it is irreducible with non-negative coefficients. Let us denote by $\beta \in [0,1]$ the Perron-Frobenius eigenvalue of $\Qk$. Since $\SG$ is connected and $m \neq 0$, there exists $l \in \Z_{>0}$ such that 
\be
    \forall x_0 \in \SV_0,~ \sum_{y_0 \overset{\SG_0}{\sim} x_0} (\Qk)^l(x_0,y_0) < 1.
\ee
By the Perron-Frobenius theorem the maximal eigenvalue $\beta^l$ of \lr{$(\Qk)^l$} satisfies 
$$
    \beta^l \leq \max_{x_0 \in \SV_0} \bigg(\sum_{y_0 \overset{\SG_0}{\sim} x_0} (\Qk)^l(x_0,y_0)\bigg) <1,
$$
and hence $\beta < 1$. This implies that} $\ker(I-\Qk) = \{0\}$. 

If (for example) we fix $z_2,\dots,z_d=1$ and let \lr{$z_1 \in \R^{+}$}, $\Qk(z_1,1,\dots,1)$ remains irreducible with non-negative coefficients, and the Perron-Frobenius theorem still applies and yields a maximal eigenvalue $\beta(z_1,1,\dots,1)$ and an associated eigenvector $\lambda(z_1,1,\dots,1)$ with positive coefficients. Assume that
\be\label{eq:infty}
	\rho(z_1,1,\dots,1) \overset{z_1 \to \infty}{\longrightarrow} \infty
\ee 
(this will be proved just after). \lr{Since we already proved that $\beta = \beta(1,\dots,1) < 1$, b}y continuity of the Perron-Frobenius eigenvalue there exists $z_1 > 1$ such that $\beta(z_1,1,\dots,1) = 1$ and hence $\lambda(z_1,1,\dots,1) \in \ker(I-\Qk(z_1,1,\dots,1)) \cap (\R_{>0})^{\SV_0}$ is a positive $(z_1,1,\dots,1)$-periodic massive harmonic function.\par
We now prove Equation~\eqref{eq:infty}, which concludes the proof. Let $x_0 \in \SV_0$. Since $\SG$ is connected we can find a path $(x(0), x(1), \dots, x(n))$ with $x(0) = (x_0,(0,\dots,0))$ and $x(n) = (x_0,(1,0,\dots,0)) = x(0) + (1,0,\dots,0)$ such that $c_{(x(i),x(i+1))} > 0$ for all $0 \leq i \leq n-1$. Hence
\begin{equation}
    (\Qk(z_1,1,\dots,1))^n_{x_0, x_0} \geq z_1 \prod_{i=0}^{n-1} \frac{c_{(x(i), x(i+1))}}{\ck(x(i))}.
\end{equation}
Hence for any $N \in \N$, 
\begin{equation}
    (\Qk(z_1,1,\dots,1))^{Nn}_{x_0,x_0} \geq \left(z_1 \prod_{i=0}^{n-1} \frac{c_{(x(i), x(i+1))}}{\ck(x(i))}\right)^N.
\end{equation}
\lr{Let us denote by $\lVert \cdot \rVert_2$ the usual Euclidean norm on square matrices with real coefficients. By Gelfand's formula, we can bound the spectral radius from below
\be\label{eq:gelfand}
    \rho(\Qk(z_1,1,\dots,1)) \geq \lim_{N \to \infty}\lVert (\Qk(z_1,1,\dots,1))^{Nn} \rVert_2^{\frac{1}{Nn}} \geq \left(z_1 \prod_{i=0}^{\lr{n}-1} \frac{c_{(x(i), x(i+1))}}{\ck(x(i))}\right)^{\frac{1}{n}}.
\ee
Since the Perron-Frobenius eigenvalue coincides with the spectral radius, $\beta(z_1,1,\dots,1)$ is also bounded by the right-hand side of Equation \eqref{eq:gelfand}, which} goes to $\infty$ when $z_1 \to \infty$. \lr{This} proves Equation~\eqref{eq:infty}.
\end{proof}

\subsection{Temperley's bijection for the free energy of $\Z^2$-periodic models}\label{subsec:Temperley:per}
In this section, we assume that $d=2$ and $c$ is symmetric. The massive Laplacian $\Dk$ that we introduced is related to a RSF model on $\SG$. This model is obtained by taking the weak limit of Boltzmann measures on \emph{cycle-rooted spanning forests} (CRSF) along the toroidal exhaustion $\SG_n = \SG/n\Z^2$ (see \cite{DeterminantalForest} for details). The characteristic polynomial $\CPk$ is an important tool to study the RSF model on the graph $\SG$: it is linked to the Ronkin function and the surface tension of the RSF model. We show that the characteristic polynomial of the massive Laplacian $\Dk$ can be related to the characteristic polynomial of a non-massive directed Laplacian.\par
Let $z_0 \in (\R_{>0})^2$ and $\lambda$ a $z_0$-periodic positive massive harmonic function as in Proposition~\ref{prop:harmonic:periodic}. Let $\Sl$ be the Doob transform of $\Sk$ by $\lambda$ on $\SV$, $\Dl$ the associated Laplacian with conductance function $\cl$: recall that we simply have for all $x \sim y \in \SV$, $\cl_{(x,y)} = \frac{\lambda(y)}{\lambda(x)}c_{xy}$. These conductances are also periodic:
\begin{equation}
    \forall x \sim y \in \SV,~\forall i \in \Z^d,~\cl_{(x+i,y+i)} = \frac{\lambda(y+i)}{\lambda(x+i)}c_{(x+i)(y+i)} = \frac{\lambda(y)z_0^i}{\lambda(x)z_0^i}c_{xy} = \cl_{(x,y)}.
\end{equation}
Denote by $\CPk$ and $\CPl$ the characteristic polynomials of $\Dk$ and $\Dl$. The Doob transform technique implies
\begin{Prop}\label{prop:harmonic:periodic}
    \be\label{eq:translate:polynomial}
	\forall z \in \C^2,~\CPl\left(\frac{z}{z_0}\right) = \CPk(z).
    \ee
\end{Prop} 
\begin{proof}
    Denote by $\Lambda$ the diagonal matrix with rows and columns indexed by $\SV$ with entries $\lambda$ on the diagonal. Observe that for all $z \in (\C \setminus \{0\})^2$, if $f \in \Omega_0(z)$, $\Lambda f \in \Omega_0(z z_0)$:
    \begin{equation}
        \forall x \in \SV,~\forall i \in \Z^2,~(\Lambda f)(x+i) = \lambda(x+i)f(x+i) = (z z_0)^i(\Lambda f)(x).
    \end{equation}
    Hence $\Lambda: \Omega_0(z) \to \Omega_0(z z_0)$. By Proposition~\ref{prop:doob}, $\Dl = \Lambda^{-1} \Dk \Lambda$, so when restricting to $\Omega_0(z/z_0)$, we obtain
    \begin{equation}
        \Dl\left(\frac{z}{z_0}\right) = \Lambda^{-1} \Delta^{m}(z) \Lambda.
    \end{equation}
    Taking determinants concludes the proof.
\end{proof}

The non-massive Laplacian $\Dl$ can be related to the drifted dimer model associated with the positive function $\lambda$ that we introduced in the preceding section (see Definition \ref{def:dimer:drifted}) by using Temperley's bijection: this is a result of \cite{Wangru}. Denote by $\CPd$ the characteristic polynomial of the drifted dimer model (see \cite{Wangru} for the definition): 
\begin{Prop}[Proposition 3.1 of \cite{Wangru}]
    For all $(z,w) \in \C^2$, $\CPl(z,w) = \det(\Dl(z,w)) = \CPd(z,w)$.
\end{Prop}
Combining this proposition with our Proposition~\ref{prop:harmonic:periodic} gives:
\begin{Prop}\label{prop:char:periodic}
	For all $(z,w) \in \C^2$,
	\begin{equation}
		\CPk(z,w) = \CPd\left(\frac{z}{z_0},\frac{w}{w_0}\right).
	\end{equation}
\end{Prop}
This result can be seen as an extension of Temperley's bijection for infinite periodic RSF models, as it relates the characteristic polynomial of a RSF model with the characteristic polynomial of a dimer model. Observe that since the free energy of the models can be computed from the characteristic polynomial, this result can be seen as an infinite analog of Temperley's bijection at the level of partition functions. Proposition~\ref{prop:char:periodic} also sheds a new light on a result of Kenyon: Theorem 1.4. of \cite{DeterminantalForest} states among other things that
\begin{Thm}\label{thm:kenyon}
	The spectral curve $\{(z,w) \in \C^2~|~\CPk(z,w)=0\}$ associated with the RSF model is a simple Harnack curve.
\end{Thm}
This could come as a surprise since the RSF model is not known to be in bijection with a dimer model while it is known by \cite{PlanarDimers} that every Harnack curve arises as the spectral curve of a dimer model (and conversely by \cite{KOS} the spectral curve of a periodic dimer model is always a Harnack curve). Proposition~\ref{prop:char:periodic} together with the result of \cite{KOS} give an alternative proof of Theorem~\ref{thm:kenyon} which fills a gap by introducing a dimer model corresponding to the Harnack curve (actually: a full family of dimer models, corresponding to all the positive massive harmonic $z$-periodic functions $\lambda$).

\section{Application to near-critical dimers}\label{sec:near-critical}
In this section, we study a near-critical dimer and RST model naturally related to a near-critical \lr{killed random walk} by the Doob transform technique and show convergence of the associated Temperleyan tree, extending the results of \cite{Berestycki} and answering in particular (iv) of their open questions \lr{section~1.7}. We introduce the model that generalizes their model to the case of isoradial graphs. We follow closely their proof to obtain analogous results and adapt the arguments when needed.

\subsection{General setting: near-critical drifted trees and dimers on isoradial graphs}\label{subsec:near-critical:preliminaries}
For the conventions, we follow as much as possible \cite{MassiveLaplacian} and \cite{DiscreteHolomorphy} for isoradial graphs and the massive Laplacian, and \cite{Berestycki} for dimers and trees. 
\subsubsection{Isoradial graphs, $Z$-invariant Laplacian and discrete exponential functions}
\paragraph{Isoradial lattice.} 
For each $\d >0$, consider an infinite \emph{isoradial grid} $\oSGd = (\oSVd, \oSEd)$ embedded in $\C$ with \emph{mesh} $\d$, that is an infinite connected planar graph such that all faces are circumscribed by circles of the same radius $\d$ in such a way that the center of the circles are in the interior of the faces. We use the same local parameters as in \cite{MassiveLaplacian}, see Figure \ref{fig:local_dimers}: for all $x \sim y \in \oSVd$, $x$ and $y$ are the opposite vertices of a lozenge of side length $\d$, and the two (vector) sides of the lozenge starting from $x$ are denoted by $\d e^{i\bar{\alpha}_{xy}}$ and $\d e^{i\bar{\beta}_{xy}}$ with $\bar{\alpha}_{xy}, \bar{\beta}_{xy} \in (-\pi, \pi]$ and $\bar{\beta}_{xy} - \bar{\alpha}_{xy} \in [0, \pi)$. We denote by $\bt_j$ the \emph{half angle} $\bt_{xy} = \left(\bar{\beta}_{xy}-\bar{\alpha}_{xy}\right)/2$. Note that $y-x = \d e^{i\bar{\alpha}_{xy}} +\d e^{i\bar{\beta}_{xy}}$.\par
We always make the so-called \emph{bounded angle assumption}: there exists an absolute constant $\eps \in (0, \pi/\lr{4})$ independent of $\d$ such that all half-angles $\bt_{xy}$ belong to $\left[\eps, \pi/2 - \eps\right]$.
\begin{figure}[!h]\centering
    \begin{overpic}[abs,unit=1mm,scale=1]{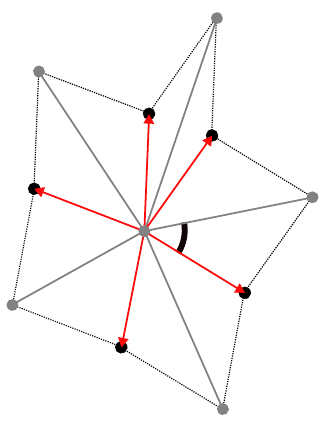}
        \put(27,37){\color{gray}$x$}
        \put(54.5,41){\color{gray}$y$}
        \put(35,29){\color{red}$\d e^{i\bar{\a}_{xy}}$}
        \put(32,43){\color{red}$\d e^{i\bar{\beta}_{xy}}$}
    \end{overpic}
    \caption{A point $x \in \oSVd$ with its neighbours in grey. All red arrows and dotted lines have length $\d$. The vertices of the dual graph are represented in black: they are the centers of the circumscribing circles (not drawn here).}
    \label{fig:local_dimers}
\end{figure}
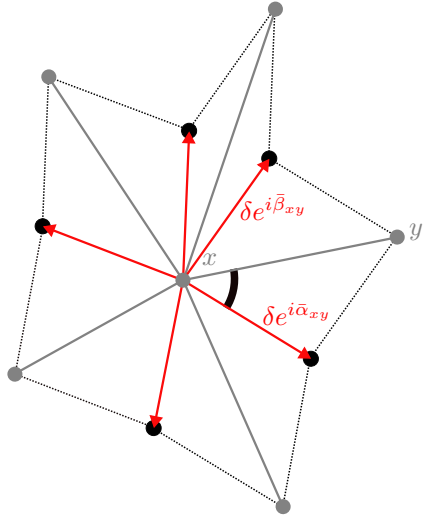

\paragraph{Critical Laplacian.} The critical Laplacian $\Dd$ on $\oSGd$, first introduced in \cite{Kenyon_2002}, is defined by
\be\label{eq:def:non-massive:D}
	\forall f: \oSVd \to \C,~\forall x \in \oSVd,~(\oDd f)(x) = \sum_{y \sim x} \tan(\bt_{xy})(f(x) - f(y)). 
\ee
\lr{Note that our Laplacian differs from a negative multiplicative constant from the Laplacian of \cite{DiscreteHolomorphy}: if $D$ denotes the diagonal matrix with entries is $-\d^2\sum_{y \sim x}\sin(2\bt_{xy})/2$, their Laplacian is $D^{-1}\Dd$.} Define
\be\label{eq:def:Tdelta}
    \forall x \in \oSVd,~\forall \d > 0,~T_{\d}(x) := \frac{\sum_{y \sim x}\sin(2\bt_{xy})}{\sum_{y \sim x}\tan\big(\bt_{xy}\big)} \quad ; \quad T := \inf_{\d > 0, x \in \oSVd} T_{\d}(x) > 0.
\ee
The latter is positive by the bounded angle assumption. Unlike \cite{DiscreteHolomorphy}, we choose not to scale by $\d^2$, which is more natural in our setting. We denote by $\Xd$ the \lr{random walk} associated with $\Dd$, and by $\xid(x)$ the law of its increments: conditionally on $(\Xd)_n = x$, it is the law of $(\Xd)_{n+1}-(\Xd)_n$.
Equation (1.2) of \cite{DiscreteHolomorphy} writes, for all $x \in \oSVd$,
\be\label{eq:increments:critical}
    \E(\xid(x))=0 \quad ; \quad \E(\Re(\xid(x))^2) = \E(\Im(\xid(x))^2) = \d^2 T_{\d}(x) \quad ; \quad \E\big(\Re(\xid(x))\Im(\xid(x))\big) = 0.
\ee
As mentioned in \cite{DiscreteHolomorphy}, this implies that a certain time-reparametrization of $\Xd$ converges towards \lr{B}rownian motion.

\paragraph{$Z$-invariant Laplacian.} Let $k \in [0,1\lr{)}$ be a parameter called the \emph{elliptic modulus}, $k' = \sqrt{1-k^2}$ be the \emph{complementary elliptic modulus}, $K = K(k) = \int_0^{\pi/2}(1-k^2\sin^2(t))^{-1/2}dt$ and $E = E(k) = \int_0^{\pi/2}(1-k^2\sin^2(t))^{1/2}dt$ be the \emph{complete elliptic integral of the first and second kind}. Denote by $K' = K(k')$ and $E'=E(k')$ their \emph{complementary}. For fixed $k \in [0,1\lr{)}$ and any angle $\bt \in \R$, we denote by $\theta = 2K\bt/\pi$ the associated \emph{abstract angle}. To an elliptic modulus $k$ is associated a massive Laplacian $\Dkd$ (and hence a \lr{killed random walk} $\Skd$) by \cite{MassiveLaplacian}:
\begin{equation}\label{def:Zinv:laplacian}
    \forall f: \oSVd \to \C,~\forall x \in \oSVd,~\Dkd f(x) =
    m^2(x|k)f(x) + \sum_{y \sim x}\sc(\theta_{xy}|k)(f(x)-f(y))
\end{equation}
where the conductance function $\oc_{xy} := \sc(\theta_{xy}|k)$ is the Jacobian elliptic function defined in (22.2.9) of \cite{DLMF} and the mass function is
\begin{equation}
    \om(x) := m^2(x|k) = \sum_{y \sim x} \left[\frac{1}{k'}\left(\int_0^{\theta_{xy}}\dc^2(v|k)dv+\frac{E-K}{K}\theta_{xy}\right)-\sc(\theta_{xy}|k)\right],
\end{equation}
with $\dc$ the Jacobian elliptic function defined in (22.2.8) of \cite{DLMF}. We will not need this explicit expression of the mass and provide it only for completeness. The positivity of the mass function $m^2(x|k)>0$ for all $x \in \oSVd$ is established in Proposition 6 of \cite{MassiveLaplacian}.\par
At $k=0$, the massive Laplacian $\Dkd$ coincides with the critical Laplacian $\Dd$ of Equation~\eqref{eq:def:non-massive:D}: $m(\cdot|0) = 0$ and $\sc(\cdot|0) = \tan(\cdot)$.  
\begin{Rem}
    The notation might be a bit confusing since the $k$ of the elliptic modulus is not the same as the superscript $\mathrm{k}$ of $\Dk$ which means killing. Since $k=0$ corresponds to $\Dk$ being a non-massive Laplacian, we find it rather convenient.
\end{Rem}
\paragraph{The discrete massive exponential functions.} A very useful property of the massive Laplacian is the existence of an explicit family of multiplicative local massive harmonic functions: we recall its definition and summarize some of its properties from \cite{MassiveLaplacian}.
\begin{Prop}\label{prop:me}
    Let $k \in [0,1\lr{)}$, $\bu \in \C$. The \emph{discrete massive exponential function} with parameter $\bu$ is defined for all $x,y \in \oSVd$ and any path $x = x_0,\dots, x_n = y$ in $\oSGd$ from $x$ to $y$ by
    \be\label{eq:def:me}
        \me_{(x_j,x_{j+1})}(u|k) := -k'\sc\Big(\frac{u-\alpha_{x_j x_{j+1}}}{2}\Big|k\Big)\sc\Big(\frac{u-\beta_{x_j x_{j+1}}}{2}\Big|k\Big), \quad \me_{(x,y)}(u|k) := \prod_{j=0}^{n-1}\me_{(x_j,x_{j+1})}(u|k)
    \ee
    It is massive harmonic in both variables: for any fixed $x \in \oSVd$,
    \begin{equation}
        \Dkd \me_{(\cdot, x)}(u|k) = \Dkd \me_{(x, \cdot)}(u|k) = 0.
    \end{equation}
    Moreover, when $\bu \in \R$, the discrete massive exponential function evaluated at $u-2K-2iK'$ is positive:
    \begin{equation}
        \forall \bu \in \R,~\forall x,y \in \oSVd,~\me_{(x,y)}(u-2K-2iK'|k) \in \R_{>0}.
    \end{equation}
\end{Prop}
Equation~\eqref{eq:def:me} is the definition (see Definition 3.3 of \cite{MassiveLaplacian}): the fact that the definition does not depend of the path is their Lemma 9. Observe that they define $\me$ and prove that it is multiplicative on the full diamond graph (i.e. the union of $\oSGd$ and its dual), but we will not need it here so we restrict the definition to $\oSVd$. The massive harmonicity is their Proposition 11. The positivity is known but we could not find a reference so we provide a short proof:
\begin{proof}
    Fix $\bu \in \R$. It is enough to prove that for all $x \sim y \in \oSVd$, $\me_{(x,y)}(u-2K-2iK'|k) \in \R_{>0}$. By the change of variable formula for elliptic functions Equations (2.2.11)-(2.2.13) and (2.2.17)-(2.2.19) of \cite{Lawden} we can write
    \be\label{ScDn}
	i k' \sc\Big(\frac{u-2K-2iK'-\alpha_{xy}}{2}\Big|k\Big) = \dn\Big(\frac{u-\alpha_{xy}}{2}\Big|k\Big)
    \ee
    where $\dn$ is the Jacobian elliptic function defined in (22.2.6) of \cite{DLMF}. Since the same holds for $\beta_{xy}$,
    \be\label{eq:me:neighbours}
        \me_{(x,y)}(u-2K-2iK'|k) = \frac{1}{k'}\dn\Big(\frac{u-\alpha_{xy}}{2}\Big|k\Big)\dn\Big(\frac{u-\beta_{xy}}{2}\Big|k\Big).
    \ee
    This concludes the proof since $\dn$ is positive on $\R$. Indeed, on the explicit expression (22.2.6) of $\dn$ in \cite{DLMF}, we see that $\dn(0)=1$ and that $\dn$ is the quotient of two real continuous functions which do not vanish on the real line (by (20.2.(iv)) of \cite{DLMF}).
\end{proof}

\subsubsection{Isoradial approximation of an open set} For the rest of the article, we fix a simply connected open set $\Omega$. \lr{Recall that $\oSGd$ is an infinite isoradial grid embedded in $\C$ with \emph{mesh} $\d$}. We say that a sequence of subgraphs $\SGd$ of $\oSGd$ induced by finite subsets of vertices $\SVd \subset \oSVd$ \emph{approximates} $\Omega$ if the domains $\hSGd$ obtained by gluing together all lozenges with two opposite vertices in $\SVd$ converge towards $\Omega$ in the Caratheodory topology. In other words, every $x \in \Omega$ belongs to $\hSGd$ for $\d$ small enough, and every boundary point $a \in \partial \Omega$ can be approximated by a sequence $\ad \in \partial \SVd$ (see \cite{Berestycki}, Section 3.1).

We can construct an approximating sequence $\SGd$ by keeping the largest connected component of the subgraph induced by $\oSVd \cap \Omega$ (see also Section 1.1 of \cite{FermionicObservable}). For the rest of the article, we consider a fixed approximating sequence $(\SGd)_{\d>0}$ of $\Omega$. \lr{An example is drawn on Figure \ref{fig:isoradial:approx}.}

\begin{figure}[!h]\label{fig:isoradial:approx}\centering
    \begin{overpic}[abs,unit=1mm,scale=3]{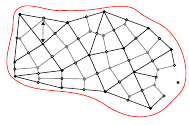}
    \put(80,54){\color{red}\huge$\Omega$}
    \put(23,46){$\delta$}
    \put(90,23){$r$}
    \end{overpic}
    \caption{An isoradial approximation $\SGd$ of $\Omega$ in grey, and the associated Temperleyan graph. The dual $\SGd^{\star}$ with a vertex $r$ removed is drawn in black. The additional white vertices of $\Gdr_{\d}$ are in white.}
    \label{fig:U}
\end{figure}

Recall that the restriction $(\Dkd)_{\SVd}$ is the Laplacian on $\SGd$ associated with the conductance and mass functions $(c,m)$ obtained by restricting $(\oc, \om)$ to $\SGd$ with wired boundary conditions. In other words, for all $xy \in \SEd$, $c_{xy} = \sc(\theta_{xy}|k)$ and the mass function
\begin{equation}
    \forall x \in \SVd,~m(x) = m^2(x|k) + \sum_{y \overset{\oSGd}{\sim} x,~y \notin \SVd} \sc(\theta_{xy}|k)
\end{equation}
differs from $m^2(x|k)$ only on the boundary.

\subsubsection{Application of the Doob transform technique and introduction of the model}
\paragraph{Doob transform.}
For the rest of the article, fix $\bu \in \R$ which we call the \emph{drift} parameter (for reasons that will become clear later). Recall that the complete integral of the first kind $K$ and and its complementary $K'$ are real constants depending only on $k$. For all $\d >0$, fix $x_0 \in \oSVd$ arbitrary and define
\begin{equation}
    \forall x \in \oSVd,~\lambda^u(x) = \me_{(x_0,x)}(u-2K-2iK'|k).
\end{equation}
By Proposition~\ref{prop:me}, $\lambda^u$ is a positive massive harmonic function for $\Dkd$ on $\oSVd$, so in particular it is massive harmonic on $\SVd$ and we can apply the Doob transform technique (this is exactly the setting mentioned in the comments after Proposition~\ref{prop:find:lambda:finite}). Recall that the Doob transform $\tSd$ of the \lr{killed random walk} $\Skd$ is the \lr{random walk} on $\oSGd$ with conductances $\ocl_{(x,y)} = \frac{\lambda^u(y)}{\lambda^u(x)}\sc(\theta_{xy}|k)$, and denote by $\Dld$ the non-massive Laplacian associated with these conductances. Denote by $\Lambda$ the diagonal matrix with rows and columns indexed by $\oSVd$ and entries $\lambda^u$ on the diagonal. Proposition \ref{prop:doob} implies
\begin{equation}
    (\Dld)_{\SVd} = \Lambda_{\SVd}^{-1}(\Dkd)_{\SVd}\Lambda_{\SVd}.
\end{equation}

\paragraph{RSF and dimers.}
To the massive Laplacian $(\Dkd)_{\SVd}$ is associated a RSF model on $\SGd$ with conductance and mass functions $(c,m)$. As was already mentioned, for $k >0$, the mass function $m$ is positive at all $x \in \SVd$: this is a typical example where the graph $\SGd^{\rho}$ of the forest-tree bijection is highly non-planar. To $(\Dld)_{\SVd}$ is associated a RST model on $\SGd^o$ (obtained by identifying all vertices in $\oSVd \setminus \SVd$ to $o$, see Figure \ref{fig:Go}) rooted at $o$ weighted by $\cl$ (see around Equation~\eqref{eq:mass:wired} for details).\par
As in Section~\ref{sec:planar}, fix an arbitrary vertex $r \in \partial \SVd^{\star}$ on the boundary of $\SGd^{\star}$, the restricted dual of $\SGd^o$. Recall from Section~\ref{SubsecTemperley} the definition of the double graph with $o$ and $r$ removed $\Gdr_{\d}$ associated with $\SGd$ (see Figure \ref{fig:U}). By Temperley's bijection and our Definition \ref{def:dimer:drifted} of the drifted dimer model, RST of $\SGd^o$ weighted by $\clo$ are in weight-preserving bijection with dimer configurations of the drifted dimer model on $\Gdr_{\d}$. The local weights for the RST model and the drifted dimer model away from the boundary are depicted in Figure \ref{fig:oriented:tree}. 
\begin{figure}[!h]\centering        
	\begin{overpic}[abs,unit=1mm,scale=1]{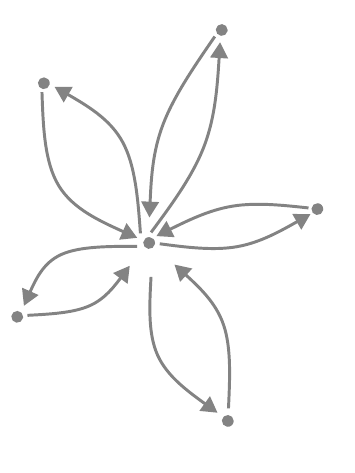}
		\put(55,42){\color{gray}$y$}
		\put(25,36){\color{gray}$x$}
		\put(40,35){\color{gray}$\tilde{c}_{(x,y)}$}
		\put(38,47.5){\color{gray}$\tilde{c}_{(y,x)}$}
	\end{overpic}
	\begin{overpic}[abs,unit=1mm,scale=1]{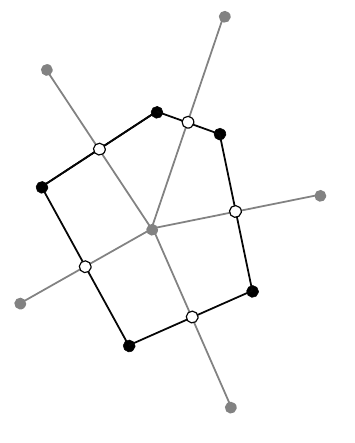}
		\put(57,41){\color{gray}$y$}
		\put(24,32){\color{gray}$x$}
		\put(43,31){$1$}
		\put(40,45){$1$}
		\put(31,40){\color{gray}$\tilde{c}_{(x,y)}$}
		\put(45,37){\color{gray}$\tilde{c}_{(y,x)}$}
        \end{overpic}
        \caption{On the left: the conductances $\cl$ around a vertex $x \in \SVd$ for the RST model on $\SGd^o$ rooted at $o$. On the right: the weights for the drifted dimer model on $\Gdr_{\d}$. The asymptotics of $\cl$ in the near-critical regime are given in Equation~\eqref{eq:cl:explicit} in terms of the local parameters.}
        \label{fig:oriented:tree}
\end{figure}
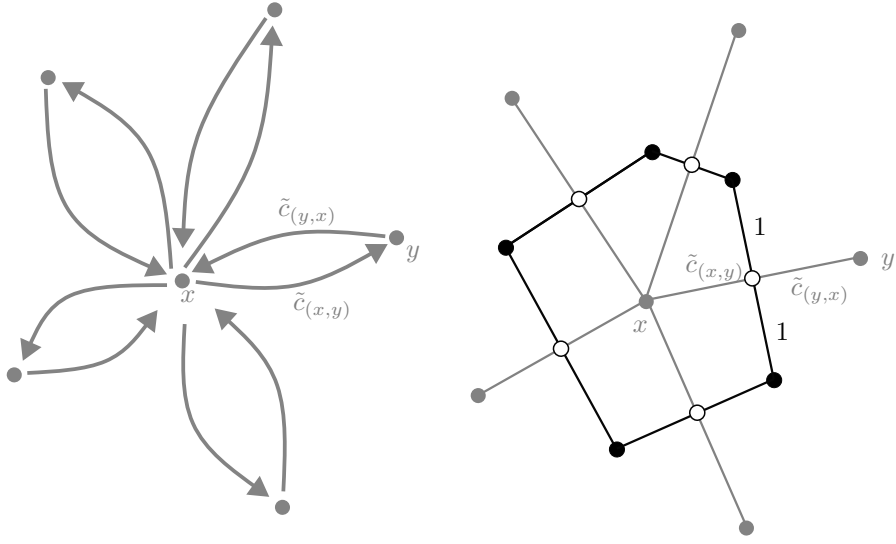

\subsubsection{Near-critical regime}\label{subsec:near-critical} We will study these models in the near-critical regime where the elliptic modulus $k$ goes to $0$ appropriately with $\delta$. Intuitively, it corresponds to the mass being of order $\d^2$ at every point (see the forthcoming Lemma~\ref{lem:mass}): it is an interesting regime where the probability that $\Skd$ leaves $\SVd$ before dying is bounded away from $0$ and $1$, so the shape of $\Omega$ has a non-trivial effect on the model.

More precisely, let $q = q(k) = \exp(-\pi K'/K)$ be the \emph{nome} associated with the elliptic parameter $k$. The \emph{near-critical regime} corresponds to $\delta \to 0$ and $q = \frac{1}{2}M\d$ where $M \in \R_{>0}$ is a real positive parameter called the \emph{mass parameter} fixed in the rest of the article (as in \cite{FermionicObservable}). In what follows, we use only one scale parameter $\d$, and write implicitly $q = q(\d)$, $k = k(\d)$. In this regime, we can compute the asymptotics of several of the quantities defined before.
\begin{Lem}\label{lem:asymp:near-critical}
    In the near-critical regime $q = \frac{1}{2}M\d$ with $\d \to 0$,
    \be
    \forall \bt \in \R,~
    \left\{
    \begin{array}{ll}
        k^2 &= 8M\d - 32M^2\d^2 + O(\d^3)\\
        \theta &= (1+2M\d+ M^2\d^2 + O(\d^3))\bt\\
        \sc(\theta|k) &= (1+2M\d+O(\d^2))\tan(\bt).
    \end{array}
    \right.
    \ee
    Moreover, for all $\bu \in \R$,
    \begin{equation}
        \forall x, y \in \oSVd,~\me_{(x,y)}(u-2K-2iK'|k) = \exp\left(2M\langle e^{i\bu}, y - x \rangle\right) + O(\delta^2|y - x|) 
    \end{equation}
    where $|y-x|$ is the \lr{E}uclidean distance on $\R^2$ and $\exp: \R \to \R$ is the usual exponential function. These asymptotics hold uniformly in the isoradial grid $\oSGd$ (given the bounded angle assumption), the angle $\bt \in \R$, the parameter $\bu \in \R$ and the points $x,y \in \oSVd$.
\end{Lem}
\begin{Lem}\label{lem:mass}
    In the near-critical regime $q = \frac{1}{2}M\delta$ with $\d \to 0$, uniformly in \lr{$x$}, 
    \begin{equation}
        \forall x \in \oSVd,~m^2(x|k) = 2M^2\d^2 \sum_{y\sim x} \sin(2\bt_{xy}) +O(\d^3). 
    \end{equation}
\end{Lem}
\begin{Rem}\label{rem:mass:positive}
In the bounded angle assumptions, the $\sin(\bt_j)$ are positive and bounded away from $0$, so the mass is of order $\d^2$ (uniformly in $x, \d$). 
\end{Rem}
We will prove Lemma~\ref{lem:mass} in Appendix~\ref{subsec:asymp:mass} using only the massive harmonicity of the discrete exponential and the approximation property of the Laplacian (which will be introduced only in the next section). We now prove Lemma~\ref{lem:asymp:near-critical}.
\begin{proof}
Using the expression of the elliptic modulus in terms of the nome and the theta functions (see (22.2.2) of \cite{DLMF}):
\begin{equation}
k^2 
= \left(\frac{2q^{1/4}(1+O(q^2))}{1+2q + O(q^2)}\right)^4
= 16q(1+O(q^2))(1-8q + O(q^2)) = 16q - 128q^2 + O(q^3).
\end{equation}
Moreover, by (22.2.2) of \cite{DLMF} and our definition of the abstract angle,
\be\label{eq:bar:angles}
    \forall \bt \in \R,~\frac{\theta}{\bt} = \frac{2K}{\pi} = (1+2q + O(q^3))^2 = 1+2M\d+ M^2\d^2 + O(\d^3).
\ee
\lr{uniformly in $\bt$.} The asymptotic of the conductances can be computed using (22.2.9) of \cite{DLMF}. Observe that for $\bt \in \R$, their notation $\zeta(\t) = \pi \t/(2K)$ corresponds to our notation for the abstract angle: $\bt = \zeta(\t)$. Hence,
	\be\label{eq:approx:sc}
		\forall \bt \in \R,~\sc(\theta|k) = (1+4q+O(q^2))\frac{\sin(\bt)+O(q^2)}{\cos(\bt)+O(q^2)} = (1+2M\d+O(\d^2))\tan(\bt),
	\ee
\lr{uniformly in $\bt$: the constant in the $O$ depends only on $M$ and the constant $\eps$ in the bounded angle assumption (for the last equality).}

We now prove the asymptotic expression of the discrete massive exponential function. A proof of this result could probably be extracted from Section 5 of \cite{UniversalityCorrelations} among much more technical computations, but  for the sake of clarity we write down the computations in our (simpler) setting. 
Fix $\bu \in \R$, $x, y \in \oSVd$. Since the massive exponential is multiplicative, it can be computed along any path from $x$ to $y$. It is a classical consequence of the bounded angle assumption (see for example Assumption 1 of \cite{ShortTime}) that we can choose a path $\gamma_{\d}: x \to y$ of length $\leq \frac{C|y - x|}{\delta}$ where $|\cdot|$ is the \lr{E}uclidean distance and where the constant $C$ depends only on the lower bound in the bounded angle assumption. Denoting $ \gamma_{\delta} = (x_0, \dots, x_n)$,

\begin{equation}
	\me_{(x, y)}(u-2K-2iK'|k) = \prod_{j=0}^{n-1}\me_{(x_j, x_{j+1})}(u-2K-2iK'|k).
\end{equation}
We now compute the asymptotics of each term of the product. Let $x_j \sim x_{j+1} \in \SVd$, $x_{j+1}-x_j = \delta e^{i\ba_j} + \d e^{i\bar{\beta_j}}$ and let $\bt_j = (\bar{\beta}_j-\ba_j)/2$ be the half angle associated to the edge $x_jx_{j+1}$. Using the change of variable formula as in Equation~\eqref{ScDn},
\be\label{ExpNeighbours}
	\me_{(x_j,x_{j+1})}(u-2K-2iK'|k) = \frac{1}{k'}\dn\Big(\frac{u-\alpha_j}{2}\Big|k\Big)\dn\Big(\frac{u-\beta_j}{2}\Big|k\Big).
\ee
By definition of $\dn$ Equation (22.2.6) of \cite{DLMF}:
\be\label{dnua}
	\dn\Big(\frac{u-\alpha_j}{2}\Big) 
	= \sqrt{k'} \frac{1+2q\cos(\bu - \ba_j) + O(q^3)}{1-2q\cos(\bu-\ba_j) + O(q^3)}
	= \sqrt{k'} \exp(4q \cos(\bu - \ba_j)) + O(q^3)
\ee
because it is easily checked that $x \to \frac{1+x}{1-x}$ and $x \to \exp(2x)$ coincide up to $O(x^3)$. The $O$ is uniform \lr{in $u, \alpha_j$} (the constant depends only on the mass parameter $M \in \R_{>0}$). Using the fact that $\bu \in \R$ and the scaling $q = \frac{1}{2}M\d$ yields 
\begin{equation}
	\frac{1}{\sqrt{k'}}\dn\Big(\frac{u-\alpha_j}{2}\Big) = \exp(2M\d \langle e^{i\bu}, e^{i\ba_j}\rangle) + O(\d^3).
\end{equation}
Since we can write the same thing with $\beta_j$ instead of $\alpha_j$, Equation~\eqref{ExpNeighbours} gives 
\begin{equation}
    \begin{aligned}
	\me_{(x_j,x_{j+1})}(u-2K-2iK'|k) 
        &= \exp(2M\d\langle e^{i\bu}, e^{i\ba_j} + e^{i\bar{\beta}_j}\rangle) + O(\d^3)\\
        &= \exp(2M\langle e^{i\bu}, x_{j+1}-x_j\rangle) + O(\d^3).
    \end{aligned}
\end{equation}
\lr{uniformly in $u, \alpha_j$.} By multiplying along the path $\gamma_{\d}$ of length $\leq \frac{C|y - x|}{\delta}$, we get the desired result.
\end{proof}

As an application of Lemma~\ref{lem:asymp:near-critical}, we can compute the asymptotics of the weights of the RST and drifted dimer models in the near-critical regime. Fix $\bu \in \R$. Let $x \sim y \in \SVd$. Using the asymptotics of $\sc$ and $\me$ of Lemma~\ref{lem:asymp:near-critical}, we obtain that the conductance $\cl$ attributed to the edge $(x,y)$ by the RST model on $\SGd^o$ associated to the Laplacian $\Dld$ is asymptotically
\be\label{eq:cl:explicit}
	\begin{aligned}
	\cl_{(x,y)} = \sc(\theta_{xy}|k) \frac{\lambda^u(y)}{\lambda^u(x)}
	&= (1+2M\d+O(\d^2))\tan(\bt_{xy})\left(\exp\left(2M\d\langle e^{i\bu},e^{i\bar{\alpha}_{xy}}+e^{i\bar{\beta}_{xy}} \rangle\right) + O(\d^3)\right)\\
	&= (1+2M\d))\tan(\bt_{xy})\left(1 + 4M\d \cos(\bt_{xy}) \cos\left(\bar{u}-\frac{\bar{\alpha}_{xy} + \bar{\beta}_{xy}}{2}\right) \right) + O(\d^2)
	\end{aligned}
\ee 
\lr{uniformly for $x \sim y \in \SV$.}

\begin{Ex}\label{ex:square}
    If the isoradial grid $\oSGd$ is the square lattice $\d \Z^2$, we check that our model of drifted dimers coincides with the near-critical dimer model of \cite{Chhita} and \lr{\cite{Berestycki} w}ith drift parameter $\bu$: any $x \in \SGd$ has four neighbours $x \pm \d$ and $x \pm i \d$, with respective edge conductances $1 \pm 2M\d\sqrt{2}\cos(\bu)$ and $1 \pm 2M\d\sqrt{2}\sin(\bu)$ (up to an error $O(\d^2)$ and a global factor $(1+2M\d)$ which does not change the transition probabilities). This is exactly Equation (1.6) of \cite{Berestycki} with drift vector $\sqrt{2}Me^{i\bu}$.
\end{Ex}

\subsection{From the killed random walk to the drifted random walk}
In this section, we follow closely \cite{Berestycki}: we provide discrete Girsanov identities between the \lr{killed random walk} $\Skd$ and its Doob transform $\tSd$, and between their loop-erasures.\par
Recall that the Doob transform $\tSd$ and associated Laplacian $\Dld$ depend on the drift parameter $\bu$ via $\lambda^u$. We adapt to our setting Corollary 2.5 of \cite{Berestycki}. Denote by $\td = \tau_{\d}(\rho) \wedge \tau_{\d}(\oSVd \setminus \SVd)$ and $\tld = \tld(\oSVd \setminus \SVd)$ the escape times of $\SVd$ for $\Skd$ and $\tSd$ respectively.

\begin{Lem}\label{lem:near-critical:doob}
Let $\gamma_{\d}$ be a sequence of simple paths in $\SGd$ from $\xd \in \SVd$ to $\yd \in \partial\SVd$, with $\xd \overset{\d \to 0}{\longrightarrow} x \in \Omega$ and $y_{\delta} \overset{\d \to 0}{\longrightarrow} y \in \lr{\partial}\Omega$. Then:
\begin{equation}
\frac{\Proba_{\xd}\big(\forall 0 \leq i < \tld,~(\tSd)_i = \gamma_{\delta}(i)\big)}{\Proba_{\xd}\big(\forall 0 \leq i < \td,~(\Skd)_i = \gamma_{\delta}(i)\big)} \overset{\d \to 0}{\longrightarrow} \exp\left(2M\langle e^{i\bu}, y - x \rangle\right)
\end{equation}
uniformly in the isoradial grid, the drift parameter $\bu$, and the path.
\end{Lem}
Observe that this is precisely the same statement as Corollary 2.5 of \cite{Berestycki}, justifying the name “drift parameter” for $\bu$: $\tSd$ is a \lr{random walk} with a drift in the direction $e^{i\bu}$ of intensity $M$.

\begin{proof}
With the notation of the lemma, by definition of the Doob transform
\begin{equation}
        \frac{\Proba_{\xd}\big(\forall 0 \leq i < \tl_{\d}(\rho),~(\tSd)_i = \gamma_{\delta}(i)\big)}{\Proba_{\xd}\big(\forall 0 \leq i < \tau_{\d}(\rho),~(\Skd)_i = \gamma_{\delta}(i)\big)} = \me_{(\xd, \yd)}(u-2K-2iK'|k)\frac{\Proba_{\yd}\big((\tSd)_1 \in \oSVd \setminus \SVd\big)}{\Proba_{\yd}\big((\Skd)_1 \in \{\rho\} \cup \oSVd \setminus \SVd\big)}.
\end{equation} 
On the one hand, by the asymptotic expression of $\me$ computed in Lemma~\ref{lem:asymp:near-critical},
\begin{equation}
    \me_{(\xd, \yd)}(u-2K-2iK'|k) = \exp\left(2M\langle e^{i\bu}, \yd - \xd \rangle\right) + O(\delta^2|\yd - \xd|).
\end{equation}
On the other hand, by definition of the Doob transform, and since $m^2(x|k) = O(\d^2)$ by Lemma~\ref{lem:mass} and for all $z \sim \yd$, $\me_{(\yd,z)}(u-2K-2iK'|k) = 1+O(\d)$ by Lemma~\ref{lem:asymp:near-critical},
\begin{equation}
    \begin{aligned}
    \frac{\Proba_{\yd}\big((\tSd)_1 \in \oSVd \setminus \SVd\big)}{\Proba_{\yd}\big((\Skd)_1 \in \{\rho\} \cup \oSVd \setminus \SVd\big)} &= \frac{\sum_{z \sim \yd, z \notin \SVd}\me_{(\yd,z)}(u-2K-2iK'|k)\sc(\t_{\yd z}|k)}{m^2(\yd |k)+\sum_{z \sim \yd, z \notin \SVd}\sc(\t_{\yd z}|k)}\\
    &= \frac{(1+O(\d))\sum_{z \sim \yd, z \notin \SVd}\sc(\t_{\yd z}|k)}{O(\d^2) + \sum_{z \sim \yd, z \notin \SVd}\sc(\t_{\yd z}|k)} = 1+O(\d).
    \end{aligned}
\end{equation}
It is also important that since $y \in \partial \SVd$, $\sum_{z \sim \yd, z \notin \SVd}\sc(\theta_{\yd z}|k)$ is bounded away from $0$ since there is at least one term in the sum (and by the bounded angle assumption). 
\end{proof}

Recall from Section~\ref{subsec:Wilson} that $LE(\gamma)$ denotes the loop-erasure of the path $\gamma$. For a transient \lr{killed random walk} $\Xk$ (it is the case of both $\Skd$ and $\tSd$ killed when exiting $\SVd$: they die almost surely since $\SGd$ is finite), $LE(\Xk)$ denotes the loop-erasure of the trajectory up to the killing time. Lemma~\ref{lem:near-critical:doob} implies, as in \cite{Berestycki}, \lr{the following lemma which is analogous to their Lemma 2.6}.

\begin{Lem}\label{lem:drift=mass}
Let $\gamma_{\delta}$ be a sequence of simple paths in $\SGd$ from $x_{\delta} \in \SVd$ to $y_{\delta} \in \partial\SVd$, with $x_{\delta} \overset{\d \to 0}{\longrightarrow} x \in \Omega$ and $y_{\delta} \overset{\d \to 0}{\longrightarrow} y \in \partial\Omega$. Then,
\begin{equation}
\frac{\Proba_{\xd}(LE(\tSd) = \gamma_{\d})}{\Proba_{\xd}(LE(\Skd) = \gamma_{\d})} \overset{\d \to 0}{\longrightarrow} 
\exp\left(2M\langle e^{i\bu}, y - x \rangle\right).
\end{equation}
uniformly in the isoradial grid, the drift parameter $\bu \in \R$ and the path.
\end{Lem}
\begin{proof}
It follows from Lemma~\ref{lem:near-critical:doob} by summing over all ways to get $\gamma_{\delta}$ as a loop-erasure, and observing that for each of these paths the ratio of the probability is exactly the right-hand side by the preceding lemma. 
\end{proof}

\subsection{Convergence of the loop-erased killed random walk to massive \texorpdfstring{$SLE_2$}{SLE2}}
The next step towards convergence of the near-critical RST and dimer models is to prove convergence of the massive \lr{loop-erased random walk} conditioned to survive. We will verify the hypothesis of a general result for convergence of the killed \lr{loop-erased random walk}, Theorem 4.2 of \cite{Berestycki}, which we now recall.\par
\paragraph{Notation for continuous processes.} We first introduce some definitions related to continuous processes. The \emph{\lr{B}rownian motion} is the diffusion process $B$ on $\R^2$ with unit covariance matrix and $0$ drift vector. Recall from Equation (4.4) of \cite{Berestycki} that the \emph{killed \lr{B}rownian motion} with killing rate $M(\cdot)$ (a continuous non-negative bounded function on $\Omega$) is the process $\Bk$ which behaves like $B$ until a stopping time $\tau$ given by $\Proba(\tau >t~|~\FF_t) = \exp\left(-\int_0^t M^2(B_s)ds\right)$, where $(\FF_t,~t \geq 0)$ is the canonical filtration of the \lr{Brownian motion}, after which it goes to an absorbing cemetery. 

Given a \lr{random walk trajectory $\Yd$} on $\Omega_{\d}$ embedded in $\C$, we also denote by $\Yd$ the affine interpolation (that is for all $n \in \N$, $t \in [n, n+1($, $(\Yd)_t = (n+1-t)(\Yd)_n + (t-n)(\Yd)_{n+1}$). 
\lr{
\begin{Def}\label{def:uniform:topology}
    We say that a sequence of discrete random walk trajectories converges in the \emph{uniform topology} if the affine interpolation converges uniformly on bounded time intervals. 
\end{Def}
When we discuss convergence of a sequence of random walks (to a continuous process), we will refer to the convergence in law with respect to this topology as the \emph{convergence in law with respect to the uniform topology}.
}

\paragraph{A general convergence result.} Let $\Omega_{\d}$ be a sequence of planar graphs embedded in the complex plane with conductance function $\oc$, and suppose that there is no accumulation point in the sense that for every compact set $K$, the number of vertices of $\Omega_{\d}$ inside $K$ is finite. Let $M(\cdot): \Omega \to [0,\infty)$ be a bounded nonnegative continuous function. Denote by $\Yd$ the \lr{random walk} with conductances $\oc$ and $\Ykd$ the massive random walk with conductances $\oc$ and mass function $\om = M^2(\cdot)\d^2$. Recall that we can couple $\Ykd$ and $\Yd$ such that $\Ykd$ is killed in position $x$ with probability proportional to $M^2(\cdot)\d^2$ and otherwise moves like $\Yd$.\par
Let $\RR$ be the horizontal rectangle $[0,3]\times[0,1]$ and $\RR'$ be the vertical rectangle $[0,1]\times[0,3]$. Let $\Bst = (\Bst)' = B((1/2,1/2), 1/4)$ be the starting ball and $\Bta = B((5/2,1/2), 1/4)$, resp. $(\Bta)' = B((1/2,5/2), 1/4)$, be the target ball. Let $\D = \RR \setminus \Bta$. For $r > 0$, $z \in \C$, let $\Rrz = r \RR + z$ , resp. $\Rrz'$, and define similarly $\Bstrz, \Btarz, \Drz$.\par
For all $\d >0$, let $\crossk$, resp. $\crossk'$, denote the event that the \lr{killed random walk} $\Ykd$ started at a point of the starting ball $\xd \in \Bstrz \cap \oSVd$ enters the target ball $\Btarz$ before leaving the rectangle $\Rrz$ or dying:
\begin{equation}
    \crossk = \Big(\td(\Btarz) < \td(\rho) \wedge \td(\oSVd \setminus \Rrz)\Big).
\end{equation}
The \lr{killed random walk} $\Ykd$ satisfies the \emph{uniform crossing estimate} if for all $R >0$, there \lr{exist} constant\lr{s} $\eta = \eta(R) >0$ \lr{and $\d_0(R) > 0$ depending only on $R$} such that uniformly in $0 <r<R$, $z \in \C$, \lr{$0<\delta<\d_0(R)$} and $\xd \in \Bstrz \cap \oSVd$,
\begin{equation}\label{eq:cross:estimate}
    \Proba_{\xd}(\crossk) \geq \eta,~\Proba_x(\crossk') \geq \eta.
\end{equation}
\lr{Before stating our theorem, we recall the definition of \emph{massive $SLE_2$} from \cite{Berestycki}. For more details on this definition, see their Section 1.4. and the references therein.
\begin{Def}\label{def:massive:SLE_2}
    Let $x \in \Omega$, $a \in \partial \Omega$ be fixed and $\varphi$ be a fixed conformal map from $\Omega$ to the disk $\disk$ sending $x$ to $0$. The \emph{massive radial Schramm-Loewner Evolution} with parameter $\kappa = 2$ ($SLE_2$) and mass $M(\cdot)$ is a random curve $\gamma$ in $\Omega$ from $x$ to $a$. The \emph{driving function} $\zeta_t$ is related to the curve $\gamma$ by $\zeta(t) = g_t(\gamma(t))$. The massive radial $SLE_2$ is defined by its associated \emph{Loewner flow}, which is defined by Loewner's equation
    \be
        \frac{dg_t(z)}{dt} = - \varphi'(z)g_t(z)\frac{g_t(z)+\zeta_t}{g_t(z)-\zeta_t} \quad ; \quad z \in \Omega_t
    \ee
    where $\Omega_t$ denotes the slit domain $\Omega \setminus \gamma([0,t])$, $g_t$ is the Loewner map from $\Omega_t$ to $\disk$. If we write the driving function in the form $\zeta_t = e^{i\xi_t}$, then $\xi_t$ solves the stochastic differential equation
    $$
        d\xi_t = \sqrt{2}dB_t+2\lambda_t dt \quad ; \quad \lambda_t = \left.\frac{\partial}{\partial g_t(x_t)}\log \frac{P^{(M)}_{\Omega_t}(x_t,z)}{P_{\Omega_t}(x_t,z)}\right|_{z=b}
    $$
    where $x_t = \gamma(t)$, $P_{\Omega_t}^{(M)}$ and $P_{\Omega_t}$ are the Poisson kernels for, respectively, the Brownian motion killed with constant rate $M(\cdot)$ and the regular Brownian motion, in $\Omega_t$
\end{Def} 
}
\lr{To state our theorem, we need to introduce a topology on the space of discrete or continuous trajectories of the loop-erased random walk.
\begin{Def}\label{def:Schramm:RandomWalk}
    Let $\cH(\Omega)$ denote the Hausdorff space of non empty compact subsets of $\Omega$, endowed with the Hausdorff distance. The \emph{Hausdorff topology} is the topology induced by this distance. 
\end{Def}
A loop-erased random walk from $\xd \in \SVd$ to $\ad \in \partial \SVd$ or a trajectory of massive $SLE_2$ from $x \in \Omega$ to $a \in \partial \Omega$ are both elements of $\cH(\Omega)$. We will refer to the convergence in law with respect to this topology as the convergence in law with respect to the Hausdorff topology.}
\begin{Thm}[Theorem 4.2 of \cite{Berestycki}]\label{thm:berestycki}
    Assume that $\Ykd$ satisfies the uniform crossing estimate and that $\Ykd$ starts at $\xd \overset{\d \to 0}{\longrightarrow} x \in \Omega$. Suppose that $\left((\Yd)_{\d^{-2}t}, t\geq 0\right)$ converges \lr{in law (with respect to the uniform topology of Definition \ref{def:uniform:topology})} towards the \lr{Brownian motion}, or equivalently that $((\Ykd)_{\d^{-2}t}, t\geq 0)$ converges towards the \lr{killed Brownian motion} $\Bk$ killed with rate $M(\cdot)$. Then, the loop erasure $LE(\lr{\Ykd})$ of $\Ykd$, conditioned to leave $\Omega$ (at the stopping time $\td$) before dying, and conditioned on $(\Ykd)_{\td} = \ad$ with $\ad \overset{\d \to 0}{\longrightarrow} a \in \partial \Omega$, converges \lr{in law (with respect to the Hausdorff topology of Definition \ref{def:Schramm:RandomWalk})} to \lr{the massive $SLE_2$ from $x$ to $a$ with mass parameter $M$ of Definition \eqref{def:massive:SLE_2}}.
\end{Thm}
To apply this theorem, we will need the following result which is proved in Appendix~\ref{subsec:cross}, using the weak Beurling estimate of \cite{DiscreteHolomorphy} on the critical Laplacian.
\begin{Prop}\label{prop:cross}
    The near-critical \lr{killed random walk} $\Skd$ satisfies the uniform crossing estimate Equation~\eqref{eq:cross:estimate}.
\end{Prop}
Using Theorem~\ref{thm:berestycki} and Proposition \ref{prop:cross}, we are able to prove the following.
\begin{Thm}\label{thm:cv:KilledRandomWalk:mSLE}
    If $\xd \overset{\d \to 0}{\longrightarrow} x \in \Omega$, the loop erasure $LE(\Skd)$ of the \lr{killed random walk} $\Skd$ started at $\xd \in \SVd$ conditioned to leave $\Omega$ (or equivalently $\SVd$, at the stopping time $\td$) before dying, and conditioned on $(\Skd)_{\td} = \ad \in \partial \SVd$ with $\ad \to a \in \partial \Omega$ converges \lr{in law (with respect to the Hausdorff topology of Definition \ref{def:Schramm:RandomWalk})} to \lr{the massive $SLE_2$ from $x$ to $a$ with mass parameter $\sqrt{2}M$ of Definition \eqref{def:massive:SLE_2}}.
\end{Thm}
Note that we cannot apply Theorem~\ref{thm:berestycki} directly to $\Skd$ which is not killed with probability $M(x)\d^2$ for a continuous function $M$ (see Lemma~\ref{lem:mass}), and moreover the \lr{random walk} $\Sd$ with conductances $\oc_{xy} = \sc(\theta_{xy}|k)$ converges towards the \lr{Brownian motion} only after a time re-parametrization. These two matters are related, and we will solve them by considering an associated lazy \lr{random walk} $\Sld$.
\begin{proof}
Denote by $\Sd$ the \lr{random walk} with conductances $\oc_{xy} = \sc(\theta_{xy}|k)$ and by $\zed(x)$ the law of its increments: conditionally on $(\Sd)_n = x$, it is the law of $(\Sd)_{n+1}-(\Sd)_n$. \lr{Recall the definition of $T_{\d}$ from Equation \eqref{eq:def:Tdelta}.} Using Equation~\eqref{eq:increments:critical}, the asymptotics of $\sc$ of Lemma~\ref{lem:asymp:near-critical} and the fact that two neighbours in $\oSGd$ are at distance $\leq 2\d$, we obtain that
\be\label{eq:increments}
    \E(\zed(x))= O(\d^3);~\E(\Re(\zed(x))^2) = \E(\Im(\zed(x))^2) = \d^2 T_{\d}(x) + O(\d^4);~\E\big(\Re(\zed(x))\Im(\zed(x))\big) = O(\d^4).
\ee
Hence, a certain lazy version of $\Sd$ converges towards the \lr{Brownian motion}. To make it formal, we define the \emph{laziness} of a vertex 
\be
    \forall x \in \oSVd,~l_{\d}(x) := \frac{T_{\d}(x)-T}{T}\sum_{y\sim x} \sc(\theta_{xy}|k).
\ee
Let $\Sld$ be the \emph{lazy \lr{random walk}}, or more precisely the \lr{random walk} on the graph $\oSGd$ with an additional loop edge $(x,x)$ with conductance $l_{\d}(x)$ at each vertex $x \in \oSVd$: in other words, it is the \lr{random walk} which stays in $x$ with probability proportional to $l_{\d}(x)$ and otherwise moves towards $y \sim x$ with probability proportional to $\oc_{xy} = \sc(\theta_{xy}|k)$. The time spent by $\Sld$ at $x$ follows a geometric law with parameter $\pi_{\d}(x) = l_{\d}(x)/\left(l_{\d}(x) + \sum_{y \sim x}\sc(\theta_{xy}|k)\right)$. If we denote by $\zeld$ the law of the increments of $\Sld$, we have for all $x \in \oSVd$ and all $\phi: \C \to \C$ with $\phi(0)=0$,
\be
    \E(\phi(\zeld(x))) 
    = (1-\pi_{\d}(x))\E(\phi(\zed(x))).
\ee
This applies in particular with $\phi(x) = x, \Re(x)^2, \Im(x)^2, \Re(x)\Im(x)$ so Equation~\eqref{eq:increments} gives
\be
    \E(\zeld(x))= O(\d^3);~\E(\Re(\zeld(x))^2) = \E(\Im(\zeld(x))^2) = T\d^2 + O(\d^4);~\E\big(\Re(\zeld(x))\Im(\zeld(x))\big) = O(\d^4).
\ee
This implies that $((\Sld)_{(T\d^2)^{-1}t}, t\geq 0)$ converges \lr{in law (with respect to the uniform topology)} towards \lr{the Brownian motion}, see Theorem 11.2.3 of \cite{StroockVaradhan} with $h = T\d^2$ (it can also be seen as a consequence of the functional CLT for martingales after centering).\par

We are ready to deal with the \lr{killed random walk} $\Skd$. Let $\Slkd$ be the \emph{lazy \lr{killed random walk}}, or more precisely the \lr{killed random walk} on the graph $\oSGd$ with additional loop edges $(x,x)$ with conductance $l_{\d}(x)$ at each vertex and mass function $\om(x)=m^2(x|k)$. Define stopping times $t_0 = 0$ and $t_{i+1} = \inf\{n \geq t_i, (\Slkd)_n \neq (\Slkd)_{t_i}\}$: they are the times at which the lazy \lr{killed random walk} changes position. The law of the trajectories of $\Slkd$ and $\Skd$ are the same, that is 
\be\label{eq:lazy:law}
    ((\Slkd)_{t_i},~i \geq 0) \overset{(law)}{=} ((\Skd)_i,~i \geq 0). 
\ee
\lr{On the one hand, since the event $\crossk$ depends only on the trajectory of the \lr{killed random walk}, this implies that $\Slkd$ satisfies the uniform crossing assumption by Proposition \ref{prop:cross}. On the other hand, t}his implies that the loop-erasures of $\Slkd$ and $\Skd$ \emph{have the same law}. Using the asymptotics of $\sc$ and $m^2(\cdot|k)$ from Lemmas \ref{lem:asymp:near-critical} and \ref{lem:mass}, we get that the lazy \lr{killed random walk} $\Slkd$ is killed with probability 
\be
    \begin{aligned}
        \Proba\left((\Slkd)_{i+1}= \rho~|~(\Slkd)_i = x\right) &= \frac{m^2(x|k)}{m^2(x|k) + l(x) + \sum_{y \sim x}\sc(\bt_{xy}|k)}\\
        &= \frac{Tm^2(x|k)}{Tm^2(x|k)+\sum_{y \sim x}\sin(2\bt_{xy})+O(\d)} = 2M^2T\d^2 + O(\d^2),
    \end{aligned}
\ee
so $((\Slkd)_{(T\d^2)^{-1}t}, t\geq 0)$ converges \lr{in law (with respect to the uniform topology)} towards the \lr{killed Brownian motion} with constant killing rate $\sqrt{2}M$.\par
We now apply Theorem~\ref{thm:berestycki} to the \lr{killed random walk} $\Ykd := S^{\mathrm{l,k}}_{T^{-1/2}\d}$ on the graph $\Omega_{\d}$ which is $\SG_{T^{-1/2}\d}$ with additional loop edges at each point (this graph is still planar). \lr{Since $\Ykd$ has the same trajectories as $\Slkd$, it also satisfies the uniform crossing assumption.} Moreover, we just obtained that $((\Ykd)_{\d^{-2}t},t\geq 0)$ converges towards the \lr{killed Brownian motion} with constant killing rate $\sqrt{2}M$. Hence, conditionally on the endpoint, $LE(\Ykd)$ converges towards massive $SLE_2$ with mass parameter $\sqrt{2}M$. In other words (replacing $\d$ by $T^{-1/2}\d$), $LE(\Slkd)$ converges towards \lr{the} massive $SLE_2$ with mass $\sqrt{2}M$. Since the loop-erasures of $\Slkd$ and $\Skd$ have the same laws, this concludes the proof of Theorem~\ref{thm:cv:KilledRandomWalk:mSLE}.
\end{proof}

\subsection{Convergence of the near-critical rooted spanning tree and drifted dimer models}
\lr{Recall from Definition \ref{def:Schramm:RandomWalk} that given a metric space $E$ we denote by $\cH(E)$ the Hausdorff space of non empty compact subsets of $E$, endowed with the Hausdorff distance. 
\begin{Def}\label{def:Schramm:tree}
    We call the \emph{Schramm topology} (introduced in \cite{Schramm}) the Hausdorff topology on the metric space $\Omega \times \Omega \times \cH(\Omega)$ endowed with the product distance.
\end{Def}}

\lr{Given a tree $T_\d \in \TT^{o}(\Go_\d)$ and $\xd, \yd \in \SVd$, let $\omega_{\xd,\yd}$ denote the unique path from $\xd$ to $\yd$ in the undirected tree obtained by forgetting all edge orientations in $T_\d$. Then, the set 
$$
    \cI_\d(T_\d) := \{(\xd,\yd,\omega_{\xd,\yd}), \xd, \yd \in \SVd\}
$$
is an element of the Hausdorff space $\cH(\Omega \times \Omega \times \cH(\Omega))$. Given a sequence of random trees $(T_\d)$, we will abuse notation and say that the sequence converges in law with respect to the Schramm topology if the associated sequence $\cI_\d(T_\d)$ converges in law with respect to this topology towards a random element of the space $\cH(\Omega \times \Omega \times \cH(\Omega))$.}

We denote by $\TT_{\d}^{M,\bu}$ the \lr{random} RST of $\SGd^o$ associated to the Laplacian $\Dld$ by Wilson's algorithm. The following theorem extends Theorem 1.2 of \cite{Berestycki} to the isoradial case. 

\lr{Let $\sigd$ denote the exit law from $\partial \Omega$ of a \emph{drifted Brownian motion} started at $x$ with drift vector $2Me^{i\bu}$, that is a diffusion process with unit covariance matrix and drift vector $\lr{2}Me^{i\bu}$}.
\begin{Thm}\label{thm:near-critical:tree}
    \lr{In the near-critical regime described above ($q = \frac{1}{2}M\d$), the tree $\TT_{\d}^{M,\bu}$ converges in law \lr{(with respect to the Schramm topology of Definition \ref{def:Schramm:tree})} to a continuum limit tree $\TT^{M,\bu}$ on $\Omega$\lr{, seen as a random element of $\cH(\Omega \times \Omega \times \cH(\Omega))$}. The law of a branch of this continuum tree can be described explicitly: given $x \in \Omega$, the endpoint $y \in \partial \Omega$ of the branch has law $\sigd$, and conditionally on $y$ the branch of the tree from $x$ to $y$ has the law of the massive radial $SLE_2$ in $\Omega$ with mass parameter $\sqrt{2}M$ from $x$ to $y$.}
\end{Thm}
Note that this theorem contains the square case of Theorem~1.2 of \cite{Berestycki} as a particular case by Example~\ref{ex:square}, but not the hexagonal case as they consider a non-reversible random walk.

\begin{proof}
\lr{This proof is strictly identical to the proof of Theorem 1.2 of \cite{Berestycki}}. Recall that $\TT_{\d}^{M,\bu}$ can be seen as a RST of $\SGd$ rooted outside $\SVd$. The law of a branch of this RST from a point $\xd \in \SVd$ to the boundary $\partial \SVd$ is the loop-erasure of $\tSd$ started at $\xd$. 

Let us denote by $\sigk(y)dy$ the exit measure on $\partial \Omega$ of the \lr{killed Brownian motion} with constant killing rate $\sqrt{2}M$ (note that this is a sub-probability measure). Lemma~\ref{lem:near-critical:doob} together with the convergence of the \lr{(lazy) killed random walk} towards the \lr{killed Brownian motion}, which we explained in the proof of Theorem \ref{thm:cv:KilledRandomWalk:mSLE}, imply that asymptotically, the law of the endpoint of a branch starting from $\xd \overset{\d \to 0}{\longrightarrow} x \in \Omega$ is
\be
    \exp(2M\langle e^{i\bu},y\rangle)\sigk(y)dy.
\ee
\lr{By the Cameron-Martin theorem, this is the exit law of $\Omega$ of a Brownian motion with drift vector $\sqrt{2}Me^{i\bu}$, that is $\sigd$.}  

By Lemma~\ref{lem:drift=mass}, conditionally on the endpoint $\ad \in \partial \SVd$, \lr{the} law \lr{of the loop-erasure of $\tSd$ from $\xd$ to $\ad$} is \lr{equal to} the \lr{law of the} loop-erasure of the \lr{killed random walk} $\Skd$ conditioned on escaping $\SVd$ via $\ad$ before dying. \lr{B}y Theorem~\ref{thm:cv:KilledRandomWalk:mSLE}, \lr{this law} converges \lr{weakly (with respect to the Hausdorff topology)} towards \lr{the law of the} massive $SLE_2$ with mass parameter $\sqrt{2}M$. \lr{Hence the loop-erasure of $\tSd$ converges in law (with respect to the Hausdorff topology) to the massive $SLE_2$ with mass parameter $\sqrt{2}M$, with exit distribution given by $\sigd$.}

\lr{
The convergence of the loop-erased random walk, applied iteratively using Wilson's algorithm, implies the convergence in law (with respect to the Schramm topology) of the tree $\TT_{\d}^{M,\bu}$ towards a continuum tree $\TT^{M,\bu}$.
}
\end{proof}

By Temperley's bijection (see our Section~\ref{sec:planar}), the tree model is in weight-preserving bijection with the drifted dimer model on the double isoradial graph $\Gdr_{\d}$. \lr{By \cite{TreesMatchings}, to each dimer configuration of $\Gdr_{\d}$ is associated a $\R$-valued height function on the faces of $\Gdr_{\d}$, uniquely defined up to an additive constant (see their Lemma 2 for a precise definition). Let $f_\d$ be some arbitrary face of $\Gdr_{\d}$ and let $h^{M,\bu}_{\d}$ be the height function of a random drifted dimer configuration with mass and drift parameters $\sqrt{2}M, \bu$, with the additive constant chosen to be zero at $f_{\d}$. The choice of $f_\d$ does not matter to us since we will only consider centered height functions.} 

As in \cite{Berestycki}, convergence \lr{in law} of the \lr{tree} implies convergence \lr{in law} of the height function \lr{associated by the Temperley bijection.} 
\begin{Thm}\label{thm:near-critical:dimers}
    In the near-critical regime described above ($q = \frac{1}{2}M\d$), the centered height function $h^{M,\bu}_{\d}- \E[h_{\d}^{M,\bu}]$ of the drifted dimer model with mass and drift parameters $\sqrt{2}M>0, \bu \in \R$ converges \lr{in law in the sense of distributions} to a limit $h^{M,\bu}$ when $\d \to 0$.
\end{Thm}
\begin{proof}
\lr{We refer to the proof of Proposition 5.2 of \cite{Berestycki} for more details. Recall that the tree $\TT^{M,\bu}_\d$ is obtained from the drifted random walk $\tSd$ by Wilson's algorithm. The idea of the proof is to apply Theorem 6.1 of \cite{BLR19}. This theorem states in a very general setting (the dimer model on a sequence of graphs approximating a Riemann surface $M$) that convergence in law (with respect to the Schramm topology) of the tree associated via the Temperley bijection implies convergence of the height function in law in the sense of distributions. For this theorem to apply, a few hypothesis must hold, which are described in Section 2.4 of \cite{BLR19}, and which we now recall:
\begin{itemize}
    \item their hypothesis (i), (ii) and (v) require that $\SGd$ is “correctly” embedded in a Riemann surface $M$: in our case the Riemann surface is simply $M = \Omega$ and these hypothesis are trivially satisfied.
    \item their hypothesis (iv) is the uniform crossing estimate of Equation \eqref{eq:cross:estimate} for the drifted random walk $\tSd$. We know from Proposition \ref{prop:cross} that the killed random walk $\Skd$ satisfies the uniform crossing estimate. By Lemma \ref{lem:near-critical:doob}, the Radon-Nikodym derivative of the law of $\tSd$ with respect to the law of $\Skd$ is bounded from above and below by constants independent of $\delta$, hence $\tSd$ also satisfies the uniform crossing estimate. 
\end{itemize}
There is an additional hypothesis (iii), requiring that the random walk $\tSd$ converges to the Brownian motion. This hypothesis is only used later in \cite{BLR19} to identify the limit in law of the loop-erased random walk, but it is not used in the proof of their Theorem 6.1 (as observed in the proof of Proposition 5.2 in \cite{Berestycki}). Hence Theorem 6.1 of \cite{BLR19} applies in our setting. As a consequence, the centered height function converges in law in the sense of distribution, that is for all $f: \Omega \to \R$ smooth and compactly supported with $\int_{\Omega} f = 0$, then
\be\label{eq:cv:distribution}
    \sum f(\xd)(h^{M,\bu}_{\d}(\xd)- \E[h_{\d}^{M,\bu}])
\ee
converges in law, where the sum is taken over faces of $\Gdr_{\d}$. Actually, Theorem 6.1 of \cite{Berestycki} states the joint convergence in law of Equation \eqref{eq:cv:distribution} and of $T_{\d}^{M,\bu}$, which is stronger.
}
\end{proof}
The limit is not explicit with these techniques, but \cite{SineGordon} seems to indicate that it is related to the Sine-Gordon field, see (i) of the open question Section 1.7. of \cite{Berestycki} for more details on this conjecture.

\lr{
\begin{Rem}\label{rem:conformal:covariance}
In Theorems \ref{thm:cv:KilledRandomWalk:mSLE}, \ref{thm:near-critical:tree}, \ref{thm:near-critical:dimers}, we expect the limit to display \emph{conformal covariance}. When we apply a conformal map $\phi: \Omega \to \Omega'$, the limit should change and be associated to a massive $SLE_2$ with \emph{variable mass} 
$$
    \forall z \in \Omega, M'(\phi(z)) = |\phi'(z)|M.
$$
This is discussed in details in \cite{Berestycki}, in particular in their Theorem 1.4. We did not define a dimer and tree model with variable mass, but we can still illustrate the \emph{conformal scaling} of the mass. Let $r > 0$ and define $\phi: \C \to \C; \phi(z) = rz$. Let $\Omega' := \phi(\Omega)$. The conformal map $\phi$ maps $\SGd$ with an isoradial approximation $\SGd' := \phi(\SGd)$ of $\Omega'$ with mesh $r\delta$ on which a massive Laplacian operator $(\Dkd)'$ is defined by 
$$
    \forall f: \SGd' \to \C, (\Dkd)' f := \Dkd (f\circ\phi^{-1}).  
$$
Letting $\delta' = r \delta$, Theorems \ref{thm:cv:KilledRandomWalk:mSLE}, \ref{thm:near-critical:tree}, \ref{thm:near-critical:dimers} apply and prove that the  loop erasure of the \lr{killed random walk} conditioned to leave $\Omega'$, the near-critical tree and the height function of the near-critical dimer model converge towards the corresponding continuum objects on $\Omega'$. Since $q = \frac{1}{2}M\d = \frac{1}{2}\left(\frac{M}{r}\right)\d'$, the corresponding mass parameter is $M' = M/r$, as expected by conformal covariance.
\end{Rem}
}


\appendix
\section{Critical and near-critical isoradial Laplacians}\label{app:critical:laplacian}
In this appendix, we use some tools developed in \cite{DiscreteHolomorphy} for the analysis of the critical Laplacian on isoradial graphs and adapt them to the near-critical regime to prove Lemma~\ref{lem:mass} and Proposition \ref{prop:cross}. We use the notation of Section~\ref{sec:near-critical} for isoradial graphs and discrete $Z$-invariant and critical Laplacians.

\subsection{Proof of the asymptotic of the mass in the near-critical regime}\label{subsec:asymp:mass}
\paragraph{The approximation property.} Many of the properties of the critical Laplacian $\Dd$ of Equation~\eqref{eq:def:non-massive:D} follow from the fundamental \emph{approximation property} Lemma 2.2 of \cite{DiscreteHolomorphy} which we now recall. Let $\Delta = \partial_{xx} + \partial_{yy}$ and $D$ denote respectively the usual continuous Laplacian and derivative operators on $\R^2$. Note that the definition of and $\Dd$ \lr{(and $\Dkd$)} are \emph{local}: for $x \in \oSVd$, $(\Dd f)(x)$ (and also $(\Dkd f)(x)$) is defined as long as $f$ is defined on the neighbours of $x$. In particular, it makes sense if $f$ is defined only on $B(x,2\d)$ (the \lr{E}uclidean ball centered at $x$ of radius $r$) which contains all the neighbours of $x$. 
\begin{Lem}[Lemma 2.2 of \cite{DiscreteHolomorphy}]\label{LemApproxMassless}
    There exists an absolute constant $C$ such that for all $x \in \oSVd$, for all smooth function $f \in \CC^3(B(x,2\d), \R)$, 
	\begin{equation}
		\left| (\oDd f)(x) + \frac{\d^2}{2}\left(\sum_{y \sim x}\sin(2\bt_{xy}\right)(\Delta f)(x) \right| \leq C \d^3 \sup_{B(x,2\d)} |D^3f|.
	\end{equation}
\end{Lem}
From this approximation property follows the \emph{approximation property for the \lr{near-critical massive Laplacian}}. \lr{Recall that we work in the near-critical regime $q = \frac{1}{2}M\d$ of Section \ref{subsec:near-critical}.}
\begin{Lem}\label{lem:approx:M:constant}
    There exists an absolute constant $C$ such that for all $x \in \oSVd$, for all smooth function $f \in \CC^3(B(x,2\d), \R)$ defined in a neighborhood of $x$,
	\begin{equation}
		\left| (\oDkd f)(x) + \frac{\d^2}{2}\left(\sum_{y \sim x}\sin(2\bt_{xy}\right)(\Delta f)(x) - m^2(x|k)f(x) \right| \leq C \d^3 \sup_{B(x,2\d)}(|Df|+|\Delta f|+|D^3f|)
	\end{equation}
\end{Lem}
\begin{proof}
     Recall that in the near-critical regime, by Equation~\eqref{eq:approx:sc}, the conductances of the massive Laplacian $\oDkd$ are $\oc_{xy} = \sc(\theta_{xy}|k)= (1+2M\d+O(\d^2))\tan(\bt_{xy})$. Moreover, for all $x \in \oSVd$, $f \in \CC^3(B(x,2\d), \R)$, $y \sim x$, we have $|f(x)-f(y)| \leq \d \sup_{B(x,2\d)}|Df|$ so
    \begin{equation}
        \begin{aligned}
        &\left|(\oDkd f)(x) - m^2(x|k)f(x) - (1+2M\d)(\oDd f)(x)\right|\\
        &\qquad =  \left|\sum_{y \sim x} \left(\sc(\t_{xy}|k)-(1+2M\d)\tan(\bt_{xy})\right)(f(x)-f(y))\right| \leq C\d^3\sup_{B(x,2\d)}|Df|.
        \end{aligned}
    \end{equation}
    For some universal constant $C$. The approximation property for the massive Laplacian follows by triangular inequality and Lemma~\ref{LemApproxMassless}.
\end{proof}

\paragraph{The asymptotics of the near-critical mass.} We use the approximation property and the massive harmonicity of $\me$ to prove Lemma~\ref{lem:mass}. 

\begin{Rem}
This proof is a shortcut: it avoids all direct computations which are hidden in the massive harmonicity of the discrete exponential. A direct computation is also possible but lengthy.
\end{Rem}

\begin{proof}
    Let $x \in \oSVd$ be fixed. \lr{Fix an arbitrary $\bu \in \R$.} We introduce a function $e^{M,\bu}: y \in \C \to \exp(2M\langle e^{i\bu},y-x\rangle)$. On the one hand, by Lemma~\ref{lem:asymp:near-critical}, for $y \sim x$,
    \begin{equation} 
	\me_{(x,y)}(u-2K-2iK'|k) = \exp(2M\langle e^{i\bu},y-x\rangle) + O(\d^3) = e^{M,\bu}(y) + O(\d^3)
    \end{equation}
    so by linearity of $\oDkd$ and massive harmonicity of $\me$, 
    \begin{equation}
        (\oDkd e^{M,\bu})(x) = O(\d^3).
    \end{equation}
    On the other hand, $e^{M,\bu}$ is (continuous) \emph{massive harmonic with mass $4M^2$}: $\Delta e^{M,\bu} = 4M^2 e^{M,\bu}$. By the approximation property applied at $x$ to the function $e^{M,\bu}$ (and since by definition, $e^{M,\bu}(x) = 1$): 
    $$
        m^2(x|k) = 2M^2\d^2 \sum_{y\sim x} \sin(2\bt_{xy}) +O(\d^3).
        \qedhere
    $$
\end{proof}

\subsection{Proof of the uniform crossing estimate}\label{subsec:cross}
\paragraph{The weak Beurling estimate for the critical Laplacian.}
We recall the notation and result of \cite{DiscreteHolomorphy}. For a complex number $z \in \C \setminus \{0\}$, denote by $\arg(z)$ its argument in $(-\pi, \pi]$. For an open \lr{E}uclidean ball $B \subset \C$ \lr{of radius $r > 0$} and $\d >0$, denote by $B_{\d}$ the \emph{discretization} of $B$: it is the subgraph induced by $\oSVd \cap B$. Note that for $r/\d \geq C_0$ some absolute constant depending only on the bounding angle assumption, it is connected and coincides with the definition of \cite{DiscreteHolomorphy}, and moreover $\partial B_{\d}$ is at distance $\leq 2\d$ of $\partial B$: this might not hold for more complicated sets. For $a,b \in \partial B$, resp. $\ad, \bd \in \partial B_{\d}$, we denote by $(a,b) \subset \partial B$, resp. $(\ad,\bd) \subset \partial B_{\d}$, the points of $\partial B$, resp. vertices of $\partial B_{\d}$, encountered when following the boundary counterclockwise. Denote by $\Xd$ the \lr{random walk} with conductances $\tan(\bt_{xy})$, that is the \lr{random walk} associated with the critical Laplacian $\Dd$.
\begin{Lem}[Lemma 2.10 of \cite{DiscreteHolomorphy}]\label{lem:beurling}
    For $o_{\d} \in \oSVd$, $r > C_0 \d$, let $B(o_{\d},2r)$ be the \lr{E}uclidean ball of radius $2r$, $B_{\d}(o_{\d},2r)$ its discretization and $\td$ the escape time of $B_{\d}(o_{\d},2r)$ by $\Xd$. Let $\ad, \bd \in \partial B_{\d}(o_{\d},2r)$ such that $\arg(\bd-o_{\d}) - \arg(\ad - o_{\d}) \geq \pi/4$. Then, 
    \be
        \forall \xd \in B_{\d}(o_{\d}, r),~\Proba_{\xd}((\Xd)_{\td} \in (\ad, \bd)) \geq \eta_0 \lr{>} 0,
    \ee
    where $\eta_0$ is an absolute constant (independent of the isoradial grid, $r$, $o_{\d}, \ad, \bd$).
\end{Lem}

\paragraph{The uniform crossing estimate for the near-critical Laplacian.}
We prove Proposition \ref{prop:cross}. We only do the proof for $\crossk$: the proof for $\crossk'$ is exactly the same.
\begin{proof}
Fix $R >0$, and let $C_0 \d \leq r \leq R$ \lr{where $C_0$ is the constant defined before the statement of Lemma \ref{lem:beurling}, depending only on the constant $\eps$ from the bounded angle assumption}. The remaining case $r < C_0 \d$ will be considered at the end of the proof.\\
\emph{First step: the uniform crossing estimate for the \lr{random walk} with no killing.} For all $0 \leq i \leq 4$, define $B^i = B((1/2 + i/2,1/2),1/4), A^i = B((1/2 + i/2,1/2),1/2)$ and note that $B^0 = \Bst$, $B^4 = \Bta$. Let $0 \leq i \leq 4$ and define $\Birz, A^i(r,z)$ for the scaled and translated balls (as in the statement of Proposition \ref{prop:cross}) and $\Birzd, A^i_{\d}(r,z)$ for their discretizations. Denote by $o^i(r,z) = r(1/2 + i,1/2)+z$ the common center of the $i$-th small and big balls. Let $a^i, b^i \in \partial A^i(r,z)$ such that
\begin{equation}
    \arg(a^i - o^i(r,z)) = -\frac{\pi}{7} \quad ; \quad \arg(b^i - o^i(r,z)) = \frac{\pi}{7}.
\end{equation}
where $\arg$ denotes the argument in $(-\pi, \pi]$. \lr{This is illustrated on Figure \ref{fig:crossing}.}

\begin{figure}[h!]
    \centering   
        \begin{overpic}[abs,unit=1mm,scale=4
        ]{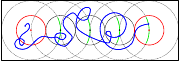}
        \end{overpic}
        \caption{\lr{An illustration of the proof of the uniform crossing estimate. The starting and target balls $\Bst, \Bta$ are in red. The intermediate balls $B_i$ are in black. The rectangle $\Rrz$ is in black. For each $i$, the arc of $\partial A^i$ between $a^i$ and $b^i$ is indicated in green. A path of the random walk $\Xd$ is in blue. The blue dots indicate the first escape of $A^i$. Note that the green arc of $A^i$ is included in the ball $B^{i+1}$, so every blue dot belongs to the next $B^{i+1}$. By the Beurling estimate, the probability that each blue dot belongs to the next green arc is bounded from below by an absolute constant. When each blue dot belongs to the next green arc, the random walk enters the target ball before leaving $\Rrz$.} }
        \label{fig:crossing}
\end{figure}

Let $\ad^i, \bd^i$ be the closest points in $\partial A^i_{\d}(r,z)$ from $a^i$ and $b^i$ respectively: they are at distance $\leq 2\d$ from $a_i$ and $b_i$, so upon increasing $C_0$ \lr{(which still depends only on the constant $\eps$ from the bounded angle assumption)}, $\arg(\bd-o^i(r,z)) - \arg(\ad - o^i(r,z)) \geq \pi/4$ for all $r \geq C_0\d$.

The weak Beurling estimate Lemma~\ref{lem:beurling} implies that 
\be\label{eq:beurling}
    \forall \xd \in \Birzd,~\Proba_{\xd}((\Sd)_{\td^i} \in (\ad^i, \bd^i)) \geq \eta_0.
\ee
Now, we observe that for all $0 \leq i < 4$, for all $y \in (a^i, b^i) \subset \partial A^i(r,z)$, denoting $\a = \arg(y-o^i(r,z)) \in [-\pi/7, \pi/7]$,
\begin{equation}
    |y - o^{i+1}(r,z)| = \left|\frac{r}{2} \sin(\a)\right| \leq \frac{r}{2} \sin(\pi/7).
\end{equation}
Since $\sin(\pi/7)/2 < 1$, $(a_i, b_i) \subset \partial\lr{A}^i(r,z)$ is a strict subset of $B^{i+1}(r,z)$. Since $(\ad^i, \bd^i)$ is within \lr{$2\d$} of $(a^i, b^i)$, upon increasing $C_{\lr{0}}$, we obtain that for all $r \geq C_{\lr{0}}\d$, for all $0 \leq i \leq 3$, $(\ad^i, \bd^i) \subset B^{i+1}_{\d}(r,z)$. We deal with the remaining case $r \leq C_{\lr{0}}\d$ at the end of the proof.\\
Let $\td^i$ denote the first escape time of $A^i_{\d}(r,z)$ by $\Xd$. Since we also have $A^i(r,z) \subset \Rrz$, on the event $(\Xd)_{\td^i} \in (\ad^i, \bd^i)$ the \lr{random walk} enters $B^{i+1}_{\d}(r,z)$ before leaving $\Rrz$, \lr{see Figure \ref{fig:crossing}. Hence,} the crossing event $\cross := \Big(\td(\Btarzd) < \td(\oSVd \setminus \Rrz)\Big)$ for the \lr{random walk} $\Xd$ satisfies 
\begin{equation}
    \bigcap_{i=0}^3 \Big((\Xd)_{\td^i} \in (\ad^i, \bd^i)\Big) \subset \cross.
\end{equation}
Applying four times the strong Markov property (with the stopping times $\td^i$ for $0 \leq i \leq 3$) and using Equation~\eqref{eq:beurling} gives the uniform crossing estimate for the critical \lr{random walk} $\Xd$:
\be\label{eq:crossing:no-killing}
    \forall \xd \in \Bstrzd,~\Proba_{\xd}(\cross) \geq \eta_0^4.
\ee
\par
\emph{Second step: upper bound on the escape time for the \lr{random walk} with no killing.} Let $z \in \oSVd$ and $\xd \in \Bstrzd$ a starting point. Recall that $(\xid)_n = (\Xd)_{n+1}-(\Xd)_n$ denotes the increments of the \lr{random walk}. The law $\xid(\yd)$ of $(\xid)_n$ depends only on the value $\yd = (\Xd)_n \in \oSVd$. Denote by $M_{\d} = \Re(\Xd)$. From Equation~\eqref{eq:increments:critical}, we see that $M_{\d}$ is a martingale in the canonical filtration $(\FF_n)_{n \in \N}$ of $\Xd$ with increments satisfying for all $\yd \in \oSVd$, $\E(\Re(\xid(\yd))) = 0$ and $\E(\Re(\xid(\yd))^2) \geq \eta_1\d^2$ for some absolute constant $\eta_1$. Hence the quadratic variation $\langle M_{\d}\rangle$ of $M_{\d}$ satisfies for all $n \geq 0$, $\langle M_{\d}\rangle_n \geq \eta_1 n \d^2$. Let $\sid = \inf \{n \geq 0, |(M_{\d})_n-\Re(\xd)| \geq 3R \}$. By the monotone convergence theorem and the optional stopping theorem applied to the bounded martingales $((M_{\d})_{\cdot \wedge n} - \Re(\xd))^2-\langle M_{\d} \rangle_{\cdot \wedge n}$ at stopping time $\sid$, 
\begin{equation}
    \eta_1 \d^2 \E(\sid) \leq \E(\langle M_{\d}\rangle_{\sid}) = \lim_{n \to \infty} \E(\langle M_{\d}\rangle_{\sid \wedge n}) = \lim_{n \to \infty} \E((M_{\d})_{\sid \wedge n} - \Re(\xd)^2) \leq (3R)^2.
\end{equation}
By Markov's inequality and since $\td(\oSVd \setminus \Rrz) \leq \sid$, this implies for the critical \lr{random walk} $\Xd$:
\be\label{eq:bounded:time}
    \forall \lambda >0,~\Proba_{\xd}\left(\td(\oSVd \setminus \Rrz) \geq \lambda \frac{R^2}{\d^2}\right) \leq \Proba\left(\sid \geq \lambda \frac{R^2}{\d^2}\right) \leq \frac{9}{\eta_1\lambda}.
\ee
\par

\emph{Third step: coupling the \lr{killed random walk} and the \lr{random walk}.} From Lemma~\ref{lem:mass} and Equation~\eqref{eq:approx:sc} we know that the mass satisfies $m^2(\xd|k) = O(\d^2)$ and that the conductances $\sc(\theta_{xy}|k)$ of the near critical Laplacian $\Dkd$ are within $O(\d^2)$ of the conductances $\tan(\bt_{xy})$ of $\Dd$ (after multiplying all conductances by a common factor $1+2M\d + O(\d^2)$ which does not change the probabilities). Hence the \lr{killed random walk} $\Skd$ and the \lr{random walk} $\Xd$ can be coupled with a probability of failing $\leq \eta_2\d^2$ at each step \lr{where $\eta_2$ is a constant depending only on the bounded angle assumption and $M$.}

\emph{Fourth step: concluding the proof when $C_0\d \leq r \leq R$.} We choose $\lambda$ such that $9/(\eta_1\lambda) = \eta_0^4/2$: Equations \eqref{eq:crossing:no-killing} and \eqref{eq:bounded:time} imply that the crossing event $\cross$ for the critical \lr{random walk} $\Xd$ happens with positive probability at a time of order $\lambda R^2/\d^2$ or less:
\begin{equation}
    \forall \xd \in \Bstrzd,~ \Proba_{\xd}\Big(\td(\Btarzd) < \td(\oSVd \setminus \Rrz) < \lambda \frac{R^2}{\d^2}\Big) \geq \frac{\eta_0^4}{2}.
\end{equation}
\lr{Recall that $R > 0$ is fixed, and let $\d \leq \eta_2^{-1/2}$.} On the event $\Big(\td(\Btarzd) < \td(\oSVd \setminus \Rrz) < \lambda R^2/\d^2\Big)$, the coupling fails with probability at most $1-(1-\eta_2\d^2)^{\lambda R^2/\d^2} \leq 1-\eta_3 <1$ for some \lr{$\eta_3 = \eta_3(R) >0$ and all $\d \leq \d_0(R)$ where the constants $\eta_3(R), \d_0(R)$ depend only on $R$, $M$ and the bounded angle assumption}. Since on the event $\cross$, when the coupling succeeds until $\td(\oSVd \setminus \Rrz$, the event $\crossk$ happens, we conclude that 
\be
    \forall \xd \in \Bstrzd,~\Proba_{\xd}(\crossk) \geq (1-\eta_3)\eta_0^4/2.
\ee
\emph{Fifth step: when $r \leq C_0\d$.}
In this case, the rectangle $\Rrz$ has a bounded diameter for the graph distance: by the bounded angle assumption, we can always find a path with at most $C$ vertices from any point in the starting ball to the target ball. The probability that the walk follows this path and survives is uniformly bounded away from $0$ by a constant of the form $\eta^{C}$ for some absolute constant $\eta >0$.\qedhere
\end{proof}

\end{document}